\documentclass[final,leqno,onefignum,onetabnum]{siamltex}
\usepackage{color}
\usepackage{amsmath,amsfonts,amssymb,bm,mathtools} %
\usepackage{graphicx,psfrag,subfigure}                    %
\usepackage{stmaryrd}
\usepackage{enumitem}
\usepackage{epstopdf}
\usepackage{mathrsfs}                           %
\usepackage{multirow}

\allowdisplaybreaks

\renewcommand{\d}{\mathrm{d}}
\newcommand{\e}{\mathrm{e}}
\newcommand{\R}{\mathbb{R}}
\newcommand{\C}{\mathbb{C}}
\newcommand{\N}{\mathbb{N}}

\newcommand{\bg}{\bm{g}}
\newcommand{\bi}{{\bm{i}}}
\newcommand{\bj}{{\bm{j}}}

\newcommand{\bn}{\bm{n}}

\newcommand{\bs}{\bm{s}}
\newcommand{\bu}{\bm{u}}
\newcommand{\bv}{\bm{v}}
\newcommand{\bw}{\bm{w}}
\newcommand{\bx}{\bm{x}}
\newcommand{\by}{\bm{y}}

\newcommand{\bxi}{\bm{\xi}}

\newcommand{\hA}{\mathcal{A}}

\newcommand{\hI}{\mathcal{I}}
\newcommand{\hL}{{\mathcal{L}}}
\newcommand{\hN}{\mathcal{N}}
\newcommand{\hT}{\mathcal{T}}
\newcommand{\hX}{\mathcal{X}}
\newcommand{\whX}{\widehat{X}}
\newcommand{\hY}{\mathcal{Y}}
\newcommand{\hZ}{\mathcal{Z}}

\newcommand{\bN}{\bm{N}}

\newcommand{\kp}{\kappa}
\newcommand{\dt}{\tau}

\newcommand{\opnorm}[1]{|\!|\!| #1 |\!|\!|}

\newcommand{\Cp}{C_\text{\rm per}}

\newtheorem{assumption}{Assumption}
\newtheorem{example}{Example}[section]
\newtheorem{remark}{Remark}

\title{Maximum bound principles for a class of semilinear parabolic equations and exponential time differencing schemes}

\author{
Qiang Du\footnotemark[2]
\and Lili Ju\footnotemark[3]
\and Xiao Li\footnotemark[4]
\and Zhonghua Qiao\footnotemark[5]
}

\begin{document}

\maketitle

\renewcommand{\thefootnote}{\fnsymbol{footnote}}

\footnotetext[2]
{Department of Applied Physics and Applied Mathematics and Data Science Institute, Columbia University, New York, NY 10027, USA (qd2125@columbia.edu).
Q. Du's work is partially supported by US National Science Foundation grant DMS-1719699,
US ARO MURI grant W911NF-15-1-0562 and US AFOSR MURI Center for Material Failure Prediction Through Peridynamics.}
\footnotetext[3]
{Department of Mathematics, University of South Carolina, Columbia, SC 29208, USA (ju@math.sc.edu).
L. Ju's work is partially supported by US National Science Foundation grant DMS-1818438
and US Department of Energy grants DE-SC0016540 and DE-SC0020270.}
\footnotetext[4]
{Department of Mathematics, University of South Carolina, Columbia, SC 29208, USA.
Current address:
Department of Applied Mathematics, The Hong Kong Polytechnic University, Hung Hom, Kowloon, Hong Kong (xiao1li@polyu.edu.hk).
X. Li's work is partially supported by National Natural Science Foundation of China grant 11801024.}
\footnotetext[5]
{Department of Applied Mathematics, The Hong Kong Polytechnic University, Hung Hom, Kowloon, Hong Kong (zhonghua.qiao@polyu.edu.hk).
Z. Qiao's work is partially supported by the Hong Kong Research Council GRF grants 15302214 and 15325816
and the Hong Kong Polytechnic University fund 1-ZE33.}

\renewcommand{\thefootnote}{\arabic{footnote}}

\begin{abstract}
The ubiquity of semilinear parabolic equations has been illustrated in their numerous applications
ranging from physics, biology, to materials and social sciences.
In this paper, we consider a practically desirable property for a class of semilinear parabolic equations
of the abstract form $u_t=\hL u+f[u]$ with $\hL$ being a linear dissipative operator
and $f$ being a nonlinear operator in space,
namely a time-invariant maximum bound principle, in the sense that
the time-dependent solution $u$ preserves for all time a uniform pointwise bound in absolute value
imposed by its initial and boundary conditions.
We first study an analytical framework for some sufficient conditions on $\hL$ and $f$
that lead to such a maximum bound principle for the time-continuous dynamic system of infinite or finite dimensions.
Then, we utilize a suitable exponential time differencing approach with a properly chosen generator of
contraction semigroup to develop first- and second-order accurate temporal discretization schemes, that satisfy
the maximum bound principle unconditionally in the time-discrete setting.
Error estimates of the proposed schemes are derived along with their energy stability.
Extensions to vector- and matrix-valued systems are also discussed.
We demonstrate that the abstract framework and analysis techniques developed here
offer an effective and unified approach to study the maximum bound principle of the abstract evolution equation
that cover a wide variety of well-known models and their numerical discretization schemes.
Some numerical experiments are also carried out to verify the theoretical results.
\end{abstract}



\pagestyle{myheadings}
\thispagestyle{plain}
\markboth{Q. DU, L. JU, X. LI, AND Z. QIAO}
{Maximum Bound Principles of Semilinear Parabolic Equations and ETD Schemes}

\section{Introduction}

Semilinear parabolic equations of the form
\begin{equation}
\label{intro_model}
u_t=\hL u+f[u],
\end{equation}
with $u$ being a time-dependent quantity of interests defined over a spatial domain  $\Omega$,
$\hL$ being a linear classic elliptic operator or its nonlocal variant,
and $f$ representing a nonlinear operator,
have been used to model numerous phenomena in nature.
Many model equations like \eqref{intro_model} and their solutions often satisfy some properties
such as maximum principle, comparison principle, existence of invariant regions, and energy decay.
These properties represent important physical features
and are also essential for mathematical analysis and numerical simulations.
For instance, classic reaction-diffusion equations can be seen as special cases of \eqref{intro_model}
where $\hL$ is given by a diffusion operator and $f$ is a reaction source:
the diffusion process causes the concentration $u$ of some substance to spread in space
and the reaction drives the dynamics based on the concentration values.
One illustration is the dimensionless ignition model \cite{KaPo80}
for a supercritical high activation energy thermal explosion of a solid fuel
in a bounded container, described by
\begin{equation}
\label{intro_exp}
u_t=\Delta u+\e^u.
\end{equation}
Since the reaction term $\e^u$ is always positive,
the solution must reach its minimum either at the initial time or on the boundary of the domain,
which is a popular form of the well-studied \emph{maximum principle} for parabolic equations.
Another popular example of \eqref{intro_model} is the so-called Allen--Cahn equation
 \cite{AlCa79},
which takes on the form
\begin{equation}
\label{intro_AllenCahn}
u_t=\varepsilon^2\Delta u+u-u^3,
\end{equation}
with $\varepsilon>0$ reflecting the width of the transition regions.
The equation is also a special case of the Ginzburg--Landau theory
that has been introduced earlier in the modeling of superconductivity \cite{GiLa50} and other phase transition problems.
It is well-known that the equation \eqref{intro_AllenCahn} and its more general form (as discussed later)
hold a \emph{maximum bound principle} (MBP) \cite{DuGuPe92,EvSoSo92}:
if the initial data and/or the boundary values are pointwisely bounded by $1$ in absolute value,
then the absolute value of the solution is also bounded by $1$ everywhere and for all time.
This MBP is related to the conventional version of maximum principle described for \eqref{intro_exp},
but also differs somewhat in scope, so we use a new name to distinguish them.
For the scalar model like the equation \eqref{intro_AllenCahn},
it is equivalent to the existence of special {\it upper} and {\it lower solutions}
(the constant functions $\overline{u}\equiv 1$ and $\underline{u}\equiv-1$ respectively)
of \eqref{intro_AllenCahn} under Dirichlet or homogeneous Neumann boundary condition \cite{Pao92}.
Moreover, it can be also described as the equation \eqref{intro_AllenCahn}
having a time-invariant region \cite{Amann78,Kuiper80,Smoller94}
since the region $[-1,1]$, of the value of the solution, remains unchanged during the time evolution.
In this paper, we consider an abstract form of evolution equations given by \eqref{intro_model}
and focus on the problems having the MBP or the existence of the special form of the invariant region,
namely, an upper bound on the pointwise absolute value of the solution
(suitable defined for vector- and matrix-valued quantities).
We note that the MBP is weaker than the conventional maximum principle in the sense that
a problem satisfying a maximum principle must satisfy an MBP.
Thus, the study of the conventional maximum principle,
particularly with respect to the linear operator $\hL$, plays relevant and important roles.

Whether the equation \eqref{intro_model} has a maximum principle or not
depends highly on the property of the linear operator $\hL$.
Concerning the linear operator $\hL$,
one can find from the standard textbooks on partial differential equations (see, e.g., \cite{Evans00}) that
the equation \eqref{intro_model} with the uniformly elliptic linear operator $\hL$ satisfies a maximum principle.
During the past several decades,
there have also been many studies devoted to the maximum principles
for numerical approximations of linear elliptic operators. While it is too numerous to
detail them all here,  a partial list of earlier works includes the cases for
finite difference method \cite{BrHu64,Ciarlet70,Price68,Varga66},
finite element method \cite{BaHi15,BrKoKr08,BuEr04,CiRa73,KaKoKr07},
collocation method \cite{Yanik87,Yanik89},
and finite volume method \cite{YuYu18}.
In \cite{Ciarlet70},
a systematic analysis was presented on some sufficient conditions for the finite difference operator
to satisfy the weak and strong maximum principles in the discrete sense.
There were also works devoted to general analysis on algebraic properties of the approximating operators,
which consist of the discrete maximum principles as special cases \cite{FaKoSz13,FaKoSz15,MiHo12}.
Extensions to parabolic cases have also been made.
For example,
\cite{Karafiat91} provided some sufficient conditions for the discrete maximum principles
of the forward and backward Euler time-stepping methods,
and additional researches can be found, e.g.,
in \cite{FaHo06,FaKaKo12,KaKo15,LiIt01,LiYu14,YaXiQiXu16}.
In addition to the maximum principles for
both linear elliptic differential operators and their finite dimensional discretizations,
there are also some recent studies on the nonlocal analogues
for linear nonlocal integral operators \cite{DuTaTiYa19,NoZh18,TiJuDu15}.

With the linear operator $\hL$ satisfying the conditions to yield the maximum principle,
a suitable nonlinear term may lead to the existence of time-invariant regions of \eqref{intro_model}.
The MBP, that is a special invariant region
of the Allen--Cahn equation \eqref{intro_AllenCahn}, was proved in \cite{EvSoSo92}.
An important and natural  question for numerical analysis is
whether such an MBP could be preserved
by some time-stepping schemes for discretizing \eqref{intro_AllenCahn}.
This question has been studied in a variety of works recently.
In \cite{StVo15}, the discrete MBPs of
a finite difference semi-discrete scheme
and its fully discrete approximations with forward and backward Euler time-stepping methods
were obtained for \eqref{intro_AllenCahn} in one-dimensional space.
Later, the first-order stabilized implicit-explicit schemes
with finite difference spatial discretization were proved to preserve the MBP \cite{TaYa16},
which was then generalized by \cite{ShTaYa16} to the case with more general nonlinear terms.
The discrete MBP was also obtained in \cite{YaDuZh18}
by proving the uniform boundedness of the average $L^p$ norm of the solution
 for any even integer $p$,
followed by passing the limit as $p$ goes to infinity.
A first-order exponential time differencing (ETD) scheme (or say, exponential integrator scheme)
in the space-continuous setting was analyzed in \cite{DuZh05},
where some properties of the heat kernel were used.
In addition, the MBP-preserving numerical schemes have been also studied
for the analogues of \eqref{intro_AllenCahn},
such as the fractional Allen--Cahn equation using the Crank--Nicolson time-stepping \cite{HoTaYa17},
the nonlocal Allen--Cahn equation by using  first- and second-order ETD schemes \cite{DuJuLiQi19},
and the complex-valued Ginzburg--Landau model of superconductivity \cite{DuGuPe92}
by considering the finite volume method \cite{Du98}
and finite element method with the mass-lumping technique \cite{Du05} in space
with backward Euler time-stepping.

In this paper, we present a unified framework on the MBPs for
a class of semilinear parabolic equations \eqref{intro_model} with general boundary conditions and their spatially discretized systems,
as well as the time-discrete analogues of numerical approximations
based on the exponential time differencing method.
The ETD method \cite{BeKeVo98,CoMa02,HoOs05} comes from the variation-of-constants formula
with the nonlinear terms approximated by polynomial interpolations,
followed by exact integration of the resulting integrals.
We mainly address the following question:
\emph{Under what conditions does the equation \eqref{intro_model} have the MBP
and do numerical approximations of \eqref{intro_model} preserve the MBP?}
We provide an abstract framework on mathematical and numerical analysis for the MBP of \eqref{intro_model},
where the ETD method is used for the temporal discretization.
Our main contribution includes several aspects.
First, we systematically formulate an abstract mathematical framework
to illustrate the essential characteristics of the linear and nonlinear operators in \eqref{intro_model}
so that the model equation satisfies the MBP
and the corresponding first- and second-order ETD schemes preserve the discrete MBP unconditionally.
Second, we are able to present  results valid for a large class of
semilinear parabolic problems subject to different boundary conditions and constraints
which significantly generalize those obtained in \cite{DuJuLiQi19}
(which is a specialized study focused  only on the nonlocal Allen--Cahn equation with periodic boundary condition).
We demonstrate that the theory also works for problems subject to
either Dirichlet boundary condition or homogeneous Neumann boundary condition
in the classic (local continuum) and nonlocal sense.
These generalizations offer our theory a much wider range of applicability.
Third, in order to establish a broad and abstract framework presented here, new analysis techniques are developed.
They differ from the ones used in earlier works.
Indeed, the abstraction allows us to discover the essential ingredients of the MBPs, which were absent from earlier analysis.
In addition,  we also derive error estimates for the approximate solutions of MBP-preserving ETD schemes,
as well as their energy stability when applied to gradient flow  models.

The necessity to take nonhomogeneous Dirichlet boundary condition into account
comes from various practical considerations.
For instance, phase separations occurring in hydrocarbon systems
could be described by a diffuse-interface model with the Peng--Robinson equation of state \cite{QiSu14},
where a nonhomogeneous Dirichlet boundary condition is usually needed
to keep the microstructures with certain phases on the boundary.
In addition, nonhomogeneous Dirichlet boundary conditions are also necessary
for many scalable and multiscale algorithms based on domain and subspace decompositions.

One of the distinctive features of the ETD schemes is the exact evaluation of the contribution of the linear operator,
which provides good stability and accuracy even though the linear part has strong stiffness.
Besides, the ETD schemes usually perform as efficient as an explicit scheme
since the operator exponentials could be often implemented by some fast algorithms in regular domains.
Such advantages lead to successful applications of ETD schemes on a large class of phase-field models
which usually yield highly stiff ODE systems under spatial discretizations.
We refer the readers to the literature, e.g., \cite{DuJuLiQi19,JuLiQiZh18,JuZhDu15,JuZhZhDu15,ZhZhWaJuDu16,ZhJuZh16}.

The rest of this paper is organized as follows.
In Section \ref{sect_model_MBP}, we recall the semilinear equation  \eqref{intro_model}
in a Banach space consisting of real scalar-valued continuous functions,
declare the basic assumptions on the linear and nonlinear operators,
and show that the model equation has a unique solution
and satisfies the MBP under these assumptions.
We also present a variety of concrete examples of the linear and nonlinear operators in the space-continuous and
space-discrete settings
covered by the theoretical framework.
In Section \ref{sect_ETD_MBP}, first- and second-order ETD time-discrete schemes are constructed
and proved to preserve the discrete MBPs along with their convergence analysis.
In addition, energy stability is also analyzed when the framework is applied to gradient flow  models.
In Section \ref{sect_extension},
we discuss the extensions to vector-valued and matrix-valued problems in the space-continuous setting
with the complex-valued one being a special case.
In Section \ref{sect_numerical},
practical implementations of the ETD schemes are discussed
and some numerical experiments are also performed to verify the theoretical results.
Finally, concluding remarks are given in Section \ref{sect_conclusion}.

\section{Maximum bound principle of the model equation}
\label{sect_model_MBP}

\subsection{Abstract framework}
\label{subsect_framework}

Let us assume $\Omega$ is either a connected spatial region
or a collection of isolated points in $\R^d$.
More precisely, we consider  the following two situations:
\begin{itemize}
\item[(D1)] $\Omega$ is an open, connected and bounded set with a Lipschitz boundary denoted by $\partial\Omega$,
and $\Omega_c$ is a closed connected set disjoint with $\Omega$ but $\partial\Omega\subset\Omega_c$;
denote $\overline{\Omega}=\Omega\cup\partial\Omega$ and $\widehat{\Omega}=\Omega\cup\Omega_c$;
\item[(D2)] {for a given pair of sets $\widetilde{\Omega}$ and $\widetilde{\Omega}_c$ that are described in (D1),
and with $\Sigma$ consisting of all the nodes in a mesh partitioning $\widehat{\widetilde{\Omega}}$,
we let $\Omega=\widetilde{\Omega}\cap\Sigma$, $\partial\Omega=\partial\widetilde{\Omega}\cap\Sigma$,
$\Omega_c=\widetilde{\Omega}_c\cap\Sigma$, $\overline{\Omega}=\Omega\cup\partial\Omega$,
and $\widehat{\Omega}=\Omega\cup\Omega_c$.}
\end{itemize}
Note that in Case (D1), we have $\Omega_c=\partial\Omega$ for classic differential operators
and $\Omega_c$ is usually a nonempty volume for nonlocal integral operators.
When $\Omega_c=\partial\Omega$, we define $\Omega^*=\Omega$ and $\Omega_c^*=\partial\Omega$;
otherwise, we define $\Omega^*=\overline{\Omega}$ and $\Omega_c^*=\Omega_c\setminus\partial\Omega$.
For a problem with periodic boundary condition,  we assume that its period cell is a rectangle in $\R^d$, i.e.,
$\Omega=\prod_{i=1}^d(a_i,b_i)$, and we use a special notation $\overline\Omega_+=\prod_{i=1}^d(a_i,b_i]$.
Case (D2) corresponds to a discrete version of (D1).
This allows us to unify the discussions presented later with the same set of notations
for both space-continuous and space-discrete cases.
For any set $D\subset\R^d$,
we denote by $C(D)$ the space consisting of real scalar-valued continuous functions defined on $D$
and by $C_b(D)$ the set of all bounded functions in $C(D)$.
Here, the continuity of functions is defined as follows \cite{Rudin76}:
\begin{equation}
\label{def_contfun}
\text{$w$ is continuous at $\bx^*\in D$}\iff
\text{$\forall\,\bx_k\to \bx^*$ in $D$ implies $w(\bx_k)\to w(\bx^*)$}.
\end{equation}
If $D$ is bounded and closed, then obviously $C(D)=C_b(D)$.

Let $\hX=C(\overline{\Omega})$ if $\Omega_c=\partial\Omega$;
otherwise $\hX=\{w:\widehat\Omega\to\R\,|\,w|_{\overline{\Omega}}\in C(\overline{\Omega})
\text{ and } w|_{\Omega_c^*}\in C_b(\Omega_c^*)\}$
and we write $\hX=C(\overline{\Omega})\cap C_b(\Omega_c^*)$ below for simplicity.
Define the supremum norm on $\hX$ by
$$\|w\|=\sup_{\bx\in\widehat{\Omega}}|w(\bx)|,\quad \forall\, w\in \hX,$$
then $(\hX,\|\cdot\|)$ becomes a Banach space.
According to the definition \eqref{def_contfun},
$\hX$ is well-defined under Cases (D1) and (D2).
Let $f:C_b(\Omega^*)\to C_b(\Omega^*)$ be a nonlinear operator
and $\hL:D(\hL)\to C_b(\Omega^*)$ be a linear operator with the domain $D(\hL)\subset\hX$.
Then we consider the following two cases:
\begin{itemize}
\item[(C1)] $X=\{w|_{\Omega^*}\,|\,w\in \whX\}$ and $\partial X=C_b(\Omega_c^*)$, where $\whX =\{w\in\hX\,|\,w|_{\Omega_c^*}=0\}$;
\item[(C2)] $X=\{w|_{\overline{\Omega}_+}\,|\,w\in \whX\}$, where $\whX=\Cp{(\overline\Omega)}:= \{w\in C(\R^d)$ is periodic
with respect to the rectangle $\Omega\}$.
\end{itemize}
 It is  easy to see that
$X$ defined in any of  Cases (C1) and (C2) can be regarded as a linear subspace of $\hX$
in the sense of {\em isometric isomorphism} ($X\simeq \whX$)
through the zero or periodic extension to $\Omega_c^*$ or $\R^d$ correspondingly.
We always omit the extension mapping between $X$ and $\whX$ for simplicity when there is no ambiguity.
We also remark that $X$ equipped with the supremum norm $\|\cdot\|$ is a Banach space in both Cases (C1) and (C2),
so is $\partial X$ in Case (C1).

For Case (D1),  define $D(\hL_0)=\{w\in D(\hL)\cap X \,|\, \hL w\in X \}$ which is a subspace of $X$;
for Case (D2), define $D(\hL_0)=X$.
Then we define the linear operator $\hL_0=\hL|_{D(\hL_0)}:D(\hL_0)\to X$.
Note that for any $w\in D(\hL_0)$, it holds $\hL_0 w=\hL w$ in $\Omega^*$,
and $\hL_0 w$ is also well-defined in $\Omega^*_c$ in the sense of isomorphism.

The model problem we consider in this paper is a class of semilinear parabolic equations taking the following form
\begin{equation}
\label{model_eq}
u_t=\hL u+f[u],\quad t>0,\ \bx\in\Omega^*,
\end{equation}
where $u:[0,\infty)\times\widehat{\Omega}\to\R$ is the unknown function subject to the initial  condition
\begin{equation}
\label{model_eq_init}
u(0,\bx)=u_0(\bx),\quad \bx\in\widehat{\Omega}
\end{equation}
with $u_0\in\hX$
and either the periodic boundary condition for Case (C2)  or
the Dirichlet boundary condition
\begin{equation}
\label{model_boundary}
u(t,\bx) = g(t,\bx), \quad t \ge 0,\ \bx\in\Omega_c^*
\end{equation}
for Case (C1) with $g\in C([0,\infty);\partial X)$.
The compatibility condition  is also assumed to  hold, that is, $g(0,\cdot)=u_0$  on $\Omega^*_c$ for Case (C1)
and $u_0$ is $\Omega$-periodic for Case (C2).

In order to establish the maximum bound principle (MBP) for the model problem \eqref{model_eq},
as well as its time discretizations proposed later,
we make the following specific assumptions regarding the operators $\hL$, $\hL_0$ and $f$ defined above.

\begin{assumption}[Requirements on $\hL$ and $\hL_0$]
\label{assump_L}
The linear operators $\hL$ and $\hL_0$ satisfy the following conditions:

{\rm(a)} for any $w\in D(\hL)$ and $\bx_0\in\Omega^*$, if
\begin{equation}
\label{L_cond_c}
w(\bx_0)=\sup_{\bx\in\widehat{\Omega}}w(\bx),
\end{equation}
then $\hL w(\bx_0)\le0$;

{\rm(b)} the domain $D(\hL_0)$ is dense in $X$;

{\rm(c)} there exists $\lambda_0>0$ such that the operator $\lambda_0\hI-\hL_0:D(\hL_0)\to X$ is surjective,
where $\hI$ is the identity operator.
\end{assumption}

Note that it can be derived from Assumption \ref{assump_L}-(a) that
$\hL$ maps any constant function in $D(\hL)$ to the zero element in $\hX$.

\begin{assumption}[Requirements on $f$]
\label{assump_f}
The nonlinear operator $f$ acts as a composite function
induced by a given one-variable continuously differentiable function $f_0:\R\to\R$, that is,
\begin{equation}
\label{assump_f_composite}
f[w](\bx)=f_0(w(\bx)),\quad\forall\,w\in C_b(\Omega^*),\ \forall\,\bx\in{\Omega^*},
\end{equation}
and there exists a constant $\beta>0$ such that
\begin{equation}
\label{lincon}
f_0(\beta)\le0\le f_0(-\beta).
\end{equation}
\end{assumption}

\begin{remark}
\label{rem1}
A more general case is that
the function $f_0$ satisfies $f_0(M)\le0\le f_0(m)$ for some $M>m$
instead of \eqref{lincon} in Assumption \ref{assump_f}.
In this case, one can carry out an affine transform to the unknown function $u$.
More precisely,
take the affine transform $\eta:\R\to\R$ defined by
$$\eta(\theta)=\frac{M-m}{2\beta}\theta+\frac{M+m}{2},\quad\theta\in\R,$$
and define $\tilde{f}_0(\theta)=\frac{2\beta}{M-m}f_0(\eta(\theta))$ for any $\theta\in\R$,
then it holds $\tilde{f}_0(\beta)\le0\le\tilde{f}_0(-\beta)$.
By letting $\tilde{u}=\eta^{-1}(u)$,
we can obtain from \eqref{model_eq} that
\begin{equation*}
\tilde{u}_t=\hL\tilde{u}+\tilde{f}[\tilde{u}],\quad t>0,\ \bx\in\Omega^*,
\end{equation*}
where the nonlinear mapping $\tilde{f}$ is determined by the composite function
$$\tilde{f}[w](\bx)=\tilde{f}_0(w(\bx)),\quad\forall\,w\in C_b(\Omega^*),\ \forall\,\bx\in{\Omega^*},$$
and satisfies Assumption \ref{assump_f}.
\end{remark}

\begin{lemma}
\label{lem_L_semigroup}
Under Assumption \ref{assump_L},  the following  properties on  $\hL_0$ hold:

{\rm(i)}  $\hL_0$ is dissipative, i.e.,
for any $\lambda>0$ and any $w\in D(\hL_0)$,
\begin{equation*}
\|(\lambda\hI-\hL_0)w\|\ge\lambda\|w\|;
\end{equation*}

{\rm(ii)} $\hL_0$ is the generator of a contraction semigroup $\{S_{\hL_0}(t)\}_{t\ge0}$ on $X$,
i.e., $\opnorm{S_{\hL_0}(t)}\le1$, where $\opnorm{\cdot}$ denotes the operator norm  defined by
\begin{equation*}
\opnorm{\hT} = \sup_{\substack{w\in X, \|w\|=1}}\|\hT w\|.
\end{equation*}
\end{lemma}

\begin{proof}
We first prove the result for Case (C1).
For any $w\in D(\hL_0)$, since
$w|_{\overline\Omega}\in C(\overline\Omega)$ and $w|_{\Omega_c^*} =0$,
it is clear that $|w(\bx)|$ reaches its maximum at  some point $\bx_0\in\widehat{\Omega}$,
namely, $|w(\bx_0)|=\|w\|$.
Without loss of generality, let us  assume $w(\bx_0)\ge 0$; otherwise, we consider $-w$ instead.
If $\bx_0\in\Omega^*$, we know that $\hL_0 w(\bx_0)=\hL w(\bx_0)\le0$  by Assumption \ref{assump_L}-(a).
If $\bx_0\in\Omega_c^*$, then $\hL_0 w(\bx_0)=0$ by the definition of $X$.
For any $\lambda>0$, we then have
$$\|(\lambda\hI-\hL_0)w\|\ge|\lambda w(\bx_0)-\hL_0 w(\bx_0)|
=\lambda w(\bx_0)-\hL_0 w(\bx_0)\ge\lambda w(\bx_0)=\lambda\|w\|,$$
which completes the proof of (i).
Then, according to (i), Assumptions \ref{assump_L}-(b) and \ref{assump_L}-(c),
the property (ii) follows from the Lumer--Phillips Theorem \cite[Theorem II.3.15]{EnNa00}.

Now we consider Case (C2) and $\bx_0\in\widehat{\Omega}$ such that $|w(\bx_0)|=\|w\|$.
If $\bx_0\in\Omega_c^*$, by the $\Omega$-periodicity, we can regard $\bx_0$ as a point in $\Omega^*$.
Then, the above analysis could be similarly done to obtain (i) and (ii).
\end{proof}

\begin{remark}
\label{proneu}
The proof of Lemma \ref{lem_L_semigroup}-(i) uses the
Assumption \ref{assump_L}-(a) (which is only related to $\hL$) and the definitions of $X$ and $D(\hL_0)$
to deduce the fact that $\hL_0 w(\bx_0)\leq 0$ if \eqref{L_cond_c} holds with $\bx_0\in\widehat\Omega$.
Lemma \ref{lem_L_semigroup}-(ii) is the consequence of Lemma \ref{lem_L_semigroup}-(i) and Assumption \ref{assump_L}-(b) and (c).
We note that Lemma \ref{lem_L_semigroup}-(ii) is the key result to be used in later discussions.
It can be established under different assumptions by using other approaches,
see discussions in Example \ref{eg_linear_fraclaplace}.
\end{remark}

Next, let us introduce a stabilizing constant $\kp>0$ and
rewrite the equation \eqref{model_eq} in the following equivalent form:
\begin{equation}
\label{model_ODE}
u_t + \kp u = \hL u + \hN[u], \quad t>0,\ \bx\in\Omega^*,
\end{equation}
where $\hN=\kp\hI+f$.
According to \eqref{assump_f_composite} in Assumption \ref{assump_f},
we know
\begin{equation*}
\hN[w](\bx)=N_0(w(\bx)),\quad\forall\,w\in C_b(\Omega^*),\ \forall\,\bx\in{\Omega^*},
\end{equation*}
where
\begin{equation}
\label{N0_definition}
N_0(\xi)=\kp\xi+f_0(\xi),\quad \xi\in\R.
\end{equation}
We impose a requirement on the selection of the stabilizing constant $\kp$ as
\begin{equation}
\label{kappacon}
\kp\ge\max_{|\xi|\le\beta}|f_0'(\xi)|.
\end{equation}
Note that \eqref{kappacon} always  can be reached
since $f_0$ is continuously differentiable.

\begin{lemma}
\label{lem_nonlinear}
Under Assumption \ref{assump_f} and the requirement \eqref{kappacon}, it holds that

{\rm(i)} $|N_0(\xi)|\le\kp\beta$ for any $\xi\in[-\beta,\beta]$;

{\rm(ii)} $|N_0(\xi_1)-N_0(\xi_2)|\le2\kp|\xi_1-\xi_2|$ for any $\xi_1,\xi_2\in[-\beta,\beta]$.
\end{lemma}

\begin{proof}
From \eqref{N0_definition}, we have $N_0'(\xi)=\kp+f_0'(\xi)$ and
$$0\le N_0'(\xi)\le2\kp,\quad\forall\,\xi\in[-\beta,\beta],$$
which leads to (ii) directly.
From Assumption \ref{assump_f}, for any $\xi\in[-\beta,\beta]$, we know that
$$-\kp\beta\le-\kp\beta+f_0(-\beta)=N_0(-\beta)\le N_0(\xi)\le N_0(\beta)=\kp\beta+f_0(\beta)\le\kp\beta,$$
which gives (i).
\end{proof}

Now we can show that the aforementioned  model  problem  admits a unique solution and possesses the MBP.
Our main theorem can be stated as follows.

\begin{theorem}
\label{thm_model_max}
Given any constant $T>0$.
Under Assumptions \ref{assump_L} and \ref{assump_f},
if
\begin{subequations}
\label{cond_initial_BC}
\begin{equation}
|u_0(\bx)| \le \beta, \quad \forall\,\bx\in\widehat{\Omega},
\end{equation}
then  the equation \eqref{model_eq} subject to the initial condition \eqref{model_eq_init} and  either the periodic boundary condition or the Dirichlet boundary condition \eqref{model_boundary} with
\begin{equation}
\label{cond_BC}
|g(t,\bx)| \le \beta, \quad \forall\,t\in[0,T],\ \forall\,\bx\in\Omega_c^*
\end{equation}
\end{subequations}
 has a unique solution $u\in C([0,T];\hX)$
and it satisfies $\|u(t)\|\le\beta$ for any $t\in[0,T]$.
\end{theorem}

\begin{proof}
Let us consider  Case (C1) first.
Denote $\hX_\beta=\{w\in \hX\,|\,\|w\|\le\beta\}$
and $C_g([0,t];\hX_{\beta})=\{w\in C([0,t];\hX_{\beta})\,|\,w|_{[0,t]\times\Omega_c^*}=g\}$ for any $t\in(0,T]$.
For a fixed $t_1>0$ and a given $v\in C_g([0,t_1];\hX_{\beta})$,
let us define $w:[0,t_1]\to\hX$ to be the solution of the following linear problem:
\begin{equation}
\label{thm_mbp_pf1}
\begin{dcases}
w_t + \kp w = \hL w + \hN[v], & t\in(0,t_1], \ \bx\in\Omega^*,\\
w(t,\bx)= g(t,\bx), & t\in[0,t_1], \ \bx\in\Omega_c^*,\\
w(0,\bx) = u_0(\bx), & \bx\in\widehat{\Omega},
\end{dcases}
\end{equation}
where the constant $\kp$ satisfies \eqref{kappacon}.
It is easy to know that $w$ is uniquely defined on $[0,t_1]\times\widehat{\Omega}$
by noticing that setting $\hN[v]=0$, $g=0$, and $u_0=0$ in \eqref{thm_mbp_pf1}
leads to $w(t)\in D(\hL_0)$, and thus $w(t)=\e^{-\kp t}S_{\hL_0}(t)u_0 = 0$ for each $t\in[0,t_1]$.
Suppose there exists $(t^*,\bx^*)\in(0,t_1]\times\widehat{\Omega}$ such that
$w(t^*,\bx^*)=\max_{0\le t\le t_1}\|w(t)\|$.
By \eqref{cond_BC},
we only need to consider the case $\bx^*\in\Omega^*$.
According to Assumption \ref{assump_L}-(a),
we have $w_t\ge0$ and $\hL w\le 0$ at $(t^*,\bx^*)$,
which implies $\kp w(t^*,\bx^*)\le N_0(v(t^*,\bx^*))$.
Since $|v(t^*,\bx^*)|\le\beta$, according to Lemma \ref{lem_nonlinear}-(i),
we obtain $w(t^*,\bx^*)\le\beta$.
Similarly, if there exists $(t^{**},\bx^{**})\in(0,t_1]\times\widehat{\Omega}$
such that
$w(t^{**},\bx^{**})=-\max_{0\le t\le t_1}\|w(t)\|$,
one can show $w(t^{**},\bx^{**})\ge-\beta$.
Thus we obtain $\|w(t)\|\le\beta$ for any $t\in[0,t_1]$,
which means $w\in C_g([0,t_1];\hX_\beta)$.

Next we define a mapping $\hA:C_g([0,t_1];\hX_\beta)\to C_g([0,t_1];\hX_\beta)$
by $\hA[v]=w$ through \eqref{thm_mbp_pf1}.
For $v,\tilde{v}\in C_g([0,t_1];\hX_\beta)$, we see that $w=\hA[v]$ and $\tilde{w}=\hA[\tilde{v}]$ satisfy
\begin{equation*}
w(t)-\tilde{w}(t)
=\int_0^t\e^{-\kp(t-s)}S_{\hL_0}(t-s)\big\{\hN[v(s)]-\hN[\tilde{v}(s)]\big\}\,\d s, \quad t\in[0,t_1].
\end{equation*}
Using Lemma \ref{lem_nonlinear}-(ii),
we can estimate the difference between $w(t)$ and $\tilde{w}(t)$ by
\begin{align*}
\|w(t)-\tilde{w}(t)\|
& \le\int_0^t\e^{-\kp(t-s)}\opnorm{S_{\hL_0}(t-s)}\|\hN[v(s)]-\hN[\tilde{v}(s)]\|\,\d s\\
& \le2\kp\int_0^t\e^{-\kp(t-s)}\|v(s)-\tilde{v}(s)\|\,\d s\\
& \le2(1-\e^{-\kp t_1})\|v-\tilde{v}\|_{C([0,t_1];\hX)}, \quad t\in[0,t_1],
\end{align*}
then,
$$\|\hA[v]-\hA[\tilde{v}]\|_{C([0,t_1];\hX)}
=\|w-\tilde{w}\|_{C([0,t_1];\hX)}
\le2(1-\e^{-\kp t_1})\|v-\tilde{v}\|_{C([0,t_1];\hX)}.$$
When $t_1<\kp^{-1}\ln2$, we see that $\hA$ is a contraction.
Since $\hX_\beta$ is closed in $\hX$,
we know that $C_g([0,t_1];\hX_\beta)$ is complete
with respect to the metric reduced by the norm $\|\cdot\|_{C([0,t_1];\hX)}$.
Then we can apply Banach's fixed-point theorem to get the existence of a unique solution $u\in  C_g([0,t_1];\hX_\beta)$ to the model equation \eqref{model_eq} on $[0,t_1]$.
Using standard continuation argument,
we then have the global existence of the unique solution $u\in C_g([0,T];\hX_\beta)$ for any $T>0$.

For Case (C2), the system \eqref{thm_mbp_pf1} still holds after removing the second equation in it.
For the case $\bx^*\in\Omega_c^*$, by the periodicity,
we could regard $\bx^*$ as a point in $\Omega^*$.
Then,  the analysis above is still valid.
Putting all the different cases together, we get a complete proof.
\end{proof}

\begin{remark}
We note that Theorem \ref{thm_model_max} also could be proved using other classical methods,
such as the method of upper and lower solutions \cite{Pao92} in the case of scalar second-order elliptic operator.
We emphasize that, in the proof given here, the MBP of the model equation \eqref{model_eq}
has no explicit dependence on the constant $\kp$,
even though we have assumed that $\kp$ satisfies \eqref{kappacon} in order to
use the Lemma \ref{lem_nonlinear}. Meanwhile, we will see later that the choice of the constant $\kp$
plays an important role  in designing MBP-preserving ETD schemes.
\end{remark}

In the following subsections,
we present some examples of the nonlinear and linear operators
which are applicable to the mathematical framework established above.

\subsection{Examples of the nonlinear function $f_0$}
\label{nonope}

One usually chooses $f_0$ as $f_0=-F'$ with $F$ being a primitive function,
when considering the phase-field models, or says, the gradient flows of some free energy functionals.

\begin{example}
\label{eg_nonlinear_logistic}
\rm Consider the function
\begin{equation}
\label{eg_f_logistic}
f_0(s)=\lambda s(1-s^p),
\end{equation}
where $\lambda>0$ and $p\in\N_+$.
According to Remark \ref{rem1},
$f_0$ satisfies that $f_0(M)\le0\le f_0(m)$ with any $M\ge1$ and $1\geq m\geq 0$.
Especially, for an even integer $p$, one can  choose $\beta\ge1$ such that
$f_0$ satisfies Assumption \ref{assump_f},
then the requirement \eqref{kappacon} becomes $\kp\ge\lambda[(p+1)\beta^p-1]$.
The special case $p=1$ gives the well-known logistic function
and the case $p=2$ with $\lambda=1$ gives
\begin{equation}
\label{eg_f_quartic}
f_0(s)=s-s^3,
\end{equation}
the derivative of $-F$ with $F(s)=\frac{1}{4}(s^2-1)^2$ (the quartic double-well potential).
\end{example}

\begin{example}
\label{eg_nonlinear_flory}
\rm Consider the Flory--Huggins free energy
\[
F(s)=\frac{\theta}{2}[(1+s)\ln(1+s)+(1-s)\ln(1-s)]-\frac{\theta_c}{2}s^2,
\]
where $\theta$ and $\theta_c$ are two constants satisfying $0<\theta<\theta_c$, and
\begin{equation}
\label{eg_f_flory}
f_0(s)=-F'(s)=\frac{\theta}{2}\ln\frac{1-s}{1+s}+\theta_cs.
\end{equation}
Denoting by $\rho$ the positive root of $f_0(\rho)=0$, i.e.,
\begin{equation}
\label{eg_f_flory_gamma}
\frac{1}{2\rho}\ln\frac{1+\rho}{1-\rho}=\frac{\theta_c}{\theta},
\end{equation}
we can find that $\rho\in\Big(\sqrt{1-\frac{\theta}{2\theta_c-\theta}},1\Big)$.
Here, $f_0$ satisfies the condition in Assumption \ref{assump_f} with $\beta\in[\rho,1)$,
and then the requirement \eqref{kappacon} becomes $\kp\ge\frac{\theta}{1-\beta^2}-\theta_c$.
\end{example}

\begin{example}
\label{pr}
\rm The Peng--Robinson equation of state \cite{PeRo76}
is widely used in the oil industries and petroleum engineering.
The Helmholtz free-energy density considered in such a model can be expressed as
\[
F(s) = RTs(\ln s-1) - RTs\ln(1-bs) + \frac{as}{2\sqrt{2}b}\ln\frac{1+(1-\sqrt{2})bs}{1+(1+\sqrt{2})bs},
\]
where $R$ is the universal gas constant, $T$ the temperature,
$a=a(T)$ the energy parameter, and $b$ the covolume parameter.
The values of these parameters could be found in \cite{QiSu14}.
Then
\begin{equation}
\label{pr_f0}
f_0(s) = - RT\ln\frac{s}{1-bs} - \frac{RTbs}{1-bs}
- \frac{a}{2\sqrt{2}b} \ln\frac{1+(1-\sqrt{2})bs}{1+(1+\sqrt{2})bs}
+ \frac{as}{1+2bs-b^2s^2},
\end{equation}
which has two zero points $m$ and $M$ satisfying $0<m<M<1/b$.
This example falls into the situation discussed in Remark \ref{rem1}.
\end{example}

\subsection{Examples of the linear operator $\hL$}
\label{linope}

\subsubsection{Infinite dimensional examples}

Here  we present some examples of the linear operator $\hL$ in the form of classic differential
or nonlocal integral  operators corresponding to Case (D1).
We will verify that $\hL$ and the {induced operator $\hL_0$} satisfy Assumption \ref{assump_L}.

\begin{example}[{\em Second-order elliptic differential operator}]
\label{eg_linear_elliptic}
\rm Consider the linear operator
\begin{equation}
\label{eg_L_elliptic}
\hL w(\bx)=A(\bx):\nabla^2w(\bx)+\bm{q}(\bx)\cdot\nabla w(\bx),\quad\bx\in\Omega,
\end{equation}
where $\bm{q}\in C(\overline{\Omega};\R^d)$ and
$A\in C(\overline{\Omega};\R^{d\times d})$ is symmetric and positive definite uniformly (i.e.,
there exists $\theta>0$ such that $\bxi^TA(\bx)\bxi\ge\theta\bxi^T\bxi$ for all $\bx\in\overline{\Omega}$ and $\bxi\in\R^d$).
In this case,  $\Omega^*=\Omega$,  $\Omega_c^*=\partial\Omega$,  $\widehat\Omega=\overline\Omega$, the boundary condition \eqref{model_boundary} is the classic Dirichlet one,
$\hX=C(\overline\Omega)$, and  $D(\hL)=C^2(\overline{\Omega})$.
For any $w\in D(\hL)$ and $\bx_0\in\Omega$ such that \eqref{L_cond_c} holds,
we always have $\nabla w(\bx_0)=\bm{0}$ and $\nabla^2w(\bx_0)$ is negative semi-definite.
Since $A(\bx_0)$ is symmetric, there exists an orthogonal matrix $O\in\R^{d\times d}$ such that
$$OA(\bx_0)O^T=\diag\{\lambda_1,\lambda_2,\dots,\lambda_d\}=:\Lambda$$
with $\lambda_i\ge\theta$ for all $i=1,2,\dots,d$.
Thus, we have
$$A(\bx_0):\nabla^2w(\bx_0)=(O^T\Lambda O):\nabla^2w(\bx_0)=\Lambda:(O^T\nabla^2w(\bx_0)O).$$
Since $O^T\nabla^2w(\bx_0)O$ is symmetric and negative semi-definite,
its diagonal entries are all non-positive.
Thus, we obtain $A(\bx_0):\nabla^2w(\bx_0)\le0$ and then $\hL w(\bx_0)\le0$,
which gives Assumption \ref{assump_L}-(a).
Assumption \ref{assump_L}-(b) is guaranteed by the fact that
$X=C_0(\overline{\Omega})$ and
$D(\hL_0)=\{w\in C^2(\overline{\Omega})\cap C_0(\overline{\Omega})\,|\,(\hL w)|_{\partial\Omega}=0\}$ for Case (C1),
and $X=\Cp(\overline{\Omega})$ and
$D(\hL_0)=C^2(\overline{\Omega})\cap\Cp(\overline{\Omega})$ for Case (C2) in the sense of isomorphism.
Assumption \ref{assump_L}-(c) can be obtained
from the standard analysis of the existence and regularity of the solution
to the partial differential equation $\lambda w-\hL_0 w=f$ in $\Omega$ for any $\lambda>0$ and $f\in X$.
\end{example}

\begin{remark}
The second-order elliptic differential operator could be also defined in the divergence form:
\begin{equation}
\label{eg_L_elliptic_div}
\widetilde{\hL}w(\bx)=\nabla\cdot(A(\bx)\nabla w(\bx))+\widetilde{\bm{q}}(\bx)\cdot\nabla w(\bx),\quad \bx\in\Omega
\end{equation}
with $A\in C(\overline{\Omega};\R^{d\times d})\cap C^1(\Omega;\R^{d\times d})$
and $\widetilde{\bm{q}}\in C(\overline{\Omega};\R^d)$.
This form could be written in a non-divergence form \eqref{eg_L_elliptic}
by setting $\bm{q}(\bx)=\nabla\cdot A(\bx)+\widetilde{\bm{q}}(\bx)$.
Similarly we can show the operator $\widetilde{\hL}$ and  $\widetilde{\hL}_0$ also satisfies Assumption \ref{assump_L}.
\end{remark}

\begin{example}[{\em Nonlocal diffusion operator}]
\label{eg_linear_nonlocal}
\rm Consider the integral operator \cite{DuGuLeZh12}
\begin{equation}
\label{eg_L_nonlocal}
\hL w(\bx)=\int_{\widehat{\Omega}}\gamma(\bx,\by)(w(\by)-w(\bx))\,\d \by,\quad \bx\in\overline\Omega,
\end{equation}
where $\gamma:\widehat{\Omega}\times\widehat{\Omega}\to\R$ is a symmetric nonnegative kernel function,
i.e., $\gamma(\bx,\by)=\gamma(\by,\bx)\ge0$.
We consider the widely studied case of the nonlocal operator \eqref{eg_L_nonlocal},  in which
the kernel is radial and parameterized by a horizon parameter $\delta>0$ much less than the size of $\Omega$,
i.e., $\gamma(\bx,\by)=\gamma_{\delta}(|\bx-\by|)$ for some nonnegative function
$\gamma_{\delta}:\R\to\R$ with $\gamma_{\delta}|_{\R\setminus(0,\delta]}=0$
and $\gamma_{\delta}|_{(0,\delta/2]}\ge\gamma^*$ for some constant $\gamma^*>0$.
In this case, $\Omega^*=\overline\Omega$,
$\Omega_c^*=\Omega_\delta := \{\by\in\R^d\setminus\overline\Omega\,|\,\exists\,\bx\in\overline\Omega \text{ such that }|\bx-\by|\leq \delta\}$,
 $\widehat\Omega=\overline\Omega\cup\Omega_{{\delta}}$,
 the boundary condition \eqref{model_boundary}  becomes a volume constraint, and $\hX=C(\overline{\Omega})\cap C_b(\Omega_{\delta})$.
Then, the operator \eqref{eg_L_nonlocal} can be re-expressed as
\begin{eqnarray}
\label{eg_L_nonlocal_var}
\hL w(\bx)&=&\int_{B_{\delta}(\bm{\bx})}\gamma_{\delta}(|\bx-\by|)\big(w(\by)-w(\bx)\big)\,\d \by\nonumber\\
&=&\frac{1}{2}\int_{B_{\delta}(\bm{0})}\gamma_{\delta}(|\by|)\big(w(\bx+\by)+w(\bx-\by)-2w(\bx)\big)\,\d \by,
\quad \bx\in\overline\Omega,
\end{eqnarray}
For suitable $\gamma_{\delta}$ subject to the finite second-order moment condition:
\begin{equation*}
\int_{B_{\delta}(\bm{0})}|\by|^2\gamma_{\delta}(|\by|)\,\d \by=2d,
\end{equation*}
the operator $\hL$ is consistent with the standard Laplacian $\Delta$ as $\delta\to0$
(see, e.g., \cite{Du19,DuGuLeZh12}). Assumption \ref{assump_L}-(a) results from the nonnegativity of the kernel.
Whether Assumptions \ref{assump_L}-(b) and \ref{assump_L}-(c) are satisfied depends on the choice of the kernel.
Let us consider the situation of integrable kernels only (i.e., $\gamma_{\delta}(|\by|)\in L^1(\R^d)$),
for which $D(\hL)=\hX$ and $\hL w(\bx)=(\gamma_\delta *w)(\bx) -\alpha_\delta w(\bx)$ for any $ \bx\in\overline\Omega$. Here,   $\gamma_\delta*w$ denotes the convolution
and  $\alpha_\delta = \| \gamma_{\delta}(|\cdot|)\|_{L^1(\R^d)}$.
For Case (C1), we have
$X=C(\overline{\Omega})$ and $D(\hL_0)=C(\overline{\Omega})$ in the sense of isomorphism using zero extension to $\Omega_{\delta}$
thus Assumption \ref{assump_L}-(b) follows automatically.
For any $\lambda>0$ and $f\in C(\overline{\Omega})$,
there exists a unique solution $w\in L^{\infty}({\widehat\Omega})$ to the equation $\lambda w - \hL_0 w =f$ in $\overline\Omega$,
and it  further can be shown  $w\in C({\overline\Omega})$ \cite{Chasseigne07,DuYi19} by the property that the convolution of $L^{\infty}$ and $L^{1}$ functions is uniformly continuous,
which verifies Assumption \ref{assump_L}-(c).
For Case (C2), we have  $X=D(\hL_0)=\Cp(\overline{\Omega})$ and all the assumptions are still satisfied.
Regarding non-integrable kernels with strong singularity at the origin like the fractional power,
we refer to discussion in the next example.
\end{example}

\begin{example}[{\em Fractional Laplace operator}]
\label{eg_linear_fraclaplace}
\rm  For a fixed $s\in(0,1)$,
the fractional Laplace operator $(-\Delta)^s$ is the pseudo-differential operator with symbol $|\bxi|^{2s}$, that is,
\begin{equation*}
\mathcal{F}[(-\Delta)^sw](\bxi)=|\bxi|^{2s}\mathcal{F}[w](\bxi), \quad \bxi\in\R^d,
\end{equation*}
where $w\in\mathscr{S}(\R^d)$, the class of Schwarz functions,
and $\mathcal{F}$ denotes the Fourier transform.
Denoting $\hL=-(-\Delta)^s$,
an equivalent form of $\hL$ on a bounded spatial domain $\overline\Omega\subset \R^d$ is given by \cite{SaKiMa93}
\begin{equation}
\label{eg_L_fracLaplace}
\hL w(\bx)={c_{d,s}}\int_{\R^d}\frac{w(\by)-w(\bx)}{|\by|^{d+2s}}\,\d \by, \quad
\bx\in \overline\Omega
\end{equation}
with $c_{d,s}=\frac{2^{2s}\Gamma(d/2+s)}{ \pi^{d/2}\Gamma(-s)}$.   In this case,
$\Omega^*=\overline\Omega$, $\Omega_c^*=\R^d\setminus\overline\Omega$, $\widehat\Omega=\R^d$ and $\hX=C(\overline{\Omega})\cap C_b(\R^d\setminus\overline\Omega)$.
As shown in \cite{TiDuGu16}, the operator $\hL$ could be regarded as the limit, when $\delta\to\infty$,
of the nonlocal operator $\hL$ defined by \eqref{eg_L_nonlocal_var} equipped with {a special rescaled and
truncated fractional} kernel $\gamma_\delta(r)=c_{d,s}r^{-d-2s}$ for $r\in(0,\delta]$.
Let us consider Case (C1) only, for which $X=C(\overline{\Omega})$ in the sense of isomorphism using zero extension to $\R^d\setminus\overline\Omega$. Define $\hL_0$ as the restriction of $\hL$ on $X$.
Due to the lack of full elliptic regularity, Lemma \ref{lem_L_semigroup} is not readily applicable.
However,
we can still get Lemma \ref{lem_L_semigroup}-(ii) using the Beurling--Deny criteria \cite{EO92}.
Indeed, as proved in \cite{FeRo16} if the  boundary $\partial \Omega$ is $C^{1,1}$ and still valid if   $\partial \Omega$ is
Lipschitz, the equation $u_t-\hL u=0 $ in $(0,\infty)\times\overline\Omega$
subject to $u(t,\cdot)|_{\R^d\setminus\overline\Omega}=0$ with the initial data $u_0\in X$
has a unique weak solution $u(t,\cdot) \in C(\R^d)\cap H^s(\R^d)$ for all $t>0$,
i.e., $u(t,\cdot)|_{\R^d\setminus\overline\Omega}=0$ and
\[
\int_{\Omega} u_t(t,\bx) v(t,\bx) \,\d \bx+  \frac{c_{d,s}}{2} \int_{\R^d}\int_{\R^d}\frac{(u(t,\bx)-u(t,\by))(v(t,\bx)-v(t,\by))}{|\bx-\by|^{d+2s}}\,\d \bx\d\by
= 0,
\]
for any $v(t,\cdot)\in H^s(\R^d)$ with $v(t,\cdot)|_{{\R^d\setminus\overline\Omega}}=0$.
To obtain Lemma \ref{lem_L_semigroup}-(ii), i.e., the solution $u(t)=S_{\hL_0}(t) u_0$
induces a contraction semigroup $\{S_{\hL_0}(t)\}_{t\ge0}$ on $X$ with respect to the supremum norm $\|\cdot\|$,
it suffices to  show that for any $u_0\in X$ with $u_0(\bx)\ge 0$ on $\overline\Omega$,
one has $u(t,\bx)\ge 0$ in $(0,\infty)\times\overline\Omega$.
The latter can be checked by taking $v=-u^{-}:=\min\{u,0\}$ in the above weak form
to get that
\begin{align*}
\int_{\Omega} u_t(t,\bx) u^-(t,\bx) \,\d \bx
& = \frac{c_{d,s}}{2} \int_{u(t,\bx)<0}\int_{u(t,\by)<0}\frac{(u(t,\bx)-u(t,\by))^2}{|\bx-\by|^{d+2s}}\,\d \bx \d\by \\
& \quad + c_{d,s} \int_{u(t,\bx)<0}\int_{u(t,\by)\ge0}\frac{(u(t,\bx)-u(t,\by))u(t,\bx)}{|\bx-\by|^{d+2s}}\,\d \bx \d\by
\ge 0,
\end{align*}
which deduces
$$0\le \int_{u(t,\bx)<0} u^2(t,\bx) \,\d\bx \le \int_{u_0(\bx)<0} u^2_0(\bx) \,\d\bx = 0,$$
thus  $u(t,\bx)\ge0$ in $(0,\infty)\times\overline\Omega$.
With Lemma \ref{lem_L_semigroup}-(ii) established,
the first part of the proof of Theorem \ref{thm_model_max} can also be modified accordingly so that the same theorem remains valid.

\end{example}

To sum up, for the above examples, we can obtain the following result on the MBP according to Theorem \ref{thm_model_max}.

\begin{corollary}
\label{eg_cor_modeleq}
Let the nonlinear function $f_0$ be given by
either \eqref{eg_f_logistic} with even $p$ or \eqref{eg_f_flory} and the operator $\hL$ defined by
any of  \eqref{eg_L_elliptic}, \eqref{eg_L_nonlocal_var}
and \eqref{eg_L_fracLaplace}.
If \eqref{cond_initial_BC} holds,
then the solution of the model problem \eqref{model_eq} exists
and satisfies $|u(t,\bx)|\le\beta$ for any $t\ge0$ and $\bx\in\widehat{\Omega}$,
where $\beta=1$ for the case of \eqref{eg_f_logistic}  and $\beta=\rho$ for \eqref{eg_f_flory}, respectively.
\end{corollary}

\begin{remark}
When $f_0$ is given by either \eqref{eg_f_logistic} with general $p$ {or \eqref{pr_f0}},
based on Remark \ref{rem1} and Corollary \ref{eg_cor_modeleq},
we can  derive a  similar result on the MBP for the model equation \eqref{model_eq} in the form of
$m\le u(t,\bx)\le M$ for any $t\ge0$ and $\bx\in\widehat{\Omega}$.
\end{remark}

\subsubsection{Finite dimensional examples}
\label{sect_finitedim}

We consider some concrete linear operators for the situation of $\dim\hX<\infty$,
in particular, the case of $(\Omega,\Omega_c)$ is given by Case (D2).
Let us denote by $\hL_h$ the linear operator, with the subscript used
to distinguish it from the notation in the infinite dimensional cases.
Without loss of generality,
we assume that $\Omega$ is the node set of a uniform Cartesian mesh
on the hypercube $\widetilde{\Omega}=(0,1)^d$.
The choice of uniform mesh simplifies the notation but we note that
the analysis can be extended to non-uniform and non-Cartesian grids as discussed later
with respect to finite element (or finite volume) discretizations.
With the uniform Cartesian mesh, for a given $M_0\in\N_+$ and the uniform mesh size $h=1/M_0$, let
$\Omega =\{\bx_\bi=h\bi\,|\,\bi\in\{1,2,\dots,M_0-1\}^d\}$,
$\overline{\Omega} =\{\bx_\bi=h\bi\,|\,\bi\in\{0,1,2,\dots,M_0\}^d\}$,
and $\overline{\Omega}_+ =\{\bx_\bi=h\bi\,|\,\bi\in\{1,2,\dots,M_0\}^d\}$.
Thus, $X$ is isomorphic to $\R^{M^d}$ with the norm $\|\cdot\|$ equivalent to the vector infinity-norm.
Here, for Case (C1), we let $M=M_0-1$ if $\widetilde\Omega_c=\partial\widetilde\Omega$ and
$M=M_0+1$ otherwise,
while $M=M_0$ for Case (C2).
The mesh can be extended as needed to represent ${\widetilde\Omega}_c^*$,
and the solution on ${\Omega}_c^*$ (mesh nodes in $\widetilde\Omega_c^*$) is explicitly given for Case (C1)
or can be obtained from that on $\overline{\Omega}_+$ by the periodicity of the solution in ${\R^d}$ for Case (C2).

We may view the grid function $w\in X$ as a vector $\bw=(w_1,w_2,\dots,w_{M^d})^T\in\R^{M^d}$
with $w_i$ denoting the values of $w$ at nodes $\bx_\bi\in\Omega^*$ {(for Case (C1))
or $\bx_\bi\in\overline{\Omega}_+$ (for Case (C2))}
ordered lexicographically.
Then, the nonlinear mapping $f$ plays a role as a vector-valued function $\R^{M^d}\to\R^{M^d}$ satisfying
$$(f[\bw])_i=f_0(w_i),\quad\forall\,\bw\in\R^{M^d},\ 1\le i\le M^d.$$
In addition, {due to the discrete nature}, we now have that  $D(\hL_{h0})=X$ and
the linear operator $\hL_{h0}$  is equivalent to $\hL_h|_{X}$, which is  an $M^d$-by-$M^d$ matrix. Thus,
the semigroup generated by $\hL_{h0}$ is actually the matrix exponential $S_{\hL_{h0}}(t)=\e^{t\hL_{h0}}$.
In these settings, {Assumption \ref{assump_L}-(b) is trivial and
Assumption \ref{assump_L}-(c) could be verified more directly by the following result.}

\begin{proposition}
\label{prop_distlinear}
If a matrix $\hL_{h0}=(a_{ij})_{M^d\times M^d}$ satisfies
\begin{equation}
\label{assump_weakNDD}
|a_{ii}|\ge\sum_{\substack{j=1\\ j\not=i}}^{M^d}|a_{ij}|,\quad
a_{ii}<0,\quad a_{ij}\ge0\ (j\not=i)
\end{equation}
for all $\,i,j=1,2,\dots,M^d$, then $\hL_{h0}$ satisfies {Assumption \ref{assump_L}-(c)}.
\end{proposition}

\begin{example}[{\em Central difference operator for the Laplacian}]
\label{eg_laplacian_dis}
\rm To make the discussion more clearly, let us begin with the one-dimensional case.
Now, $\hX\simeq\R^{M_0+1}$.
For Case (C1), we have $\partial X\simeq\R^2$.
For any $w\in\hX$, the second-order central difference approximation of the Laplacian is defined by
\begin{equation}
\label{eg_local_diff}
\hL_hw(x_i)=\frac{1}{h^2}\big(w(x_{i-1})-2w(x_i)+w(x_{i+1})\big)
\end{equation}
for $x_i\in\Omega$ with $w(x_0)$ and $w(x_{M_0})$ determined by the boundary condition.
For Case (C2), for a grid function $w\in\hX$,
we still define $\hL_h$ as \eqref{eg_local_diff} for $x_i\in\overline{\Omega}_+$
but with $w(x_0)=w(x_{M_0})$ and $w(x_{M_0+1})=w(x_1)$ due to the periodicity.
Assumption \ref{assump_L}-(a) holds according to \eqref{eg_local_diff}.
Define
\[
D_h=\frac{1}{h^2}
\begin{pmatrix}
-2 & 1 & ~ & ~ & c\\
1 & -2 & 1 \\
~ & \ddots & \ddots & \ddots \\
~ & ~ & 1 & -2 & 1 \\
c & ~ & ~ & 1 & -2
\end{pmatrix}
\in\R^{M\times M},\quad\text{with }
c=
\begin{cases}
0, & \text{for Case (C1)},\\
1, & \text{for Case (C2)},
\end{cases}
\]
and $D_h$ satisfies \eqref{assump_weakNDD} obviously.
We have $\hL_{h0}=D_h$, so Assumption \ref{assump_L}-(c) is satisfied.

For the two- and three-dimensional cases,
we define $\hL_h$ as a summation of the central difference approximations
in every coordinate direction given as \eqref{eg_local_diff}.
Thus, $\hL_h$ satisfies Assumption \ref{assump_L}-(a).
Also, we can present the matrix $\hL_{h0}$ by using the Kronecker products as
\begin{align}
\hL_{h0} & =I_M\otimes D_h+D_h\otimes I_M,\quad\;\qquad\qquad\qquad\qquad\qquad\qquad\quad\;\;\text{(2-D)}\label{eg_local_diff2d}\\
\hL_{h0} & =I_M\otimes I_M\otimes D_h+I_M\otimes D_h\otimes I_M+D_h\otimes I_M\otimes I_M,\qquad\text{(3-D)}\label{eg_local_diff3d}
\end{align}
where $I_M$ is the $M\times M$ identity matrix,
and the vectors determined by the boundary conditions could be given in the similar way.
It is easy to check that $\hL_{h0}$ corresponding to \eqref{eg_local_diff2d} or \eqref{eg_local_diff3d}
satisfies Assumption \ref{assump_L}-(c).
\end{example}

\begin{remark}
It is well-known that \eqref{eg_local_diff2d}
provides the central difference approximation of the Laplacian in the two-dimensional case.
However, if one considers more general elliptic operators, such as \eqref{eg_L_elliptic},
including mixed partial derivatives and a convection term,
the corresponding discretized operator cannot be written in the form of Kronecker products.
Instead, the upwind formulas are adequate to discretize the derivatives in the convection term
and those in the mixed partial derivatives.
The detailed analysis of the resulted linear operator
is quite similar to Example \ref{eg_laplacian_dis} and we omit it.
A general formulation for the finite difference approximation of \eqref{eg_L_elliptic}
could be found in \cite{Ciarlet70},
where a sufficient and necessary condition (more general than Proposition \ref{prop_distlinear})
is given for the discrete MBP of the approximating operator.
\end{remark}

\begin{example}[{\em Quadrature-based difference operator for the nonlocal diffusion}]
\label{eg_nonlocal_dis}
\rm Recalling the nonlocal diffusion operator \eqref{eg_L_nonlocal_var},
we consider its quadrature-based finite difference discretization,
which is well known to be asymptotically compatible \cite{DuTaTiYa19,TiDu13}.
The dimension of $\Omega_c$ clearly depends on the value of $\delta$ in this case.
For a given {$w\in\hX$}, we define at $\bx_\bi\in\overline\Omega$ for Case (C1)
or $\bx_\bi\in\overline{\Omega}_+$ for Case (C2) via
\begin{equation}
\label{eg_L_nonlocal_discrete}
\hL_hw(\bx_\bi)
= \sum_{\bm{0}\not=\bs_\bj\in B_\delta(\bm{0})}
\frac{w(\bx_\bi+\bs_\bj)+w(\bx_\bi-\bs_\bj)-2w(\bx_\bi)}{|\bs_\bj|^2}|\bs_\bj|_1\beta_\delta(\bs_\bj),
\end{equation}
where the values of $w(\by)$ with $\by\in\Omega_c^*$
could be supplementally defined using the boundary conditions
so that the operator \eqref{eg_L_nonlocal_discrete} is well-defined.
Here, $|\cdot|_1$ stands for the vector $1$-norm, and
\begin{equation*}
\beta_\delta(\bs_\bj)=\frac{1}{2}\int_{B_\delta(\bm{0})}\psi_\bj(\bs)\frac{|\bs|^2}{|\bs|_1}\gamma_\delta(|\bs|)\,\d\bs,
\end{equation*}
where $\psi_\bj$ is the piecewise $d$-multilinear basis function
satisfying $\psi_\bj(\bs_\bi)=0$ when $\bi\not=\bj$ and $\psi_\bj(\bs_\bj)=1$.
{It is obvious that $\hL_h$ given by \eqref{eg_L_nonlocal_discrete} satisfies Assumption \ref{assump_L}-(a).}
By ordering the nodes lexicographically, we regard the function $w\in X$ as a vector in $\R^{M^d}$
and it is easy to check that the matrix $\hL_{h0}$, the restriction of $\hL_h$ on $X$,
is diagonally dominant in the sense of \eqref{assump_weakNDD}.
Thus, {Assumption \ref{assump_L}-(c)} is satisfied.
Moreover, $\hL_{h0}$ is a band matrix with the bandwidth depending on the value of $\delta$.
\end{example}

\begin{example}[{\em Finite difference operator for the fractional Laplacian}]
\rm For the fractional Laplace operator \eqref{eg_L_fracLaplace},
we consider its finite difference discretization in Case (C1),
and for simplicity the homogeneous Dirichlet boundary condition is enforced
so that the operators $\hL_h$ and $\hL_{h0}$ are essentially the same in this case.
For the one-dimensional case, a second-order finite difference discretization \cite{DuWyZh18,HuOb14}
(a limiting case of a similar scheme given in \cite{TiDu13} for nonlocal diffusion models) is
defined by:
\begin{equation}
\label{eg_L_frac_discrete}
\hL_h w(x_i) = \sum_{j=0}^{M_0} (\hL_{h0})_{ij} w(x_j), \quad x_i\in\overline\Omega,
\end{equation}
for any $w\in X$, where the entries of the matrix $\hL_{h0}$ are given by
\[\small
(\hL_{h0})_{ij} = \frac{c_{1,s}}{\nu h^{2s}}
\begin{dcases}
- (2^\nu+\lfloor\gamma\rfloor-1) - \sum_{k=2}^{M_0-1} \frac{(k+1)^\nu-(k-1)^\nu}{k^\gamma} \\
\qquad\qquad\qquad\qquad - \frac{M_0^\nu-(M_0-1)^\nu}{M_0^\gamma} - \frac{\nu}{s M_0^{2s}}, & j=i, \\
\frac{1}{2} (2^\nu+\lfloor\gamma\rfloor-1), & |j-i|=1, \\
\frac{(|j-i|+1)^\nu-(|j-i|-1)^\nu}{2|j-i|^\gamma}, & 2\le |j-i|\le M_0-1, \\
\frac{M_0^\nu-(M_0-1)^\nu}{2M_0^\gamma}, & |j-i|=M_0
\end{dcases}
\]
for $0\le i,j\le M_0$,
where $\nu=(1-s)\lfloor\gamma\rfloor$ with $\gamma=2$ or $\gamma=1+s$.
Obviously, the matrix $\hL_{h0}$ is diagonally dominant in the sense of \eqref{assump_weakNDD},
and thus Assumptions \ref{assump_L}-(a) and \ref{assump_L}-(c) are satisfied.
For the two- and three-dimensional cases,
the operators $\hL_h$ and $\hL_{h0}$ could be given in the similar way,
we refer to \cite{DuZh19,HuOb14} for details,
and it is easy to verify that Assumption \ref{assump_L} is still satisfied.
\end{example}

\begin{example}[{\em Finite element operator for the Laplacian}]
\label{eg_laplacian_fem}
\rm Let us consider Case (C1) again.
Let $\hT_h=\{K\}$ be a uniform partition of $\widetilde{\Omega}$
into isosceles right triangles $K$ non-overlapping with each other
based on the set of nodes $\overline{\Omega}$.
Let $V_h$ be the space of continuous and piecewise linear functions with respect to $\hT_h$
and $V_h^0$ be its subspace with zero-trace:
$$V_h=\{w\in C(\overline{\widetilde{\Omega}})\,|\,\text{$w|_K$ is linear for any $K\in\hT_h$}\},
\quad V_h^0=\{w\in V_h\,|\,w|_{\partial\widetilde{\Omega}}=0\}.$$
We denote by $N_I=(M_0-1)^d$ and $N_B=(M_0+1)^d-N_I$ the numbers
of the nodes in $\widetilde{\Omega}$ and on $\partial\widetilde{\Omega}$ respectively,
and order the nodes such that $\Omega=\{\bx_j\,|\,1\le j\le N_I\}$ consisting of all interior nodes
and $\partial\Omega=\{\bx_{j+N_I}\,|\,1\le j\le N_B\}$ all nodes on the boundary.
Denote by $\{\phi_j\,|\,1\le j\le N_I+N_B\}$ the basis functions of $V_h$
satisfying $\phi_i(\bx_j)=\delta_{ij}$ for any $1\le i,j\le N_I+N_B$.
Define the operator $\hL_h$ for $w\in\hX\simeq V_h$ by
\begin{equation}
\label{eg_fem_L}
\hL_h w(\bx_i) = \sum_{j=1}^{N_I+N_B} l_{ij} w(\bx_j), \quad \bx_i \in \Omega,
\end{equation}
where
\[
l_{ij} = - \frac{\int_{\widetilde{\Omega}}\nabla\phi_i\cdot\nabla\phi_j\,\d\bx}{\int_{\widetilde{\Omega}}\phi_i\,\d\bx},
\quad 1\le i\le N_I, \ 1\le j\le N_I+N_B.
\]
For each $i=1,2,\dots,N_I$, it is known that
\begin{equation}
\label{eg_fem_pf1}
\int_{\widetilde{\Omega}}\phi_i\,\d \bx>0,\quad\int_{\widetilde{\Omega}}|\nabla\phi_i|^2\,\d \bx>0,\quad
\text{and}\quad\int_{\widetilde{\Omega}}\nabla\phi_i\cdot\nabla\phi_j\,\d \bx\le0\ \text{for $j\not=i$}.
\end{equation}
Noticing that $\sum\limits_{j=1}^{N_I+N_B}\phi_j(\bx)\equiv1$ on $\overline{\widetilde{\Omega}}$,
we have that
\begin{equation}
\label{eg_fem_pf2}
\sum_{j=1}^{N_I+N_B}\int_{\widetilde{\Omega}}\nabla\phi_i\cdot\nabla\phi_j\,\d\bx
=\int_{\widetilde{\Omega}}\nabla\phi_i\cdot\nabla\bigg(\sum_{j=1}^{N_I+N_B}\phi_j\bigg)\,\d\bx=0, \quad 1\le i\le N_I.
\end{equation}
Combination of \eqref{eg_fem_pf1} and \eqref{eg_fem_pf2} yields
\begin{equation}
\label{eg_fem_pf3}
|l_{ii}|=\sum_{\substack{j=1\\ j\not=i}}^{N_I+N_B}|l_{ij}|,\quad
l_{ii}<0,\quad l_{ij}\ge0\ (j\not=i)
\end{equation}
for $1\le i\le N_I$ and $1\le j\le N_I+N_B$,
which implies Assumption \ref{assump_L}-(a).
The limitation of $\hL_h$ on $X\simeq V_h^0$ is given by
the matrix $\hL_{h0}\in\R^{N_I\times N_I}$ with $(\hL_{h0})_{ij}=l_{ij}$, $i,j=1,2,\dots,N_I$.
Then, by \eqref{eg_fem_pf3}, $\hL_{h0}$ satisfies \eqref{assump_weakNDD} and Assumption \ref{assump_L}-(c) holds.
\end{example}

\begin{remark}
For a general partition of $\widetilde{\Omega}$,
the FEM-based linear operator could be also defined by \eqref{eg_fem_L}
and still satisfies \eqref{assump_weakNDD}, see, e.g., \cite{NoZh18} for  discussions.
In fact, it is well-known in the literature that the discrete maximum principle holds
for piecewise linear finite element approximations of the Laplacian on general two-dimensional Delaunay triangulations
for which the circumscribing circle of any Delaunay triangle does not contain any other vertices in its interior,
see \cite{DuNiWu98}.
The brief note \cite{BuEr04} analyzes a stabilized Galerkin approximation of the Laplacian
guaranteeing the discrete maximum principle on arbitrary meshes.
In addition, one can also apply the finite volume method (FVM) \cite{LiChWu00} to obtain the discretized operator of the Laplacian.
The FVM-based linear operator possesses similar properties as those we have considered in the above.
\end{remark}

Applying Theorem \ref{thm_model_max} to the cases with
$f_0$ given by either Example \ref{eg_nonlinear_logistic} or \ref{eg_nonlinear_flory}
(or Example \ref{pr} when a transformation is first applied as explained in Remark \ref{rem1})
and $\hL=\hL_h$ determined by any one of Examples \ref{eg_laplacian_dis}--\ref{eg_laplacian_fem},
we obtain the following result on the MBP of \eqref{model_eq} in the space-discrete version.

\begin{corollary}
Let the nonlinear function $f_0$ be given by
either \eqref{eg_f_logistic} with even $p$ or \eqref{eg_f_flory}
and the linear operator $\hL=\hL_h$ defined by
any of \eqref{eg_local_diff}, \eqref{eg_L_nonlocal_discrete}, \eqref{eg_L_frac_discrete} and \eqref{eg_fem_L}.
If \eqref{cond_initial_BC} holds,
then the solution of the model problem \eqref{model_eq} exists
and satisfies $|u(t,\bx)|\le\beta$ for any $t\ge0$ and $\bx\in\widehat{\Omega}$,
where $\beta=1$ for the case of \eqref{eg_f_logistic} and $\beta=\rho$ for \eqref{eg_f_flory}, respectively.
\end{corollary}

\subsection{Discussion on homogeneous Neumann boundary conditions}
\label{sec:Neu}

In Section \ref{subsect_framework},
 the MBPs for the cases of Dirichlet and periodic boundary conditions are studied,
 let us now discuss the case of  homogeneous Neumann boundary condition.
For simplicity, we focus the analysis on the infinite dimensional operators.
The key ingredient is to find a Neumann operator $\hN_c$ such that the following Gaussian formula holds:
\begin{equation*}
\int_{\Omega^*} \hL u(\bx) \,\d\bx = \int_{\Omega_c^*} \hN_c u(\bx) \,\d\bx.
\end{equation*}
The homogeneous Neumann boundary condition then could be given by
\begin{equation}
\label{neu cont}
\hN_c u(\bx)=0, \quad \bx\in\Omega_c^*.
\end{equation}

For the second-order elliptic operator in the divergence form  \eqref{eg_L_elliptic_div},
using the classic Gaussian formula and assuming $\widetilde{\bm{q}}$ is divergence-free,
we obtain
\begin{equation*}
\int_\Omega \hL u(\bx) \,\d\bx
= \int_{\partial\Omega} (A(\bx)\nabla u(\bx) + \widetilde{\bm{q}}(\bx)u(\bx))\cdot\bn(\bx) \,\d\bx,
\end{equation*}
where $\bn$ denotes the unit outward normal vector to $\partial\Omega$.
Thus, the Neumann operator $\hN_c$ could be defined as
\begin{equation*}
\hN_c u(\bx)=A(\bx)\nabla u(\bx)\cdot\bn(\bx) + u(\bx)\widetilde{\bm{q}}(\bx)\cdot\bn(\bx), \quad \bx\in\partial\Omega,
\end{equation*}
and the corresponding constraint becomes the classic Robin boundary condition.
Alternatively, if we apply the classic Gaussian formula
only to the diffusion part $\hL_{\text{diff}} u(\bx) = \nabla\cdot(A(\bx)\nabla u(\bx))$,
then $\hN_c$ could be also defined by
\begin{equation}
\label{rmk_neumann}
\hN_c u(\bx)=A(\bx)\nabla u(\bx)\cdot\bn(\bx), \quad \bx\in\partial\Omega,
\end{equation}
and the corresponding constraint becomes the classic Neumann boundary condition.

For the nonlocal diffusion operator $\hL$ defined in \eqref{eg_L_nonlocal},
by the symmetry of the kernel, we have
\begin{align*}
& \int_{\widehat{\Omega}} \int_{\widehat{\Omega}}\gamma(\bx,\by)(u(\by)-u(\bx))\,\d \by \d\bx
= \int_{\widehat{\Omega}} \int_{\widehat{\Omega}}\gamma(\bx,\by)(u(\bx)-u(\by))\,\d \bx \d\by
= 0, \\
& \int_{\Omega_c^*} \int_{\Omega_c^*}\gamma(\bx,\by)(u(\by)-u(\bx))\,\d \by \d\bx
= \int_{\Omega_c^*} \int_{\Omega_c^*}\gamma(\bx,\by)(u(\bx)-u(\by))\,\d \bx \d\by
= 0.
\end{align*}
Then we have
$$\int_{\Omega^*} \hL u(\bx) \,\d\bx = \int_{\Omega_c^*} \hN_c u(\bx) \,\d\bx,$$
where $\hN_c$ is defined by
\begin{equation*}
\hN_c u(\bx) =  \int_\Omega \gamma(\bx,\by)(u(\bx)-u(\by))\,\d \by, \quad \bx\in\Omega_c^*.
\end{equation*}
This definition, different from those in \cite{DuGuLeZh12,DuGuLeZh13},
can be found in \cite{Du19} as a nonlocal extension of \eqref{rmk_neumann}.

We next consider the model problem \eqref{model_eq}
subject to the homogeneous Neumann boundary condition \eqref{neu cont}.
In order to show that the MBP is still satisfied for the above classic  and nonlocal  cases,
we need to generate a contraction semigroup as before
so that all the analysis conducted above could be similarly used.
To this end, we still  let $\hX=C(\overline{\Omega})$ if $\Omega_c=\partial\Omega$
and $\hX=C(\overline{\Omega})\cap C_b(\Omega_c^*)$ otherwise, and
 Assumption \ref{assump_L}-(a) is clearly true for  the above two operators $\hL$.
Then we need to define $X$ and $\hL_0$ appropriately and verify whether all conditions related to them in Remark \ref{proneu}
hold.

For the classic elliptic differential operator case, without loss of generality,
let us assume $A(\bx) \equiv I_d$ and $\bm{q}(\bx)\equiv\bm{0}$ in \eqref{eg_L_elliptic}.
For any $w\in C^2(\overline{\Omega})$, we can directly extend the domain of $\Delta w$
from $\Omega$ to $\overline\Omega$ by continuity so that $\hL=\Delta: C^2(\overline{\Omega})\to C(\overline{\Omega})$.
Let us define $X=C(\overline{\Omega})$,  {$D(\hL_0)=\{w\in C^2(\overline{\Omega})\,|\,(\nabla w\cdot\bn)|_{\partial\Omega}=0\}$}
and $\hL_0=\hL|_{D(\hL_0)}$,
then Assumptions \ref{assump_L}-(b) and \ref{assump_L}-(c) obviously hold.
For any $w\in D(\hL_0)$, assume that \eqref{L_cond_c} is true for some $\bx_0\in\partial\Omega$.
By  using the  boundary condition ($\nabla w \cdot\bn)(\bx_0)=0$ and the reflective extension in classic analysis,
we have that $\nabla w(\bx_0)=\bm{0}$ and the Hessian $\nabla^2w(\bx_0)$ is negative semi-definite,
and consequently it holds $\hL_0 w(\bx_0) = \Delta w(\bx_0)\le0$ by using an orthogonal transformation of the coordinates.
Thus, we can similarly show that Lemma \ref{lem_L_semigroup}-(i) is  valid
and consequently $\hL_0$ generates a contraction semigroup $\{S_{\hL_0}(t)\}_{t\ge0}$ on $X$ (i.e., Lemma \ref{lem_L_semigroup}-(ii)).

Next we turn to the nonlocal diffusion operator case, where $\hL$ is defined by \eqref{eg_L_nonlocal_var}
with the integrable kernel as discussed in the Example~\ref{eg_linear_nonlocal}.
For any $w\in C(\overline\Omega)$, there exists a unique $\widetilde{w}\in \hX$
satisfying $\widetilde{w}|_{\overline\Omega}=w$ and $(\hN_c \widetilde w)|_{\Omega_{\delta}}=0$;
in particular, we have
\begin{equation}
\label{nonneuext}
\widetilde w(\bx)
= \dfrac{\int_{\Omega\cap B_{\delta}(\bx)} \gamma_{\delta}(|\bx-\by|) w(\by) \,\d \by}{\int_{\Omega\cap B_{\delta}(\bx)} \gamma_{\delta}(|\bx-\by|) \,\d \by},
\quad \bx\in\Omega_{\delta}.
\end{equation}
By defining $X=\{w|_{\overline\Omega}\,|\,w\in\whX\}$ with $\whX=\{w\in \hX \,|\, (\hN_c w)|_{\Omega_{\delta}}=0\}$,
we can regard $X=C(\overline{\Omega})$ as a linear subspace of $\hX$
in the sense of isomorphism ($X\simeq \whX$) using such Neumann extension.
Assumption \ref{assump_L}-(b) automatically comes from  $D(\hL_0)=C(\overline{\Omega})$.
For any $\lambda>0$ and $f\in C(\overline{\Omega})$,
it is also not hard to verify that the equation $\lambda w-\hL_0 w=f$ in $\overline\Omega$
has a unique solution in $C(\overline{\Omega})$,
similar to the discussion in the earlier example for the zero Dirichlet boundary condition,
which implies Assumption \ref{assump_L}-(c).
For any $w\in C(\overline{\Omega})$, assume that \eqref{L_cond_c} holds for some $\bx_0\in\Omega_c^*$,
then we know by \eqref{nonneuext} that the maximum of $w$ on $\widehat{\Omega}$ must also be attained on $\overline\Omega$.
Thus, according to the proof of Lemma \ref{lem_L_semigroup}-(i),
we see that its conclusion still remains valid,
and consequently we obtain Lemma \ref{lem_L_semigroup}-(ii).

\section{Exponential time differencing for temporal approximation}
\label{sect_ETD_MBP}

In this section, we construct temporal approximations of the model equation \eqref{model_eq}
by using the exponential time differencing (ETD) method.
We will begin with the equivalent form \eqref{model_ODE}
to develop the MBP-preserving ETD schemes of first- and second-order,
followed by the convergence analysis.
For all discussions and theorems established in this section,
we only focus on Case (C1)
while the results are all valid for Case (C2)
by removing the equations containing the boundary term $g$ and using the periodicity of the solution in the proofs.

\subsection{ETD schemes and discrete MBPs}

Given a time step size $\dt>0$, we divide the time interval by $\{t_n=n\dt\}_{n\ge0}$.
To establish the ETD schemes for solving the model problem \eqref{model_eq},
we focus the equivalent equation \eqref{model_ODE}
on the interval $[t_n,t_{n+1}]$,
or equivalently, $w(s,\bx)=u(t_n+s,\bx)$ satisfying the following system
\begin{equation}
\label{model_ODE_solution2}
\begin{dcases}
w_s + \kp w = \hL w + \hN[w], & s\in(0,\dt], \ \bx\in\Omega^*, \\
w(s,\bx) = g(t_n+s,\bx), & s\in[0,\dt], \ \bx\in\Omega_c^*, \\
w(0,\bx) = u(t_n,\bx), & \bx\in\widehat{\Omega}.
\end{dcases}
\end{equation}
For the first-order (in time) approximation of \eqref{model_ODE_solution2},
we set $\hN[u(t_n+s)]\approx\hN[u(t_n)]$ to obtain the first-order ETD (ETD1) scheme:
for $n\ge 0$ and given $v^n$,
find $v^{n+1}=w^n(\dt)$ with  $w^n:[0,\dt]\to \hX$ solving
\begin{equation}
\label{model_ETD1}
\begin{dcases}
w^n_s + \kp w^n = \hL w^n + \hN[v^n], & s\in(0,\dt],  \ \bx\in\Omega^*,\\
w^n(s,\bx) = g(t_n+s,\bx), & s\in[0,\dt],\ \bx\in{\Omega_c^*}, \\
w^n(0,\bx) = v^n(\bx), & \bx\in\widehat{\Omega},
\end{dcases}
\end{equation}
where $v^n$ represents an approximation of $u(t_n)$ and $v^0=u_0$ is given.
Similar to \eqref{thm_mbp_pf1}, it is easy to show that $w^n$ is uniquely defined on $[0,\dt]\times\widehat{\Omega}$,
thus $v^{n+1}$ is well-defined.
We note that unlike the continuous-in-time case,
different choices of $\kp$ do lead to different discretization schemes.
Indeed, $\hL_0-\kp \hI$ serves as the generator of the semigroup $\{\e^{-\kp t}S_{\hL_0}(t)\}_{t\ge0}$
to control the nonlinear term and then to stabilize the time discretizations.
Thus, a suitable choice of $\kp$ becomes very important in our design of ETD schemes.

\begin{theorem}[Maximum bound principle of the ETD1 scheme]
\label{thm_ETD1}
Suppose that Assumptions \ref{assump_L} and \ref{assump_f}, \eqref{kappacon} and \eqref{cond_initial_BC} hold.
Then the ETD1 scheme preserves the discrete MBP unconditionally,
i.e., for any time step size $\dt>0$, the ETD1 solution satisfies $\|v^n\|\le\beta$ for any $n\ge0$.
\end{theorem}

\begin{proof}
Since $\|v^0\|\le\beta$,
we just need to show that $\|v^k\|\le\beta$ implies $\|v^{k+1}\|\le\beta$ for any $k$.
We have $v^{k+1}=w^k(\dt)$,
where $w^k$ satisfies \eqref{model_ETD1} with the superscript $n$ replaced by $k$.
Suppose there exists $(s^*,\bx^*)\in(0,\dt]\times\widehat{\Omega}$
such that
\begin{equation}
\label{thm_ETD1_pf}
w^k(s^*,\bx^*)=\max_{0\le s\le \dt}\|w^k(s)\|.
\end{equation}
Similar to the first part of the proof of Theorem \ref{thm_model_max},
we can obtain $\kp w^k(s^*,\bx^*)\le \hN[v^k](\bx^*)$,
and then, $w^k(s^*,\bx^*)\le\beta$.
The lower bound $-\beta$ could be obtained similarly.
Thus we obtain $\|w^k(s)\|\le\beta$ for any $s\in[0,\dt]$,
and thus $\|v^{k+1}\|\le\beta$,
which completes the proof.
\end{proof}

Next, we consider the higher-order approximations of the solution of \eqref{model_ODE_solution2}.
Let $P_r(s)$ be an interpolation of $\hN[u(t_n+s)]$
with degree $r\ge1$ on the times $\{s_k=\frac{k}{r}\dt\}_{k=0}^r$.
Then, we approximate $\hN[u(t_n+s)]$
by $P_r(s)$ to obtain the higher-order ETD Runge--Kutta scheme:
{find $v^{n+1}=w^n(\dt)$ with $w^n:[0,\dt]\to \hX$ solving
\begin{equation*}
\begin{dcases}
w^n_t + \kp w^n = \hL w^n + P_r(s), & s\in(0,\dt], \ \bx\in\Omega^*,\\
w^n(s,\bx) = g(t_n+s,\bx), & s\in[0,\dt],  \ \bx\in\Omega_c^*,\\
w^n(0,\bx) = v^n(\bx), &  \bx\in\widehat\Omega,
\end{dcases}
\end{equation*}
where $v^n$ is an approximation of $u(t_n)$ with $v^0=u_0$.}
More precisely, we have
\begin{equation*}
P_r(s)=\sum_{k=0}^r\ell_{r,k}(s)\hN[\tilde{v}^{n+\frac{k}{r}}],
\quad s\in[0,\dt],
\end{equation*}
where $\{\ell_{r,k}(s)\}_{k=0}^r$ are the standard Lagrange basis functions associated with the times $\{s_k\}_{k=0}^r$,
and $\tilde{v}^{n+\frac{k}{r}}$ is an approximated value of $u(t_n+s_k)$, generated by the lower-order ETD schemes and $\tilde{v}^{n}=v^n$.
In the spirit of the proof of Theorem \ref{thm_ETD1},
we know that the discrete MBP would be preserved as long as the following condition is satisfied:
\begin{equation}
\label{interpolation_condition}
\|P_r(s)\|\le\max\{\|\hN[\tilde{v}^{n+\frac{k}{r}}]\|:0\le k\le r\},\quad\forall\,s\in[0,\dt],
\end{equation}
with $\|\tilde{v}^{n+\frac{k}{r}}\|\le\beta$ for all $k=0,1,\dots,r$.
However, the unique interpolation satisfying \eqref{interpolation_condition} corresponds to the case $r=1$,
that is, the linear interpolation
\begin{equation*}
P_1(s)=\Big(1-\frac{s}{\dt}\Big)\hN[\tilde{v}^n]+\frac{s}{\dt}\hN[\tilde{v}^{n+1}],
\quad s\in[0,\dt].
\end{equation*}

Approximating $\hN[u(t_n+s)]$ in \eqref{model_ODE_solution2} by $P_1(s)$,
we obtain the second-order ETD Runge--Kutta (ETDRK2) scheme:
find $v^{n+1}=w^n(\dt)$ with $w^n:[0,\dt]\to \hX$ solving
\begin{equation}
\label{model_ETDRK2}
\begin{dcases}
w^n_t + \kp w^n = \hL w^n + \Big(1-\frac{s}{\dt}\Big)\hN[v^n]+\frac{s}{\dt}\hN[\tilde{v}^{n+1}], & s\in(0,\dt], \ \bx\in\Omega^*,\\
w^n(s,\bx) = g(t_n+s,\bx), & s\in[0,\dt],  \ \bx\in\Omega_c^*,\\
w^n(0,\bx) = v^n(\bx), &  \bx\in\widehat\Omega,
\end{dcases}
\end{equation}
where $\tilde{v}^{n+1}$ is generated by the ETD1 scheme.
It is worth noting that both ETD1 and ETDRK2 schemes are linear.

\begin{theorem}[Maximum bound principle of the ETDRK2 scheme]
\label{thm_ETDRK2}
Suppose that Assumptions \ref{assump_L} and \ref{assump_f}, \eqref{kappacon} and \eqref{cond_initial_BC} hold.
Then the ETDRK2 scheme preserves the discrete MBP unconditionally, i.e., for any time step size $\dt>0$,
the ETDRK2 solution satisfies $\|v^n\|\le\beta$ for any $n\ge0$.
\end{theorem}

\begin{proof}
Again, we suppose $\|v^k\|\le\beta$ for some $k$.
Similar to the proof of Theorem \ref{thm_ETD1},
for $(s^{*},\bx^{*})\in(0,\dt]\times\Omega^*$ such that \eqref{thm_ETD1_pf} holds,
we can obtain
\begin{equation*}
\kp w^k(s^{*},\bx^{*})\le \Big(1-\frac{s^*}{\dt}\Big)\hN[v^k](\bx^*) + \frac{s^*}{\dt}\hN[\tilde{v}^{k+1}](\bx^*),
\end{equation*}
{and then, $w^k(s^{*},\bx^{*})\le\beta$ using $\|\tilde{v}^{k+1}\|\le\beta$ by Theorem \ref{thm_ETD1}.}
The lower bound $-\beta$ could be obtained similarly.
Therefore, $\|v^{k+1}\|\le\beta$,
which completes the proof.
\end{proof}

As an application of Theorems \ref{thm_ETD1} and \ref{thm_ETDRK2},
one can obtain the MBPs of the semi-discrete and fully discrete ETD schemes
for the concrete examples presented in Sections  \ref{nonope} and \ref{linope}.

\begin{corollary}
For the space-continuous ETD1 and ETDRK2 schemes with $\hL$ given by
any of  \eqref{eg_L_elliptic}, \eqref{eg_L_nonlocal_var}, and \eqref{eg_L_fracLaplace},
as well as the fully discrete ETD1 and ETDRK2 schemes with the (discretized) linear operator $\hL=\hL_h$ defined by
any of \eqref{eg_local_diff}, \eqref{eg_L_nonlocal_discrete}, \eqref{eg_L_frac_discrete}, or \eqref{eg_fem_L}, it holds that

{\rm(i)} if $f_0$ is given by \eqref{eg_f_logistic} with even $p$,
$\kp\ge\lambda p$, and \eqref{cond_initial_BC} is satisfied with $\beta=1$,
the solution to the ETD1 scheme \eqref{model_ETD1} or the ETDRK2 scheme \eqref{model_ETDRK2}
satisfies $\|v^n\|\le1$ for any $n\ge0$;

{\rm(ii)} if $f_0$ is given by \eqref{eg_f_flory},
$\kp\ge\frac{\theta}{1-\rho^2}-\theta_c$,
and \eqref{cond_initial_BC} is satisfied with $\beta=\rho$,
the solution to the ETD1 scheme \eqref{model_ETD1} or the ETDRK2 scheme \eqref{model_ETDRK2}
satisfies $\|v^n\|\le\rho$ for any $n\ge0$,
where $\rho$ is the positive root of \eqref{eg_f_flory_gamma}.
\end{corollary}

In summary, we have proved that the first- and second-order ETDRK schemes
preserve the MBP of the underlying problem \eqref{model_eq} unconditionally in the discrete sense.
Meanwhile, we actually show that
\emph{the classic ETDRK approximations, as defined in \cite{CoMa02}, with order greater than two
fail to maintain the  MBP unconditionally}
due to the lack of the property \eqref{interpolation_condition} for the interpolation polynomials with $r\ge2$.
Apart from the ETDRK schemes,
multistep approximations could be also utilized in the ETD temporal discretization \cite{HoOs10}.
The ETD-multistep scheme of order $r+1$ is generated from
the extrapolation polynomial of degree $r\ge1$ for the nonlinear term
based on the approximated solutions at the previous $r$ time levels.
Since the extrapolation polynomials with $r\ge1$ cannot be bounded
by the maximums and minimums of the extrapolated data,
\emph{the ETD-multistep schemes with even order greater than one also fail to preserve the discrete MBP unconditionally}.

Up till now, we only presented the differential forms of the ETD1 and ETDRK2 schemes
by \eqref{model_ETD1} and \eqref{model_ETDRK2}
since they are convenient for the theoretical analysis.
In practice,
we need  formulas that can be more directly implemented for computations,
in particular, the fully discrete schemes (in both time and space)
with $X\simeq\R^{M^d}$  and $\hL=\hL_h$ discussed in Section \ref{sect_finitedim}.
Here we define a new operator $\hL_{hc}:\partial X\to X$ as follows:
for any $g\in\partial X$ and $\widehat{w}\in D(\hL)$ with $\widehat{w}|_{\Omega_c^*}=g$, let
\begin{equation}
\label{def_boundaryopt}
\hL_{hc}g = (\hL_h \widehat{w})|_{\Omega^*} - \hL_{h0} (\widehat{w}|_{\Omega^*}).
\end{equation}
It is easy to check that the right-hand side of \eqref{def_boundaryopt} depends only on $g$,
so $ \hL_{hc}$ is well-defined through \eqref{def_boundaryopt}.

For a given boundary data $g:[0,T]\to\partial X$ and a function $u:[0,T]\to \hX$ such that $u(t)|_{\Omega_c^*}=g(t)$,
we obtain from \eqref{model_ODE_solution2} and the definition of $\hL_{hc}$ that
\begin{equation}
\label{model_ODE_disspace}
\begin{dcases}
u_t + \kp u = \hL_{h0} u + \hN[u] + \hL_{hc}g, & t\in(0,T], \ \bx\in\Omega^*, \\
u(0,\bx) = u_0(\bx), & \bx\in\widehat{\Omega},
\end{dcases}
\end{equation}
where $\hL_{h0}$ is an $M^d$-by-$M^d$ matrix.
Let
\begin{equation}
\label{def_Lkp}
\hL_{\kp,h} = \kp I_{M^d} - \hL_{h0}
\end{equation}
and define the $\varphi$-functions as follows:
\begin{equation*}
\varphi_0(a)=\e^{-a},\quad
\varphi_1(a)=\frac{1-\e^{-a}}{a},\quad
\varphi_2(a)=\frac{a-1+\e^{-a}}{a^2},\qquad a\not=0.
\end{equation*}
Let $v^0=u_0$.
Then, we give the fully discrete ETD1 scheme by
\begin{equation}
\label{fullydis_ETD1}
v^{n+1}
= \e^{-\dt\hL_{\kp,h}} v^n + \int_0^\dt \e^{-(\dt-s)\hL_{\kp,h}} \big\{\hN[v^n]+\hL_{hc}g(t_n+s)\big\}\,\d s
\end{equation}
or equivalently,
\begin{equation*}
v^{n+1}
= \varphi_0(\dt\hL_{\kp,h} ) v^n + \dt \varphi_1(\dt\hL_{\kp,h} )\hN[v^n]
+ \int_0^\dt \e^{-(\dt-s)\hL_{\kp,h}} \hL_{hc} g(t_n+s)\,\d s,
\end{equation*}
and the fully discrete ETDRK2 scheme by
\begin{equation*}
\begin{dcases}
\tilde{v}^{n+1} = \e^{-\dt\hL_{\kp,h}} v^n + \int_0^\dt \e^{-(\dt-s)\hL_{\kp,h}} \big\{\hN[v^n]+\hL_{hc}g(t_n+s)\big\}\,\d s,\\
v^{n+1} = \e^{-\dt\hL_{\kp,h}} v^n + \int_0^\dt \e^{-(\dt-s)\hL_{\kp,h}}
\Big\{\Big(1-\frac{s}{\dt}\Big)\hN[v^n] +\frac{s}{\dt}\hN[\tilde{v}^{n+1}]+\hL_{hc}g(t_n+s)\Big\}\,\d s
\end{dcases}
\end{equation*}
or equivalently,
\begin{equation}
\label{fullydis_ETDRK2_phi}
\begin{dcases}
\tilde{v}^{n+1}
=\varphi_0(\dt\hL_{\kp,h} )v^n+\dt \varphi_1(\dt\hL_{\kp,h} )\hN[v^n]
+ \int_0^\dt \e^{-(\dt-s)\hL_{\kp,h}} \hL_{hc} g(t_n+s)\,\d s,\\
v^{n+1}
=\tilde{v}^{n+1}+\dt\varphi_2(\dt\hL_{\kp,h} )
\big\{\hN[\tilde{v}^{n+1}]-\hN[v^n]\big\}.
\end{dcases}
\end{equation}
Note that the corresponding semigroup is given by the matrix exponential $S_{-\hL_{\kp,h}}(t)=\e^{t(\hL_{h0}-\kp I)}$,
which again depends crucially on the choice of the constant $\kp$.

\subsection{Convergence analysis of the ETD schemes}

As an important application of the MBP,
we now consider the convergence of the ETD1 and ETDRK2 schemes.
From the practical point of view, we are mainly interested in the convergence
of the fully discrete solution to the semi-discrete (in space) solution (with a fixed spatial mesh size)
as the time step size goes to zero.
Again, we have $\hL=\hL_h$ and $S_{\hL_{h0}}(t)=\e^{t\hL_{h0}}$
as we discussed in Section \ref{sect_finitedim}.
First, let us state the result for the ETD1 scheme \eqref{model_ETD1}.

\begin{theorem}
Suppose that Assumptions \ref{assump_L} and \ref{assump_f}, \eqref{kappacon} and \eqref{cond_initial_BC} are satisfied.
For the  fixed terminal time $T>0$ and spatial mesh size $h>0$,
assume that the exact solution $u$ to the space-discrete model equation \eqref{model_ODE_disspace} belongs to $C^1([0,T];\hX)$
and $\{v^n\}_{n\ge0}$ is generated by the fully discrete ETD1 scheme \eqref{fullydis_ETD1} with $v^0=u_0$.
Then we have
\begin{equation}
\label{ETD1error}
\|v^n-u(t_n)\|\le C\e^{\kp t_n}\dt,\quad\forall\,t_n\le T
\end{equation}
for any $\dt>0$,
where the constant $C>0$ depends on the $C^1([0,T];\hX)$ norm of $u$,
but independent of $\dt$ and $\kp$.
\end{theorem}

\begin{proof}
We know that the ETD1 solution is given by $v^{n+1}=w^n(\dt)$ with the function $w^n:[0,\dt]\to \hX$ solving
\begin{equation}
\label{model_ETD1_eq}
\begin{dcases}
w^n_s + \kp w^n = \hL_{h0} w^n + \hN[v^n] + \hL_{hc} g(t_n+s), & s\in(0,\dt], \ \bx\in\Omega^*,\\
w^n(0,\bx) = v^n(\bx), &  \bx\in\widehat{\Omega}.
\end{dcases}
\end{equation}
Let $e_1^n=v^n-u(t_n)$.
The difference between \eqref{model_ETD1_eq} and \eqref{model_ODE_disspace} yields
\begin{equation}
\label{ETD1error_eq}
e_1^{n+1}=\e^{-\kp\dt}\e^{\dt\hL_{h0}}e_1^n
+\int_0^\dt\e^{-\kp(\dt-s)}\e^{(\dt-s)\hL_{h0}}\{\hN[v^n]-\hN[u(t_n)]+R_1(s)\}\,\d s,
\end{equation}
where $R_1(s)$ is the truncated error as
$$R_1(s)=\hN[u(t_n)]-\hN[u(t_n+s)],\quad s\in[0,\dt].$$
Due to the MBP of \eqref{model_eq}, we have $\|u(t_n)\|\le\beta$ and $\|u(t_n+s)\|\le\beta$,
and then, using Lemma \ref{lem_nonlinear}-(ii), we derive
$$\|R_1(s)\| = \|N_0(u(t_n))-N_0(u(t_n+s))\|\le2\kp\|u(t_n)-u(t_n+s)\|\le C_1\kp\dt,\quad\forall\,s\in[0,\dt],$$
where the constant $C_1$ depends on the $C^1([0,T];\hX)$ norm of $u$, but independent of $\dt$ and $\kp$.
Similarly, since $\|v^n\|\le\beta$, we can obtain
\begin{equation}
\label{ETD1error_pf1}
\|\hN[v^n]-\hN[u(t_n)]\|\le2\kp\|v^n-u(t_n)\|=2\kp\|e_1^n\|.
\end{equation}
Then, we derive from \eqref{ETD1error_eq} that
\begin{align}
\|e_1^{n+1}\|
& \le\e^{-\kp\dt}\opnorm{\e^{\dt\hL_{h0}}}\|e_1^n\|\nonumber\\
& \quad +\int_0^\dt\e^{-\kp(\dt-s)}\opnorm{\e^{(\dt-s)\hL_{h0}}}\big\{\|\hN[v^n]-\hN[u(t_n)]\|+\|R_1(s)\|\big\}\,\d s\nonumber\\
& \le\e^{-\kp\dt}\|e_1^n\|+(2\kp\|e_1^n\|+C_1\kp\dt)\int_0^\dt\e^{-\kp(\dt-s)}\,\d s\nonumber\\
& =\e^{-\kp\dt}\|e_1^n\|+\frac{1-\e^{-\kp\dt}}{\kp}(2\kp\|e_1^n\|+C_1\kp\dt)\nonumber\\
& =(2-\e^{\kp\dt})\|e_1^n\|+\frac{1-\e^{-\kp\dt}}{\kp\dt}\cdot C_1\kp\dt^2\nonumber\\
& \le(1+\kp\dt)\|e_1^n\|+C_1\kp\dt^2,\label{ETD1error_pf2}
\end{align}
where in the last step we have used the fact that $1-\e^{-a}\le a$ for any $a>0$.
By induction, we have
\begin{align*}
\|e_1^n\|
& \le(1+\kp\dt)^n\|e_1^0\|+C_1\kp\dt^2\sum_{k=0}^{n-1}(1+\kp\dt)^k\\
& =(1+\kp\dt)^n\|e_1^0\|+C_1\dt[(1+\kp\tau)^n-1]\\
& \le\e^{\kp n\dt}\|e_1^0\|+C_1\e^{\kp n\dt}\dt.
\end{align*}
By letting $C=C_1$, we obtain \eqref{ETD1error} since $e_1^0=0$ and $n\dt=t_n$.
\end{proof}

\begin{remark}
We note that in the estimate \eqref{ETD1error},
there is an exponential coefficient $\e^{\kp t_n}$
which could be very large
(e.g., in the case of Allen--Cahn equation, $f$ may include a negative power of a small parameter,
which leads to large $\kp$).
The similar result is also obtained for the ETDRK2 scheme as shown later.
Theoretically, this is inevitable due to the application of the Gronwall's lemma.
On the other hand, one may be able to further improve the error estimate by using the techniques in \cite{FePr03,FePr04}.
\end{remark}

Now, we turn to the convergence of the ETDRK2 scheme \eqref{model_ETDRK2}.
Assume that $f_0$ is twice continuously differentiable and denote by $M^f_2$ the maximum of $|f_0''|$ on $[-\beta,\beta]$.

\begin{theorem}
Suppose that Assumptions \ref{assump_L} and \ref{assump_f}, \eqref{kappacon} and \eqref{cond_initial_BC} are satisfied.
For the fixed terminal time $T>0$ and spatial mesh size $h>0$,
assume that the exact solution $u$ to the space-discrete model equation \eqref{model_ODE_disspace} belongs to $C^2([0,T];\hX)$
and $\{v^n\}_{n\ge0}$ is generated by the fully discrete ETDRK2 scheme \eqref{fullydis_ETDRK2_phi} with $v^0=u_0$.
Then we have
\begin{equation}
\label{ETDRK2error}
\|v^n-u(t_n)\|\le C_\kp\e^{\kp t_n}\dt^2,\quad\forall\,t_n\le T
\end{equation}
for any $\dt>0$,
where the constant $C_\kp>0$ depends on $\kp$, $M^f_2$, and the $C^2([0,T];\hX)$ norm of $u$,
but independent of $\dt$.
\end{theorem}

\begin{proof}
The proof strategy is quite similar to the case of the first-order scheme.
Let $e_2^n=v^n-u(t_n)$, then we have
\begin{align}
e_2^{n+1} & =\e^{-\kp\dt}\e^{\dt\hL_{h0}}e_2^n+\int_0^\dt\e^{-\kp(\dt-s)}\e^{(\dt-s)\hL_{h0}}
\Big\{\Big(1-\frac{s}{\dt}\Big)(\hN[v^n]-\hN[u(t_n)])\nonumber\\
& \quad +\frac{s}{\dt}(\hN[\tilde{v}^{n+1}]-\hN[u(t_{n+1})])+R_2(s)\Big\}\,\d s,\label{ETDRK2error_eq}
\end{align}
where $R_2(s)$ is the truncated error given by
\begin{equation*}
R_2(s)=\Big(1-\frac{s}{\dt}\Big)\hN[u(t_n)]+\frac{s}{\dt}\hN[u(t_{n+1})]-\hN[u(t_n+s)],
\quad s\in[0,\dt].
\end{equation*}
Using the MBP and the error estimates of the linear interpolation, we have
$$\|R_2(s)\|\le(C_2\kp+C_3)\dt^2,\quad\forall\,s\in[0,\dt],$$
where the constant $C_2$ depends on the $C^2([0,T];\hX)$ norm of $u$ and
$C_3$ depends on $M^f_2$ and the $C^2([0,T];\hX)$ norm of $u$;
both of them are independent of $\dt$ and $\kp$.
From the last inequality in \eqref{ETD1error_pf2}, we know
$$\|\tilde{v}^{n+1}-u(t_{n+1})\|\le(1+\kp\dt)\|v^n-u(t_n)\|+C_1\kp\dt^2,$$
and then, using Lemma \ref{lem_nonlinear}-(ii), we obtain
$$\|\hN[\tilde{v}^{n+1}]-\hN[u(t_{n+1})]\|
\le 2\kp\|\tilde{v}^{n+1}-u(t_{n+1})\|
\le 2\kp(1+\kp\dt)\|e_2^n\|+2C_1\kp^2\dt^2.$$
By combining with \eqref{ETD1error_pf1}, we have, for any $s\in[0,\dt]$,
\begin{align*}
& \Big\|\Big(1-\frac{s}{\dt}\Big)(\hN[v^n]-\hN[u(t_n)])+\frac{s}{\dt}(\hN[\tilde{v}^{n+1}]-\hN[u(t_{n+1})])\Big\|\\
& \qquad \le\Big(1-\frac{s}{\dt}\Big)2\kp\|e_2^n\|+\frac{s}{\dt}[2\kp(1+\kp\dt)\|e_2^n\|+2C_1\kp^2\dt^2]\\
& \qquad =2\kp(1+\kp s)\|e_2^n\|+2C_1\kp^2\dt s.
\end{align*}
Then, we obtain from \eqref{ETDRK2error_eq} that
\begin{align*}
\|e_2^{n+1}\|
& \le\e^{-\kp\dt}\|e_2^n\|
 +\int_0^\dt\e^{-\kp(\dt-s)}[2\kp(1+\kp s)\|e_2^n\|+2C_1\kp^2\dt s+(C_2\kp+C_3)\dt^2]\,\d s\\
& =\e^{-\kp\dt}\|e_2^n\|+[2\kp\|e_2^n\|+(C_2\kp+C_3)\dt^2]\int_0^\dt\e^{-\kp(\dt-s)}\,\d s\\
& \quad +(2\kp^2\|e_2^n\|+2C_1\kp^2\dt)\int_0^\dt s\e^{-\kp(\dt-s)}\,\d s\\
& =\e^{-\kp\dt}\|e_2^n\|+\frac{1-\e^{-\kp\dt}}{\kp}[2\kp\|e_2^n\|+(C_2\kp+C_3)\dt^2]\\
& \quad +\frac{\e^{-\kp\dt}-1+\kp\dt}{\kp^2}(2\kp^2\|e_2^n\|+2C_1\kp^2\dt)\\
& =(\e^{-\kp\dt}+2\kp\dt)\|e_2^n\|+\frac{1-\e^{-\kp\dt}}{\kp}(C_2\kp+C_3)\dt^2
+\frac{\e^{-\kp\dt}-1+\kp\dt}{\kp^2}\cdot 2C_1\kp^2\dt\\
& \le\Big(1+\kp\dt+\frac{(\kp\dt)^2}{2}\Big)\|e_2^n\|+(C_2\kp+C_3+C_1\kp^2)\dt^3,
\end{align*}
where we have used the inequality $1-a\le\e^{-a}\le1-a+\frac{a^2}{2}$ for any $a>0$.
Finally, by induction, we obtain
\begin{align*}
\|e_2^n\|
& \le\Big(1+\kp\dt+\frac{(\kp\dt)^2}{2}\Big)^n\|e_2^0\|
+(C_1\kp^2+C_2\kp+C_3)\dt^3\sum_{k=0}^{n-1}\Big(1+\kp\dt+\frac{(\kp\dt)^2}{2}\Big)^k\\
& \le\Big(1+\kp\dt+\frac{(\kp\dt)^2}{2}\Big)^n\|e_2^0\|
+\Big(C_1\kp+C_2+\frac{C_3}{\kp}\Big)\dt^2\Big[\Big(1+\kp\dt+\frac{(\kp\dt)^2}{2}\Big)^n-1\Big]\\
& \le\e^{\kp n\dt}\|e_2^0\|+\Big(C_1\kp+C_2+\frac{C_3}{\kp}\Big)\e^{\kp n\dt}\dt^2.
\end{align*}
By letting $C_\kp=C_1\kp+C_2+C_3/\kp$, we obtain \eqref{ETDRK2error}.
\end{proof}

By specifying the linear operator by any one of Examples \ref{eg_laplacian_dis}--\ref{eg_laplacian_fem}
and the nonlinear one by either Example \ref{eg_nonlinear_logistic}, \ref{eg_nonlinear_flory}, or \ref{pr},
one can obtain the convergence results of corresponding concrete problems.

\subsection{Energy stability of the ETD schemes for gradient flow models}

Next we show the application of the MBP and convergence of the ETD schemes
to gradient flow models, a class of important examples of the model equation \eqref{model_eq}.
We only consider  the periodic or {\em time-independent} Dirichlet boundary condition
and also assume that the linear operator $\hL$ is dissipative on $L^2(\Omega)$ in the sense that
the inner product $(w,\hL w)_{L^2(\Omega)}\le0$ for any $w\in D(\hL)$,
which can be satisfied by a large number of operators
such as those given in Examples \ref{eg_linear_elliptic}--\ref{eg_laplacian_fem}.

Phase-field (or diffuse-interface) models are usually derived as the gradient flows with respect to the energy functional
\begin{equation*}
E[u]=-\frac{1}{2}(u,\hL u)_{L^2(\Omega)}+\int_\Omega F(u(\bx))\,\d\bx
\end{equation*}
with $F:\R\to\R$ such that $f_0=-F'$.
Some simple calculations give us the energy law under the periodic or time-independent Dirichlet boundary condition:
\begin{equation*}
E[u(t_2)]\le E[u(t_1)],\quad\forall\,t_2\ge t_1\ge0.
\end{equation*}
There have been numerous researches devoted to the energy stable numerical methods for the phase-field models,
see, e.g., \cite{JuLiQiZh18,LiQiTa16,ShYa10,WaWiLo09} and the references in the recent review \cite{DuFe19}.

An interesting problem is whether the energy law can be preserved by the proposed ETD discretization schemes.
In our recent work \cite{DuJuLiQi19},
we investigated the MBP-preserving ETD schemes for the nonlocal Allen--Cahn equation with periodic boundary condition,
namely, the equation \eqref{model_eq} with the linear operator given by \eqref{eg_L_nonlocal_var}
and nonlinear function defined as \eqref{eg_f_quartic}.
We concluded that the solution to the ETD1 scheme decreases the energy along the time steps
and the energy for the ETDRK2 scheme is uniformly bounded by the initial energy plus a constant.

For the more general case we consider in this work,
the ETD1 and ETDRK2 schemes for \eqref{model_eq} still satisfy the energy stability.
Since the proof could be conducted in a quite similar way as done in \cite{DuJuLiQi19},
we state the result directly as follows and leave the details of the proof to interested readers.

\begin{proposition}
Suppose that Assumptions \ref{assump_L}--\ref{assump_f}, \eqref{kappacon} and \eqref{cond_initial_BC} hold.
Then

{\rm(i)} the solution $\{v^n\}_{n\ge0}$ to the ETD1 scheme \eqref{model_ETD1} satisfies
$$E[v^{n+1}]\le E[v^n],\quad \forall\, n\ge0,$$
for any $\dt>0$, i.e., the ETD1 scheme is unconditionally energy stable;

{\rm(ii)}
the solution $\{v^n\}_{n\ge0}$ to the ETDRK2 scheme \eqref{model_ETDRK2} satisfies
$$E[v^n]\le E[v^0]+\widehat{C},\quad \forall\,t_n\le T,$$
for any $\dt\in(0,1]$,
where the constant $\widehat{C}=\widehat{C}(|\Omega|,T,\kp)$ is independent of $\dt$,
i.e., the energy is uniformly bounded.
\end{proposition}

\section{Some extensions}
\label{sect_extension}

In the framework we have established,
although the function spaces $\hX$ and $X$ consists of only real scalar-valued functions,
the main results on the MBP could be extended to some complex scalar-valued equations,
or more generally, real vector-valued ones.
In addition, the equations on real matrix-valued fields also share some similar characteristics.
The MBP for the case of systems implies
the existence of the invariant regions of the solution \cite{Amann78,Kuiper80,ReWa78,Smoller94}.
One can find studies on the invariant-region-preserving numerical methods
for classic reaction-diffusion systems (see, e.g., \cite{FrMaSgVe19,MaMo84})
and hyperbolic systems of conservation laws (see, e.g., \cite{GuPoSa20,GuPoTo19,JiLi18}).
For simplicity, we focus our analysis on the space-continuous setting (i.e., Case (D1)) in this section.

\subsection{Extension to complex scalar-valued and real vector-valued equations}

The Ginzburg--Landau model for superconductivity \cite{Du94,DuGuPe92}
is one of the popular models describing the phase transition of the superconducting material
under effects of magnetic and electric fields.
For simplicity, we only consider the equation with respect to the order parameter without electric effect,
that is, the electric potential vanishing in the Ginzburg--Landau model.
Let $\phi:[0,T]\times\overline{\Omega}\to\C$ be a complex-valued order parameter satisfying
\begin{equation}
\label{complex_GinzburgLandau}
\phi_t=(\nabla+\mathrm{i}\bm{A})^2\phi+(1-|\phi|^2)\phi,\quad t\in(0,T],\ \bx\in\Omega,
\end{equation}
subject to the initial  condition
\begin{equation*}
\phi(0,\cdot)=\phi_0,\quad\bx\in \overline{\Omega}
\end{equation*}
and either the Dirichlet, periodic, or homogeneous natural boundary condition
\[
(\nabla + \mathrm{i}\bm{A})\phi \cdot \bn = 0,\quad t\in(0,T],\ \bx\in\partial\Omega,
\]
where $\bm{A}\in\R^d$ is a given magnetic potential and $|\phi|$ denotes the modulus of $\phi$.
The MBP of the solution to \eqref{complex_GinzburgLandau} with the natural boundary condition was
proved in \cite{Du94} in the sense of  weak solution as follows.

\begin{proposition}
\label{prop_MBP_GL}
If $|\phi_0(\bx)|\le1$ for a.e. $\bx\in\overline{\Omega}$,
then it holds $|\phi(t,\bx)|\le1$ for a.e. $t\in[0,T]$ and $\bx\in\overline{\Omega}$.
\end{proposition}

If we let $\psi=\e^{\mathrm{i}\bm{A}\cdot\bx}\phi$,
simple calculations give us
$\Delta\psi=\e^{\mathrm{i}\bm{A}\cdot\bx}(\nabla+\mathrm{i}\bm{A})^2\phi$ and $|\psi|=|\phi|$.
Thus, the equation \eqref{complex_GinzburgLandau} is equivalent to
\begin{equation}
\label{complex_AllenCahn}
\psi_t=\Delta\psi+(1-|\psi|^2)\psi,\quad t\in(0,T],\ \bx\in\Omega,
\end{equation}
subject to the initial  condition $\psi(0,\bx)=\e^{\mathrm{i}\bm{A}\cdot\bx}\phi_0(\bx)$
and corresponding boundary conditions.
Noting that $\psi$ and $\phi$ have the same modulus,
the MBP (Proposition \ref{prop_MBP_GL}) is also valid for \eqref{complex_AllenCahn}.
Since a complex number can be viewed as an element in $\R^2$ in the sense of isomorphism,
the complex-valued equation \eqref{complex_AllenCahn} is actually equivalent to a real vector-valued equation
\begin{equation}
\label{vector_AllenCahn}
\bu_t=\Delta\bu+(1-|\bu|^2)\bu,\quad t\in(0,T],\ \bx\in\Omega
\end{equation}
with $|\cdot|$ denoting the standard Euclidean norm,
where $\bu:[0,T]\times\overline{\Omega}\to\R^m$ ($m=2$ for the Ginzburg--Landau model)
is subject to the initial  condition
\begin{equation*}
\bu(0,\cdot)=\bu_0,\quad\bx\in\overline{\Omega}
\end{equation*}
and either the Dirichlet boundary condition
\begin{equation*}
\bu = \bg, \quad t\in[0,T],\ \bx\in\partial\Omega
\end{equation*}
with $\bg\in C([0,T]\times\partial\Omega;\R^m)$,
the periodic boundary condition, or
the homogeneous Neumann boundary condition
\[
\nabla\bu\cdot\bn=\bm{0},\quad t\in[0,T],\ \bx\in\partial\Omega.
\]

Similar to the Allen--Cahn equation,
the vector-valued equation \eqref{vector_AllenCahn}
could be also regarded as the $L^2$ gradient flow with respect to the energy functional
\begin{equation}
\label{energy_vector}
E[\bu] = \int_\Omega \Big( \frac{1}{2}|\nabla\bu|_F^2 + \frac{1}{4}(|\bu|^2-1)^2\Big) \,\d\bx,
\end{equation}
where $|\nabla\bu|_F$ denotes the Frobenius norm of the Jacobian matrix $\nabla\bu$.
The solution to \eqref{vector_AllenCahn} decreases the energy \eqref{energy_vector} along with the time
under either the time-independent Dirichlet, the homogeneous Neumann, or the periodic boundary condition.

Introducing the stabilizing constant $\kp>0$ as before,
the equation \eqref{vector_AllenCahn} is then equivalent to
\begin{equation*}
\bu_t+\kp\bu=\Delta\bu+\bN_0(\bu),\quad t\in(0,T],\ \bx\in\Omega,
\end{equation*}
where $\bN_0(\bxi)=\kp\bxi+\bm{f}_0(\bxi)$ and
$\bm{f}_0(\bxi)=(1-|\bxi|^2)\bxi$
is the vector-valued analogue of the scalar function \eqref{eg_f_quartic}.
Corresponding to \eqref{kappacon} and Lemma \ref{lem_nonlinear},
we choose $\kp\ge2$ and then have the following lemma on the vector-valued function $\bN_0$.

\begin{lemma}
\label{lem2_nonlinear}
Denote $Y_1=\{\bxi\in\R^m\,|\,|\bxi|\le1\}$. It holds that

{\rm(i)} $|\bN_0(\bxi)|\le\kp$ for any $\bxi\in Y_1$;

{\rm(ii)} $|\bN_0(\bxi_1)-\bN_0(\bxi_2)|\le2\kp|\bxi_1-\bxi_2|$ for any $\bxi_1,\bxi_2\in Y_1$.
\end{lemma}

\begin{proof}
(i) For $\bxi\in Y_1$, we note that
\begin{equation*}
|\bN_0(\bxi)|\le\kp|\bxi|+|\bm{f}_0(\bxi)|=\kp|\bxi|+f_0(|\bxi|)=N_0(|\bxi|),
\end{equation*}
where $N_0$ is the scalar function \eqref{N0_definition}.
The rest follows the proof of Lemma \ref{lem_nonlinear}-(i) with $\beta=1$.

(ii) The Jacobian matrix of $\bm{N}_0$ at $\bm{\xi}\in\R^m$ is given by
\begin{equation*}
\nabla\bm{N}_0(\bm{\xi})=(\kp+1-|\bm{\xi}|^2)I_m-2(\bm{\xi}\otimes\bm{\xi}),
\end{equation*}
where $\otimes$ denotes the tensor product.
The $m$ eigenvalues of $\nabla\bm{N}_0(\bm{\xi})$ are given by
\begin{equation*}
\lambda_1(\bm{\xi})=\kp+1-3|\bm{\xi}|^2,\quad\lambda_2(\bm{\xi})=\lambda_3(\bm{\xi})=\cdots=\lambda_m(\bm{\xi})=\kp+1-|\bm{\xi}|^2.
\end{equation*}
Since $\kappa\ge 2$, for any $\bm{\xi}\in Y_1$ we have
\begin{equation*}
0\le\lambda_1(\bm{\xi})\le\lambda_m(\bm{\xi})\le\kp+1<2\kp.
\end{equation*}
Thus, using the mean-value theorem, we obtain (ii).
\end{proof}

Let $\hY = C(\overline{\Omega};\R^m)$ the space of continuous $\R^m$-valued functions defined on $\overline\Omega$
equipped with the supremum norm
$$\|\bw\|_{\hY}=\max_{\bx\in\overline{\Omega}}|\bw(\bx)|,\quad\forall\,\bw\in {\hY}.$$
By using Banach's fixed-point theorem as done in the proof of Theorem \ref{thm_model_max},
it is easy to show the existence and MBP of the problem \eqref{vector_AllenCahn}
with any integer $m\ge2$ as follows.
The main tools are the uniform ellipticity of the Laplace operator (see, e.g., \cite{Evans00}),
the properties of the nonlinear term given by Lemma \ref{lem2_nonlinear}, and the inequality
\begin{equation}
\label{keyeqn}
\Delta|\bu|^2 = 2|\nabla\bu|_F^2 + 2\bu\cdot\Delta\bu\geq 2\bu\cdot\Delta\bu.
\end{equation}

\begin{proposition}
If it holds that
\begin{subequations}
\label{cond_IBC_vector}
\begin{equation}
|\bu_0(\bx)| \le 1, \quad \forall\,\bx\in\overline{\Omega},
\end{equation}
the equation \eqref{vector_AllenCahn} with either periodic, homogeneous Neumann,
or Dirichlet boundary condition subject to
\begin{equation}
|\bg(t,\bx)| \le 1, \quad \forall\,t\in[0,T],\ \forall\,\bx\in\partial\Omega
\end{equation}
\end{subequations}
has a unique solution $\bu\in C([0,T];\hY)$ and it
satisfies $\|\bu(t)\|_\hY\le1$ for any $t\in[0,T]$.
\end{proposition}

This proposition actually implies that the closed unit ball in $\R^m$
is an invariant region of the solution to the equation \eqref{vector_AllenCahn}.
Moreover, according to Corollary 14.8-(b) in \cite{Smoller94},
the closed unit ball is the smallest invariant region.

Next we show that,
the ETD1 and ETDRK2 schemes for time integration of the equation \eqref{vector_AllenCahn}
both preserve the discrete MBP unconditionally.
Let $\bv^0=\bu_0$.
For the case of Dirichlet boundary condition,
the ETD1 and ETDRK2 solutions are then given by $\bv^{n+1}=\bw^n(\dt)$ with $\bw^n:[0,\dt]\to\hY$ such that
\begin{equation}
\label{vector_ETD12}
\begin{dcases}
\bw^n_s + \kp \bw^n = \Delta \bw^n + \widehat{\bN}_0(s,\bv^n), & s\in(0,\dt], \ \bx\in\Omega, \\
\bw^n(s,\bx) = \bg(t_n+s,\bx), & s\in[0,\dt], \ \bx\in\partial\Omega, \\
\bw^n(0,\bx) = \bv^n(\bx), & \bx\in\overline{\Omega},
\end{dcases}
\end{equation}
where
\[
\widehat{\bN}_0(s,\bv^n)=
\begin{dcases}
\bN_0(\bv^n), & \text{for ETD1},\\
\Big(1-\frac{s}{\dt}\Big)\bN_0(\bv^n)+\frac{s}{\dt}\bN_0(\tilde{\bv}^{n+1}), & \text{for ETDRK2}
\end{dcases}
\]
with $\tilde{\bv}^{n+1}$ generated by the ETD1 scheme.
For the cases of periodic or homogeneous Neumann boundary condition,
the ETD1 and ETDRK2 solutions are still given by the system \eqref{vector_ETD12} with small modifications of
removing its second equation  and using corresponding properties of the solution on the boundary.

\begin{theorem}
Assume that the stabilizing constant $\kp\ge2$ and \eqref{cond_IBC_vector} holds.
Then  the ETD1 and ETDRK2 schemes of the equation \eqref{vector_AllenCahn} both preserve the discrete MBP unconditionally,
i.e., for any time step size $\dt>0$, the solutions satisfy $\|\bv^n\|_{\hY}\le1$ for any $n\ge0$.
\end{theorem}

\begin{proof}
Since $\|\bv^0\|_{\hY}\le1$,
we just need to show that $\|\bv^k\|_{\hY}\le1$ implies $\|\bv^{k+1}\|_{\hY}\le1$ for any $k$.
We have $\bv^{k+1}=\bw^k(\dt)$,
where $\bw^k$ satisfies \eqref{vector_ETD12} with the superscript $n$ replaced by $k$.
Taking the dot product of the first equation in \eqref{vector_ETD12} with $\bw^k$
and using the fact from \eqref{keyeqn} that $\Delta|\bw^k|^2 \ge 2\bw^k\cdot\Delta\bw^k$,
we obtain
\begin{equation}
\label{thm_vector_ETD_pf}
\frac{1}{2}(|\bw^k|^2)_s + \kp|\bw^k|^2 \le \frac{1}{2}\Delta|\bw^k|^2 + |\widehat{\bN}_0(s,\bv^k)|\,|\bw^k|.
\end{equation}
Suppose there exists $(s^*,\bx^*)\in(0,\dt]\times\widehat{\Omega}$
such that
\begin{equation*}
|\bw^k(s^*,\bx^*)|=\max_{0\le s\le \dt}\|\bw^k(s)\|_{\hY}.
\end{equation*}
Since $|\bw^k(s,\bx)|^2$ is a real scalar-valued function,
we have $(|\bw^k|^2)_s\ge0$ at $(s^*,\bx^*)$.
If $\bx^*\in\Omega$, we have $\Delta|\bw^k(s^*,\bx^*)|^2\le0$.
If $\bx^*\in\partial\Omega$, we have $|\bw^k(s^*,\bx^*)|\le1$
for the case of Dirichlet boundary condition (the proof is then in fact completed in this case)
or $\Delta|\bw^k(s^*,\bx^*)|^2\le0$
for the cases of periodic and homogeneous Neumann boundary conditions.
Then we obtain from \eqref{thm_vector_ETD_pf} that
$\kp|\bw^k(s^*,\bx^*)| \le |\widehat{\bN}_0(s^*,\bv^k(\bx^*))|$.
Since $\|\bv^k\|_{\hY}\le1$, according to Lemma \ref{lem2_nonlinear}-(i),
for both ETD1 and ETDRK2 schemes, we always have $|\widehat{\bN}_0(s^*,\bv^k(\bx^*))|\le\kp$,
and thus $|\bw^k(s^*,\bx^*)|\le1$.
Then we have $\|\bv^{k+1}\|_{\hY}\le1$, which completes the proof.
\end{proof}

\begin{remark}
For the space-discrete version of the equation  \eqref{vector_AllenCahn},
an essential condition for the MBP to hold is that $\Delta_h$,
the spatial discretization of the operator $\Delta$,
satisfies $\Delta_h|\bu|^2 \geq 2\bu\cdot\Delta_h\bu$.
\end{remark}

Similar to the scalar-valued problem,
here we present the fully discrete ETD1 and ETDRK2 schemes for practical computations.
With $\hL=\Delta$, we still use the notations  $\hL_{hc}$ and  $\hL_{\kp,h}$
as defined by \eqref{def_boundaryopt} and \eqref{def_Lkp} respectively, then
the fully discrete ETDRK2 scheme reads
\[
\begin{dcases}
\tilde{\bv}^{n+1}
=\varphi_0(\dt\hL_{\kp,h} )\bv^n+\dt \varphi_1(\dt\hL_{\kp,h} )\bN_0(\bv^n)
+ \int_0^\dt \e^{-(\dt-s)\hL_{\kp,h}} \hL_{hc} \bg(t_n+s)\,\d s,\\
\bv^{n+1} =\tilde{\bv}^{n+1} + \dt\varphi_2(\dt\hL_{\kp,h} ) [\bN_0(\tilde{\bv}^{n+1})-\bN_0(\bv^n)],
\end{dcases}
\]
and the fully discrete ETD1 scheme is given by the first step of the ETDRK2 scheme.
The schemes for cases of homogeneous Neumann and periodic boundary conditions
could be given in the similar way.

\subsection{Extension to real matrix-valued equations}

In \cite{OsWa20}, a problem of
finding the stationary points of an energy for orthogonal matrix-valued functions was studied.
Since the orthogonality constraint is non-trivial to enforce,
a penalty term is added to the energy to offer a relaxed (phase-field or diffuse-interface) formulation.
The gradient flow for such energy reads
\begin{equation}
\label{matrix_AllenCahn}
U_t=\Delta U+U(I_m-U^TU),\quad t\in(0,T],\ \bx\in\Omega,
\end{equation}
where $U:[0,T]\times\overline{\Omega}\to\R^{m\times m}$ is subject to the initial condition
\[
U(0,\cdot)=U_0,\quad\bx\in\overline{\Omega}
\]
and either homogeneous Dirichlet, periodic, or homogeneous Neumann boundary condition.
Denote by $|\cdot |_2$ the matrix $2$-norm
and by $\R^{m\times m}_s$ the set of all real symmetric $m$-by-$m$ matrices.

Let us define $\hZ = C(\overline{\Omega};\R^{m\times m}_s)$ the space of continuous $\R^{m\times m}_s$-valued functions defined
on $\overline\Omega$ equipped with the supremum norm
$$\|W\|_{\hZ}=\max_{\bx\in\overline{\Omega}} |W(\bx)|_2,\quad\forall\, W\in \hZ.$$
Similarly to $X$ for the scalar-value equation case,
we then define $Z$, as well as $D(\hL_0)$ and $\hL_0$,  in accordance with $\hL=\Delta$ and the respective boundary conditions.
Then, we can show that $\hL_0:D(\hL_0)\to Z$ satisfies
the matrix-valued analogue of Lemma \ref{lem_L_semigroup}.

\begin{lemma}
\label{lem3_L_semigroup}
For all $\lambda>0$ and all $W\in D(\hL_0)$, it holds
\begin{equation}
\label{lem3_L_dissipative}
\|(\lambda\hI-\Delta)W\|_{\hZ}\ge\lambda\|W\|_{\hZ},
\end{equation}
and thus $\hL_0$ generates a contraction semigroup $\{S_{\hL_0}(t)\}_{t\ge0}$ on $Z$, i.e., $\opnorm{S_{\hL_0}(t)}\le1$.
\end{lemma}

\begin{proof}
First, for any diagonal matrix $W(\bx)=\diag\{w_i(\bx):1\le i\le m\}$,
there exists $\bx_0\in\Omega$ (for the homogeneous Dirchlet boundary condition)
or $\bx_0\in\overline{\Omega}$ (for the periodic or homogeneous Neumann boundary condition)
and $i_0$ such that
$$\|W\|_{\hZ}= |W(\bx_0)|_2=|w_{i_0}(\bx_0)|=\max_{\bx\in\overline{\Omega}}|w_{i_0}(\bx)|.$$
Since $|w_{i_0}(\bx)|^2$ is a real scalar-valued function, we have
$$2w_{i_0}(\bx_0)\Delta w_{i_0}(\bx_0)\le2w_{i_0}(\bx_0)\Delta w_{i_0}(\bx_0)+2|\nabla w_{i_0}(\bx_0)|^2
=\Delta|w_{i_0}(\bx_0)|^2\le0.$$
Then, for any $\lambda>0$, we have
\begin{align}
\lambda |W(\bx_0) |_2^2
& \le\lambda|w_{i_0}(\bx_0)|^2-w_{i_0}(\bx_0)\Delta w_{i_0}(\bx_0)\nonumber\\
& =w_{i_0}(\bx_0)\cdot(\lambda\hI-\Delta)w_{i_0}(\bx_0)
\le |W(\bx_0) |_2 |(\lambda\hI-\Delta)W(\bx_0) |_2,\label{lem3_L_dissipative_pf}
\end{align}
which implies \eqref{lem3_L_dissipative} for any diagonal matrix $W$.

Next, for any $W\in D(\hL_0)$, let $\bx_0$ be the point such that
$$\|W\|_{\hZ}= |W(\bx_0)|_2=\max_{\bx\in\overline{\Omega}} |W(\bx) |_2.$$
Since $W(\bx_0)$ is symmetric, there exists an orthonormal matrix $O$ such that $\widehat{W}=O^TW(\bx_0)O$ is diagonal.
We then derive from \eqref{lem3_L_dissipative_pf} that
$$\lambda |W(\bx_0) |_2=\lambda |\widehat{W} |_2
\le |(\lambda\hI-\Delta)\widehat{W} |_2= |(\lambda\hI-\Delta)W(\bx_0) |_2\le\|(\lambda\hI-\Delta)W\|_{\hZ},$$
which leads to \eqref{lem3_L_dissipative}.
\end{proof}

Introducing a stabilizing constant $\kp>0$, the equation \eqref{matrix_AllenCahn} is equivalent to
\begin{equation}
\label{matrix_intform}
U(t+\dt)=\e^{-\kp\dt}S_{\hL_0}(\dt)U(t)+\int_0^\dt\e^{-\kp(\dt-s)}S_{\hL_0}(\dt-s)\hN_0(U(t+s))\,\d s,
\end{equation}
where
\begin{equation}
\label{N0mat_definition}
\hN_0(Q)=\kp Q+Q(I_m-Q^TQ),\quad Q\in\R^{m\times m}.
\end{equation}
We again require $\kp\ge 2$, and then obtain the following lemma on the matrix-valued function $\hN_0$.

\begin{lemma}
\label{lem3_nonlinear}
Denote ${\mathcal M}_1=\{Q\in\R^{m\times m}_s \,|\, |Q|_2\le1\}$. It holds that

{\rm(i)} $ |\hN_0(Q) |_2\le\kp$ for any $Q\in {\mathcal M}_1$;

{\rm(ii)} $ |\hN_0(Q_1)-\hN_0(Q_2) |_2\le2\kp |Q_1-Q_2 |_2$ for any $Q_1,Q_2\in {\mathcal M}_1$.
\end{lemma}

\begin{proof}
(i) Since any real symmetric matrix can be diagonalized orthonormally,
and $\hN_0(Q)$ is also diagonal for any diagonal matrix $Q\in {\mathcal M}_1$,
the property (i) is the direct consequence of Lemma \ref{lem_nonlinear}-(i).

(ii) Since $\R^{m\times m}$ is identical to $\R^{m^2}$ in the sense of isomorphism,
the matrix-valued function $\hN_0:\R^{m\times m}\to\R^{m\times m}$ defined by \eqref{N0mat_definition}
could be regarded as a vector-valued mapping $\R^{m^2}\to\R^{m^2}$,
whose Jacobian matrix gives the matrix derivative of $\hN_0$.
In this sense, the matrix derivative of $\hN_0$ at $Q\in\R^{m\times m}_s$ is given by \cite[Theorem 4]{MacRae74}
\begin{equation*}
\mathrm{D}\hN_0(Q) = (\kp+1)I_{m^2} - (Q^2\otimes I_m+I_m\otimes Q^2+Q\otimes Q).
\end{equation*}
Denote by $\{\mu_j\}_{j=1}^m$ the eigenvalues of $Q$. Then
the eigenvalues of $\mathrm{D}\hN_0(Q)$, denoted by $\{\lambda_{ij}\}_{i,j=1}^m$, are given by \cite[Theorem 4.2.12]{HoJo91}
\begin{equation*}
\lambda_{ij} = \kp+1 - (\mu_i^2+\mu_j^2+\mu_i\mu_j), \quad 1\le i,j\le m.
\end{equation*}
For any $Q\in {\mathcal M}_1$, it holds $0\le \mu_i^2+\mu_j^2+\mu_i\mu_j\le 3$.
Since $\kp\ge 2$, we have
\begin{equation*}
0 \le \lambda_{ij} \le \kp+1 < 2\kp, \quad 1\le i,j\le m.
\end{equation*}
Thus,  we obtain the property (ii) by using the mean-value theorem.
\end{proof}

By conducting the similar analysis as done in Section \ref{subsect_framework} and \cite{DuJuLiQi19},
we can prove the MBP for the matrix-valued equation \eqref{matrix_AllenCahn}.

\begin{proposition}
If $U_0(\bx)$ is symmetric and  $|U_0(\bx)|_2\le1$ for any $\bx\in\overline{\Omega}$, then
the equation \eqref{matrix_AllenCahn} with either homogeneous Dirichlet, periodic, or homogeneous Neumann boundary condition
has a unique solution $U\in C([0,T];\hZ)$ and it satisfies $\|U(t)\|_{\hZ}\le1$ for any $t\in[0,T]$.
\end{proposition}

\begin{proof}
Setting $t=0$ in \eqref{matrix_intform} gives us
\begin{equation*}
U(\dt)=\e^{-\kp\dt}S_{\hL_0}(\dt)U_0+\int_0^\dt\e^{-\kp(\dt-s)}S_{\hL_0}(\dt-s)\hN_0(U(s))\,\d s,\quad\dt\ge0.
\end{equation*}
Denote $Z_1=\{W\in Z \,|\, \|W\|_{\hZ}\le1\}$.
For a fixed $t_1>0$ and a given $V\in C([0,t_1];Z_1)$, let us define $W:[0,t_1]\to Z$ by
\begin{equation}
\label{matrix_solution_pf1}
W(\dt)=\e^{-\kp\dt}S_{\hL_0}(\dt)U_0+\int_0^\dt\e^{-\kp(\dt-s)}S_{\hL_0}(\dt-s)\hN_0(V(s))\,\d s,\quad\dt\in[0,t_1].
\end{equation}
Obviously, $W$ is uniquely defined and
$$\|W(\dt)\|_{\hZ}\le\e^{-\kp\dt}\opnorm{S_{\hL_0}(\dt)}\|U_0\|_{\hZ}
+\int_0^\dt\e^{-\kp(\dt-s)}\opnorm{S_{\hL_0}(\dt-s)}\|\hN_0(V(s))\|_{\hZ}\,\d s.$$
According to Lemmas \ref{lem3_L_semigroup} and \ref{lem3_nonlinear}-(i), we derive
$$\|W(\dt)\|_{\hZ}\le\e^{-\kp\dt}+\Big(\int_0^\dt\e^{-\kp(\dt-s)}\,\d s\Big)\kp
=\e^{-\kp\dt}+\frac{1-\e^{-\kp\dt}}{\kp}\cdot\kp=1,\quad\dt\in[0,t_1],$$
which means $W\in C([0,t_1];Z_1)$.

Then, for $t_1<\kp^{-1}\ln2$, similar to the proof of Theorem \ref{thm_model_max} and using Lemma \ref{lem3_nonlinear}-(ii),
the mapping $\hA:C([0,t_1];Z_1)\to C([0,t_1];Z_1)$ defined by $\hA[V]=W$ according to \eqref{matrix_solution_pf1}
is a contraction.
We then can conclude, similarly to the earlier scalar case,
a unique solution $U(t)\in Z_1$ to the matrix-valued equation \eqref{matrix_AllenCahn}
exists on the time interval $[0,t_1]$ and can be further extended to $[0,T]$.
This completes the proof.
\end{proof}

By using Lemma \ref{lem3_nonlinear}, it also can be shown that both the ETD1 and ETDRK2 schemes
for time integration of \eqref{matrix_AllenCahn} again
preserve the discrete MBP unconditionally when the stabilizing constant $\kp\ge 2$.

\section{Numerical experiments}
\label{sect_numerical}

There exists a large amount of literature comparing and
showing the excellent  performance of the ETD schemes  in numerical simulations for local continuum and nonlocal models
\cite{DuJuLiQi19,JuLiQiZh18,JuZhDu15,JuZhZhDu15,WaJuDu16,ZhZhWaJuDu16,ZhJuZh16}.
The practical efficiency of the fully discrete ETD schemes (i.e., the finite dimensional cases)
depends highly on the implementation of the actions of the operator/matrix exponentials.
We first review existing  algorithms for computing the products of the matrix exponentials with vectors,
and then we present some detailed experimental results.

\subsection{Implementations of matrix exponentials}

Let us recall the fully discrete ETDRK2 scheme \eqref{fullydis_ETDRK2_phi}.
There exist many efficient algorithms for
computing the matrix functions $\varphi_\gamma(\dt\hL_{\kp,h} )$, $\gamma=0,1,2$
and their products with vectors.

When the spatial domain of the problem is regular,
for instance, a rectangle $\prod_{j=1}^d(a_j,b_j)$,
and the matrix $\hL_{\kp,h} $ has some certain special structure,
fast Fourier transform (FFT) based algorithms are adequate
to calculate the above products of $\varphi$-functions with vectors.
This case arises commonly in the models for material sciences.
For example, the matrices $\hL_h$ given
in Examples \ref{eg_laplacian_dis} and \ref{eg_nonlocal_dis}
are symmetric Toeplitz matrices for Case (C1) or circulant matrices for Case (C2).
As we know, the product of a circulant matrix with a vector actually gives a circulant convolution,
which could be implemented by using the FFT.
For Toeplitz matrices, their products with a vector can be calculated by the sine or cosine transform.
Alternatively, one can expand a Toeplitz matrix to a large circulant matrix and again make use of FFT
for fast implementations.

When the shape of the spatial domain is arbitrary
or the matrix $\hL_{\kp,h} $ does not possess a certain special structure,
it is usually difficult to develop fast algorithms for matrix exponentials and their products with vector.
In \cite{Hig08,MoVan78,MoVan03}, many methods are surveyed for computing the exponential of a matrix,
such as Taylor series, ODE solver, inverse Laplace transform, matrix decomposition and so on.
The Matlab built-in command \verb"expm(A)" computes the matrix exponential of $A$
by using a scaling and squaring algorithm with a Pad\'e approximation.
The performance of these methods often depends on the target problems
and some of them are appropriate only for certain special problems.
In the past two decades,
Krylov subspace method based on Arnoldi or Lanczos iterations has become a powerful tool
for computing the products of matrix exponentials with vectors \cite{GaSa92,HoLu97,HoLuSe98,Sidje98},
especially for large-scale sparse problems.
More recently, this method also has been further combined with
incomplete orthogonalization \cite{GaPu16} and adaptive time stepping \cite{GaRaTo18,NiWr12}.

\begin{remark}
In the ETD1 scheme, by approximating $\e^{\dt\hL_{\kp,h}}\approx\hI+\dt\hL_{\kp,h} $,
one can obtain the first-order semi-implicit scheme,
which is linear just as the ETD1 scheme and also preserves the MBP \cite{ShTaYa16,TaYa16}.
In general, the ETD1 scheme is slightly more time-consuming than
the classic semi-implicit schemes but with the added benefits of being more accurate \cite{JuLiQiZh18,JuZhDu15}.
The fully-implicit backward Euler scheme naturally possesses good numerical stability
(no extra stabilization term is needed),  but it is not linear and usually more time-consuming
due to the need of nonlinear iterations at each time step.
A nonlinear Crank--Nicolson scheme was proposed and proved to conditionally preserve the MBP in \cite{HoTaYa17}.
To our best knowledge, there is so far no other linear  second-order (in time) scheme
like the ETDRK2 scheme which is unconditionally MBP-preserving.
Furthermore, the ETD schemes preserve the exponential behavior of the linear operator,
which is often crucial in the practical simulations of stiff systems
(i.e., the spectral radius $\rho(\hL_{\kp,h})\gg1$).
\end{remark}

\subsection{Experimental results}

Two examples will be tested to show
the MBP-preserving property and numerical performance of the ETD methods.
The first example focuses on the scalar equation \eqref{model_eq}
with the nonlinear function \eqref{eg_f_flory} consisting of logarithmic terms.
The second example solves the vector-valued equation \eqref{vector_AllenCahn}
to show some interesting evolutions of vortices in composite domains.
In all experiments, the ETDRK2 scheme with uniform time step size $\dt=0.01$ is used for time integration.

\begin{example}
\label{eg_num1}
Consider the scalar equation \eqref{model_eq} of the unknown function $u:\Omega\subset\R^2\to\R$ in the domain $\Omega=(0,2\pi)\times(0,2\pi)$
with $\hL=0.01\Delta$ and $f_0$ given by \eqref{eg_f_flory}, i.e., the negative of the derivative of Flory--Huggins potential.
The maximum bound principle plays an important role in this case
since the equation consists of the logarithmic terms
which will yield complex numbers if the value of the solution is located out of the interval $(-1,1)$.
The periodic and homogeneous Neumann boundary conditions are considered.
\end{example}

We use  the uniform rectangular mesh with size $h=2\pi/512$ for partition of the domain in this example and
a random data ranging from $-0.9$ to $0.9$ is generated on the mesh as the initial configuration of $u$.
The spatial discretization is done through the central finite difference as Example \ref{eg_laplacian_dis}
so that the FFT can be used here for fast calculations needed in the ETDRK2 scheme.
We set $\theta=0.8$ and $\theta_c=1.6$ in \eqref{eg_f_flory}.
According to Example \ref{eg_nonlinear_flory}, the positive root of $f_0(\rho)=0$ is $\rho\approx 0.9575$
and the stabilizing constant is thus chosen as $\kp=8.02$.

\figurename~\ref{fig_scalar_per} shows the configurations of the approximate solution $u$
at $t=1$, $5$, $8$, $30$, $100$, and $160$ subject to the periodic boundary condition.
The simulation results illustrate
the dynamics beginning with a random state and towards the homogeneous steady state of $u\equiv -\rho$,
which is reached after about $t=166$ in our simulations.
The evolution of the energy is plotted in \figurename~\ref{fig_scalar_energy}-(left)
and that of the supremum norm of $u$ in \figurename~\ref{fig_scalar_maxmod}-(left).
The simulation results subject to the homogeneous Neumann boundary condition
are presented in \figurename~\ref{fig_scalar_neu},
where the same homogeneous steady state is reached after about $t=547$.
\figurename~\ref{fig_scalar_energy}-(right) and \figurename~\ref{fig_scalar_maxmod}-(right)
show the corresponding evolutions of the energy and the supremum norm respectively.
We observe that the energy decreases monotonically under both boundary conditions
and the MBP is also preserved perfectly so that the solution is always located in the interval $(-1,1)$.

\begin{figure}[!ht]
\centerline{
\hspace{-0.3cm}
\includegraphics[width=0.34\textwidth]{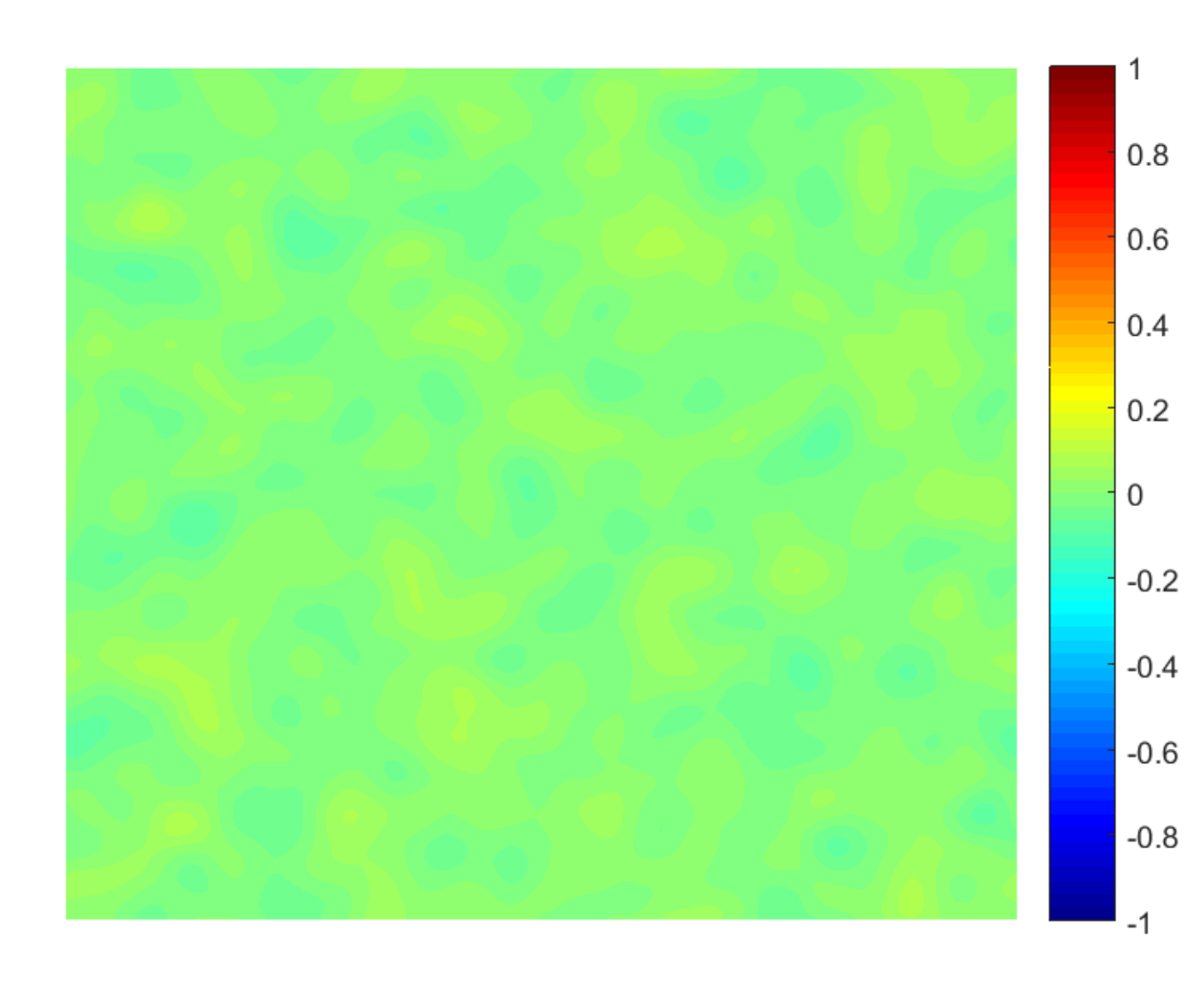}\hspace{-0.2cm}
\includegraphics[width=0.34\textwidth]{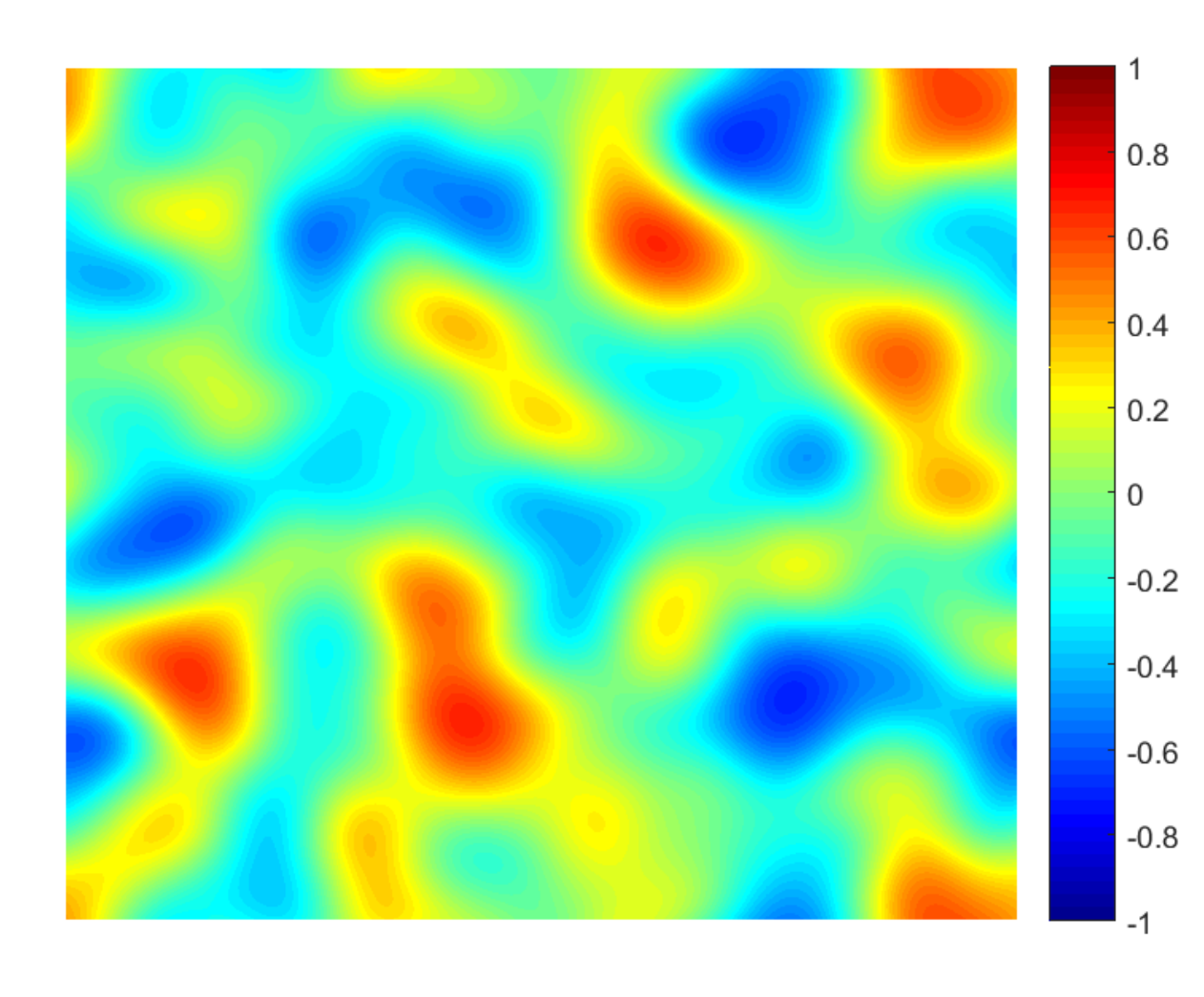}\hspace{-0.2cm}
\includegraphics[width=0.34\textwidth]{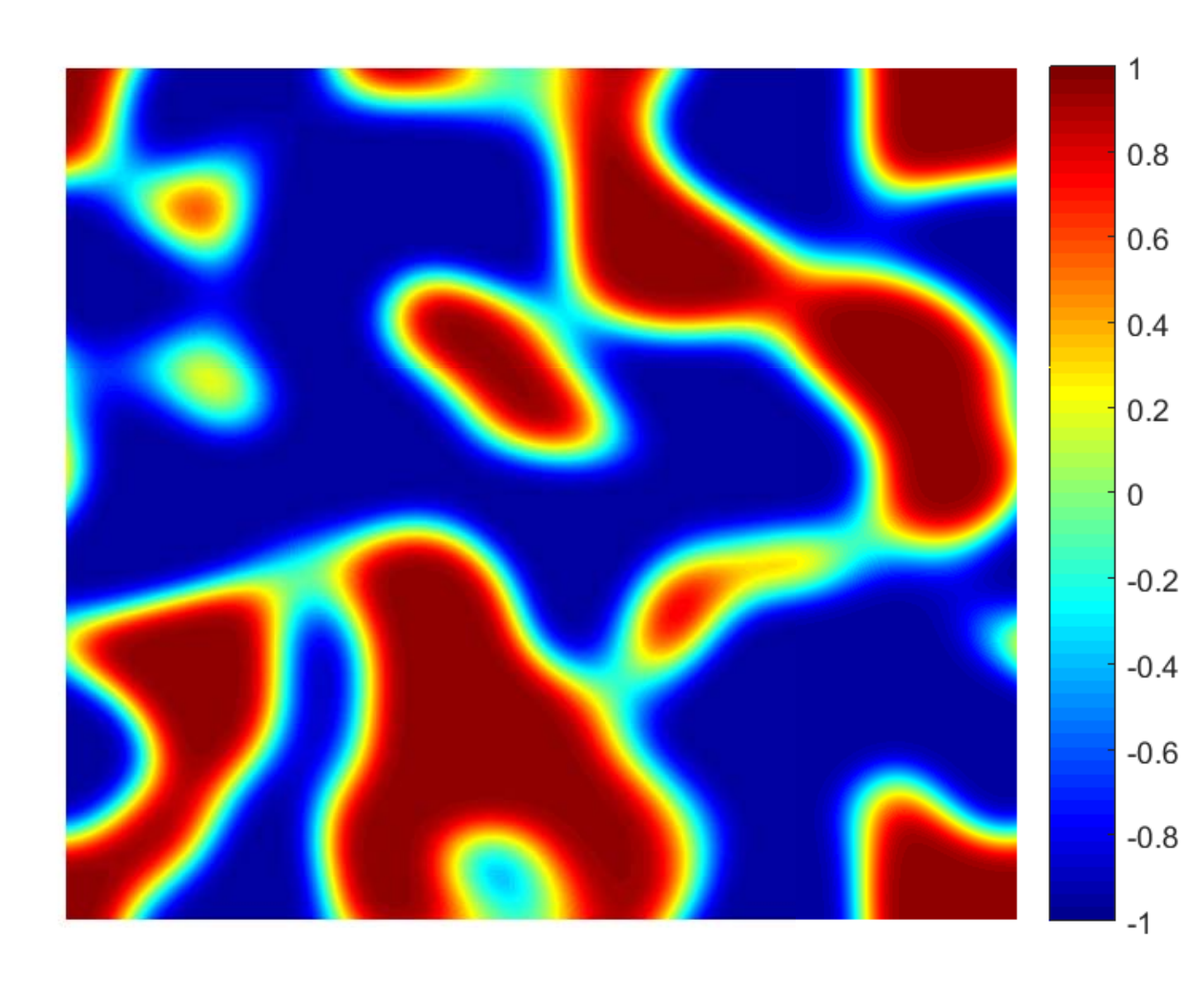}}
\vspace{-0.2cm}
\centerline{
\hspace{-0.3cm}
\includegraphics[width=0.34\textwidth]{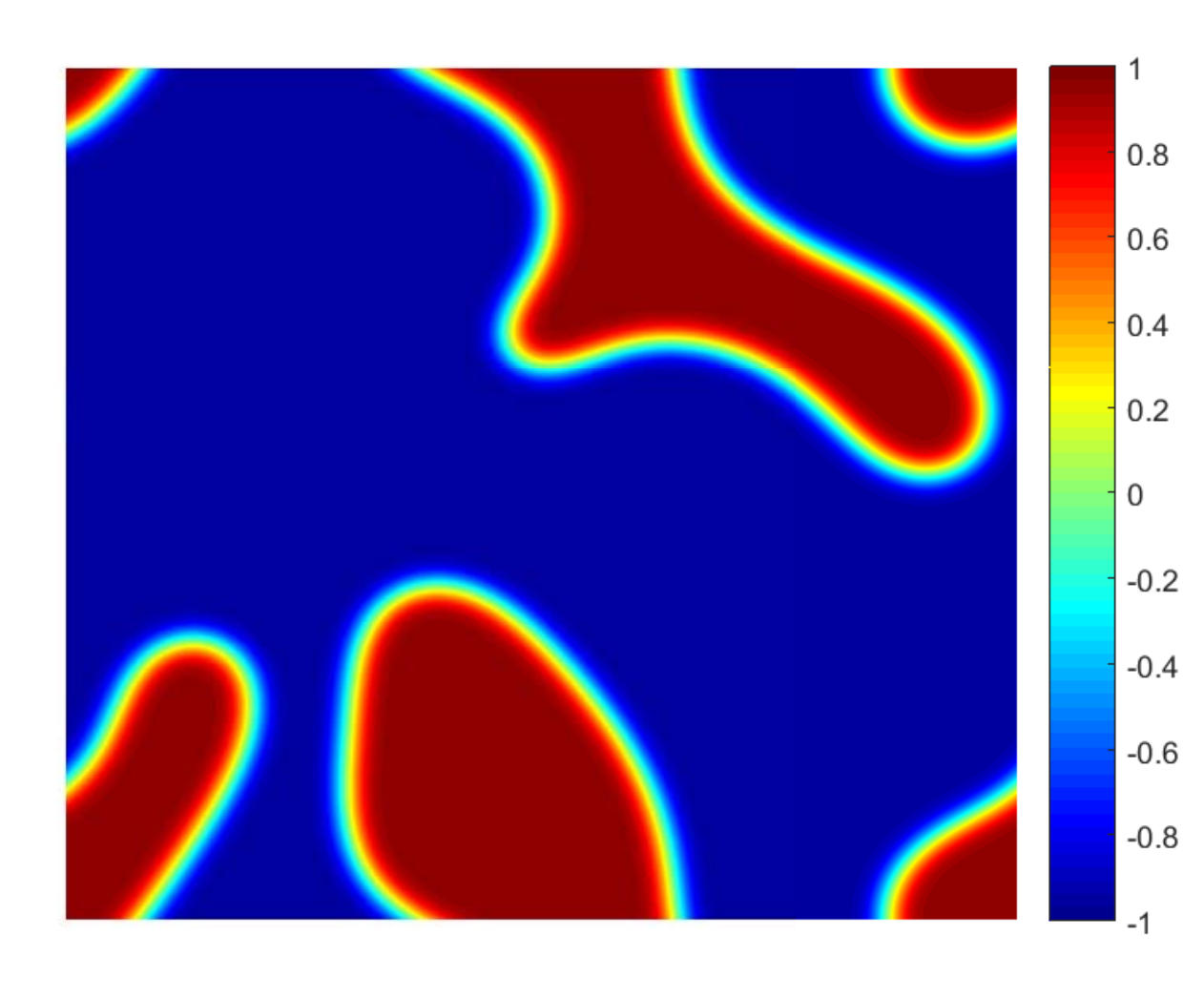}\hspace{-0.2cm}
\includegraphics[width=0.34\textwidth]{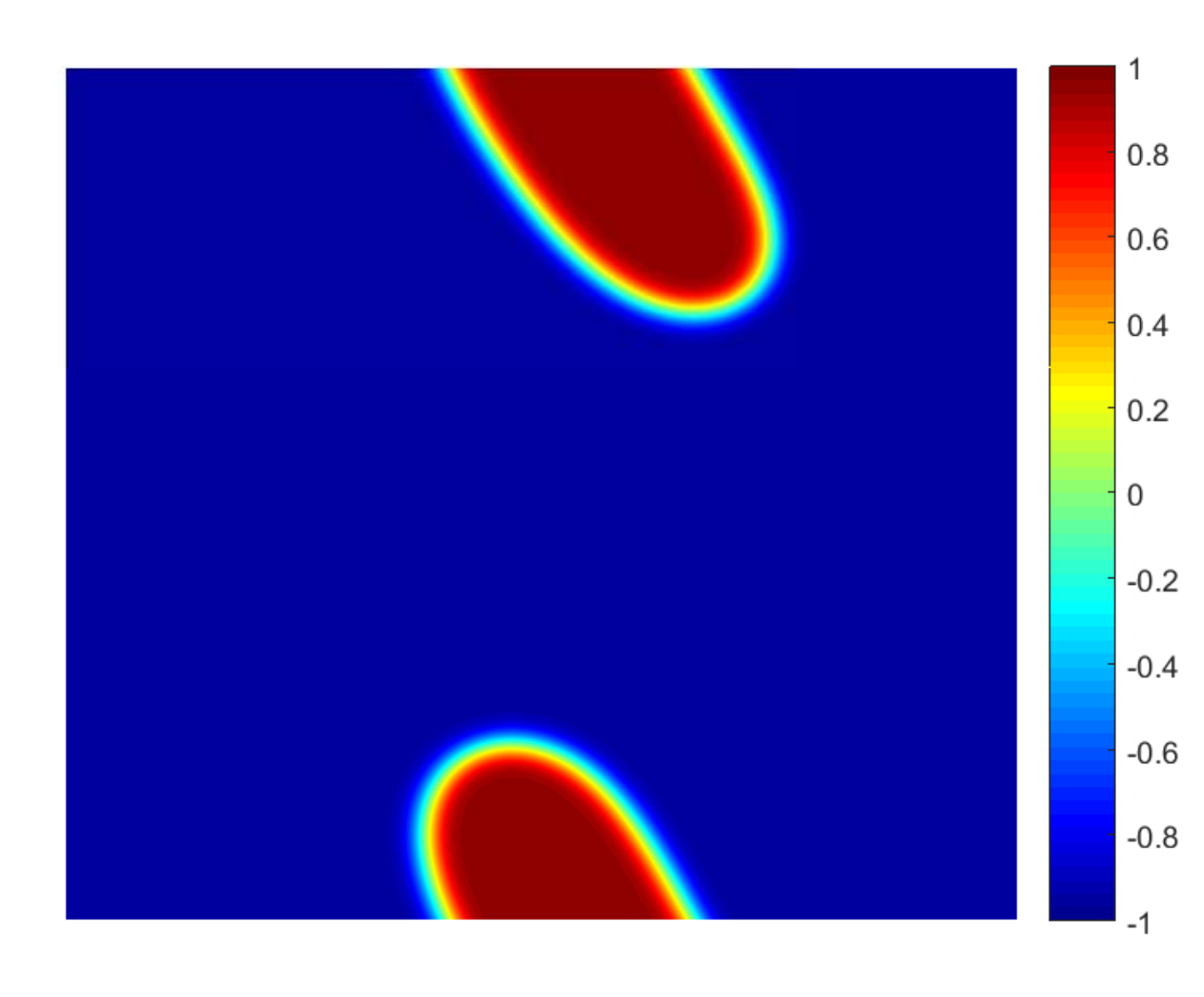}\hspace{-0.2cm}
\includegraphics[width=0.34\textwidth]{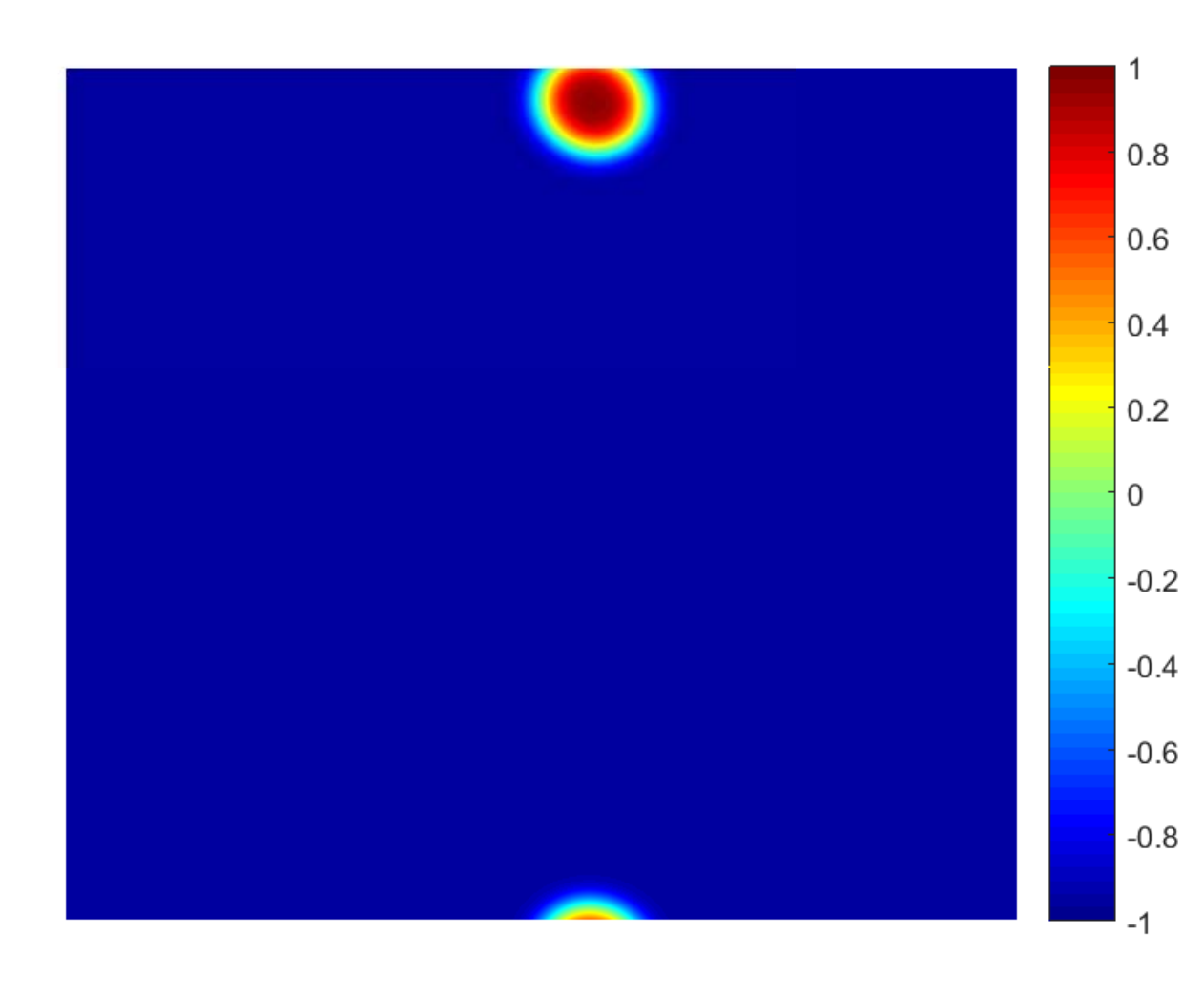}}
\vspace{-0.2cm}
\caption{The simulated solution $u$ subject to periodic boundary condition
at $t=1$, $5$, $8$, $30$, $100$, and $160$ respectively
(left to right and top to bottom) in Example \ref{eg_num1}.}
\label{fig_scalar_per}
\end{figure}

\begin{figure}[!ht]
\centerline{
\hspace{-0.3cm}
\includegraphics[width=0.34\textwidth]{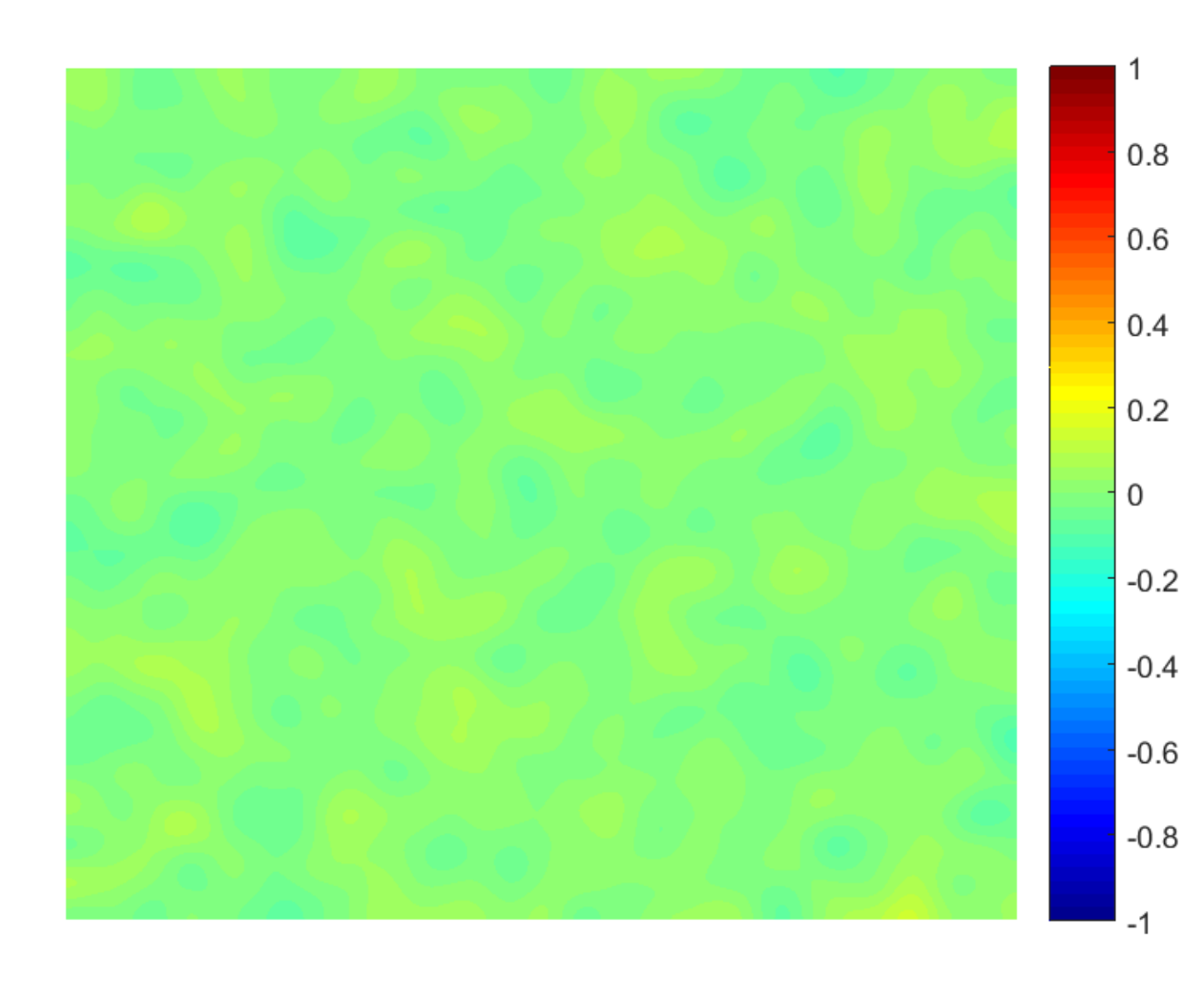}\hspace{-0.2cm}
\includegraphics[width=0.34\textwidth]{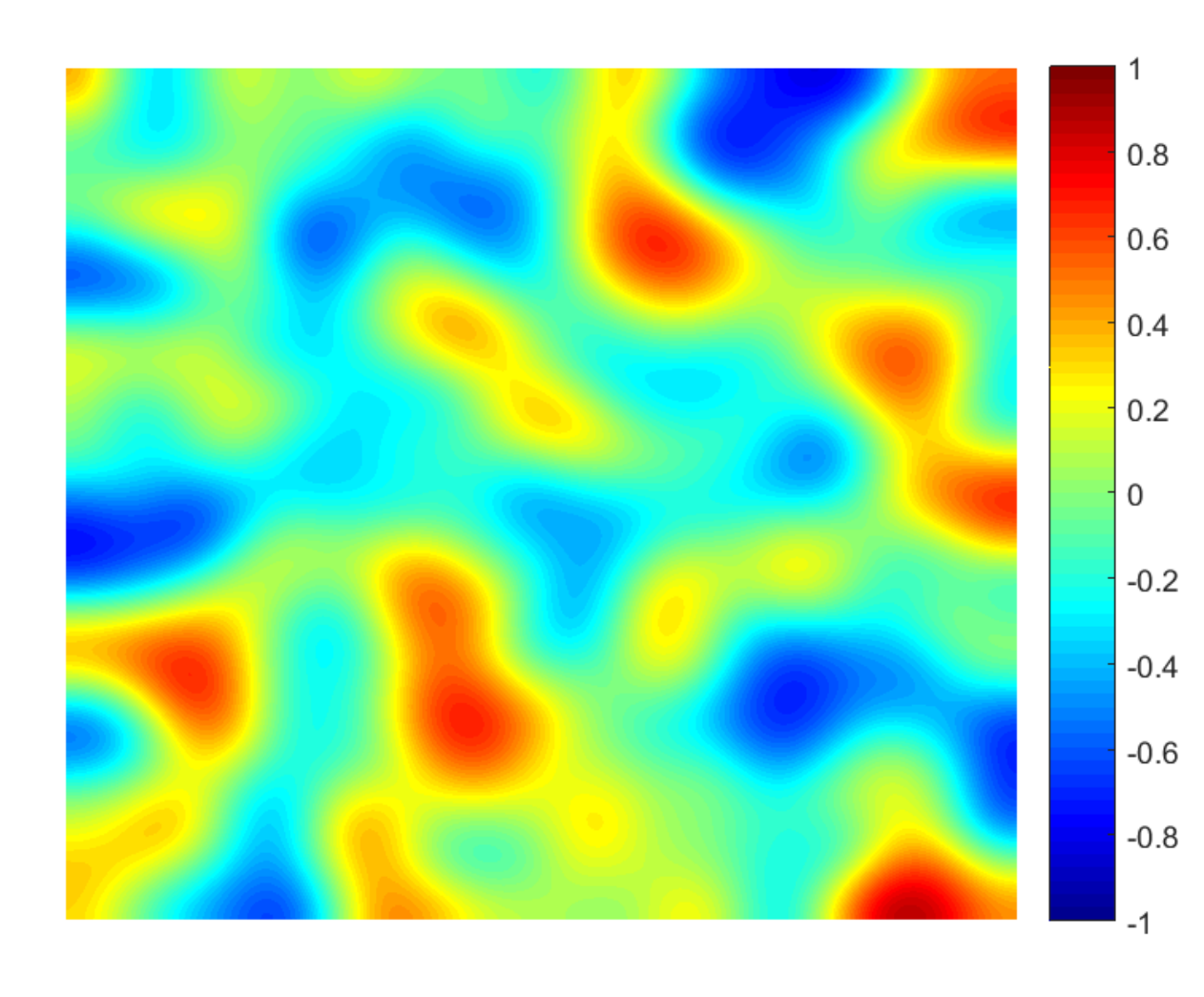}\hspace{-0.2cm}
\includegraphics[width=0.34\textwidth]{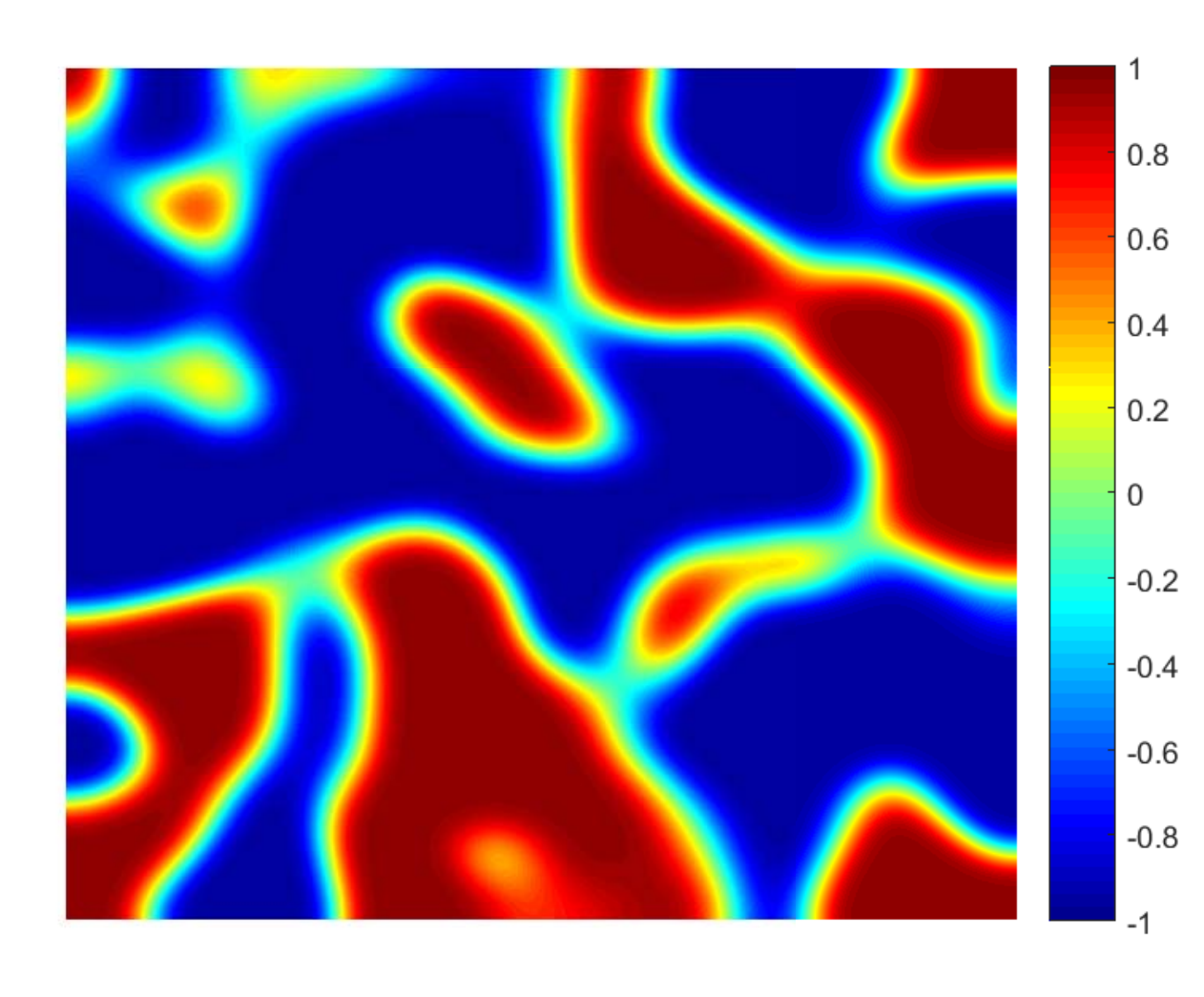}}
\vspace{-0.2cm}
\centerline{
\hspace{-0.3cm}
\includegraphics[width=0.34\textwidth]{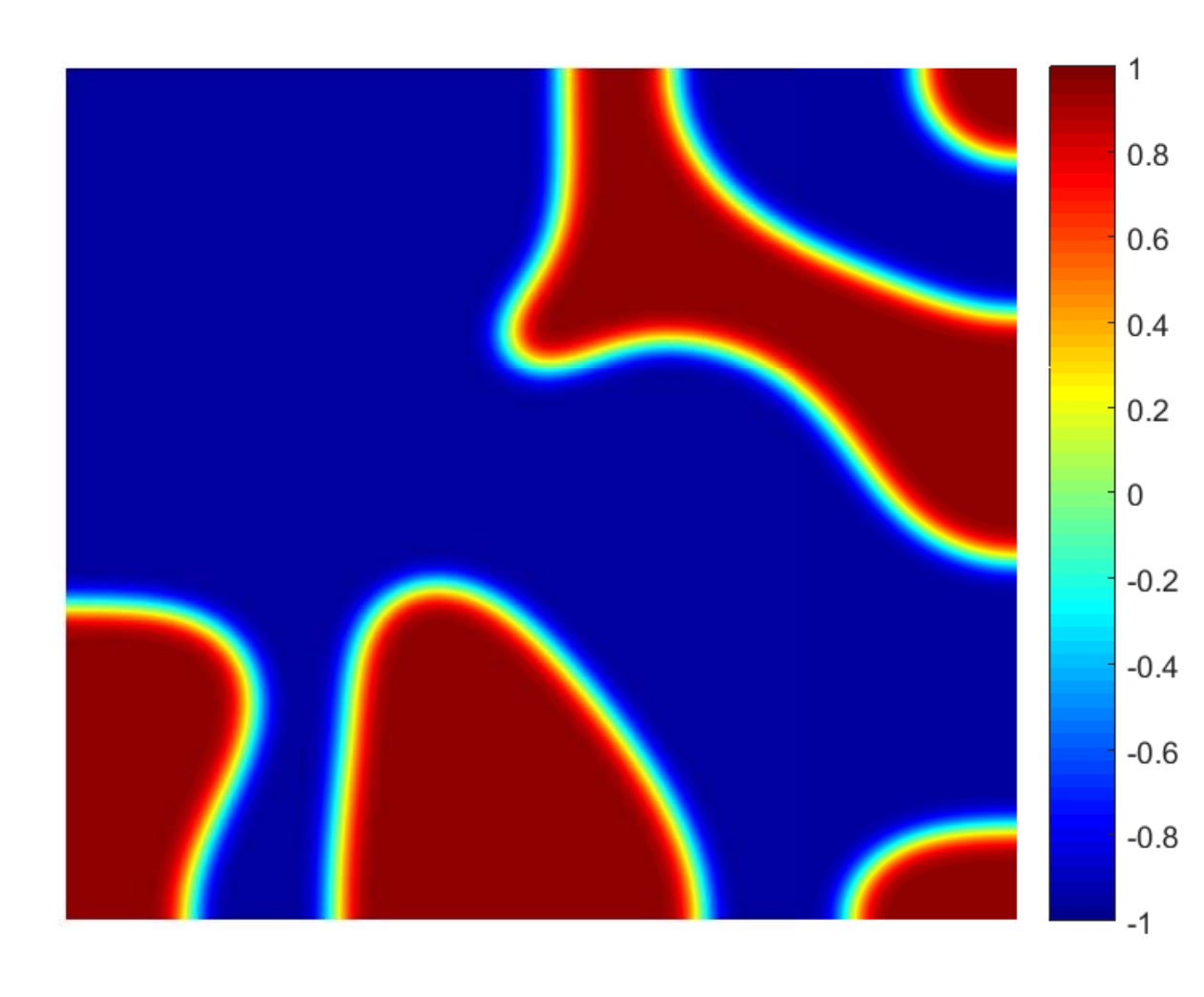}\hspace{-0.2cm}
\includegraphics[width=0.34\textwidth]{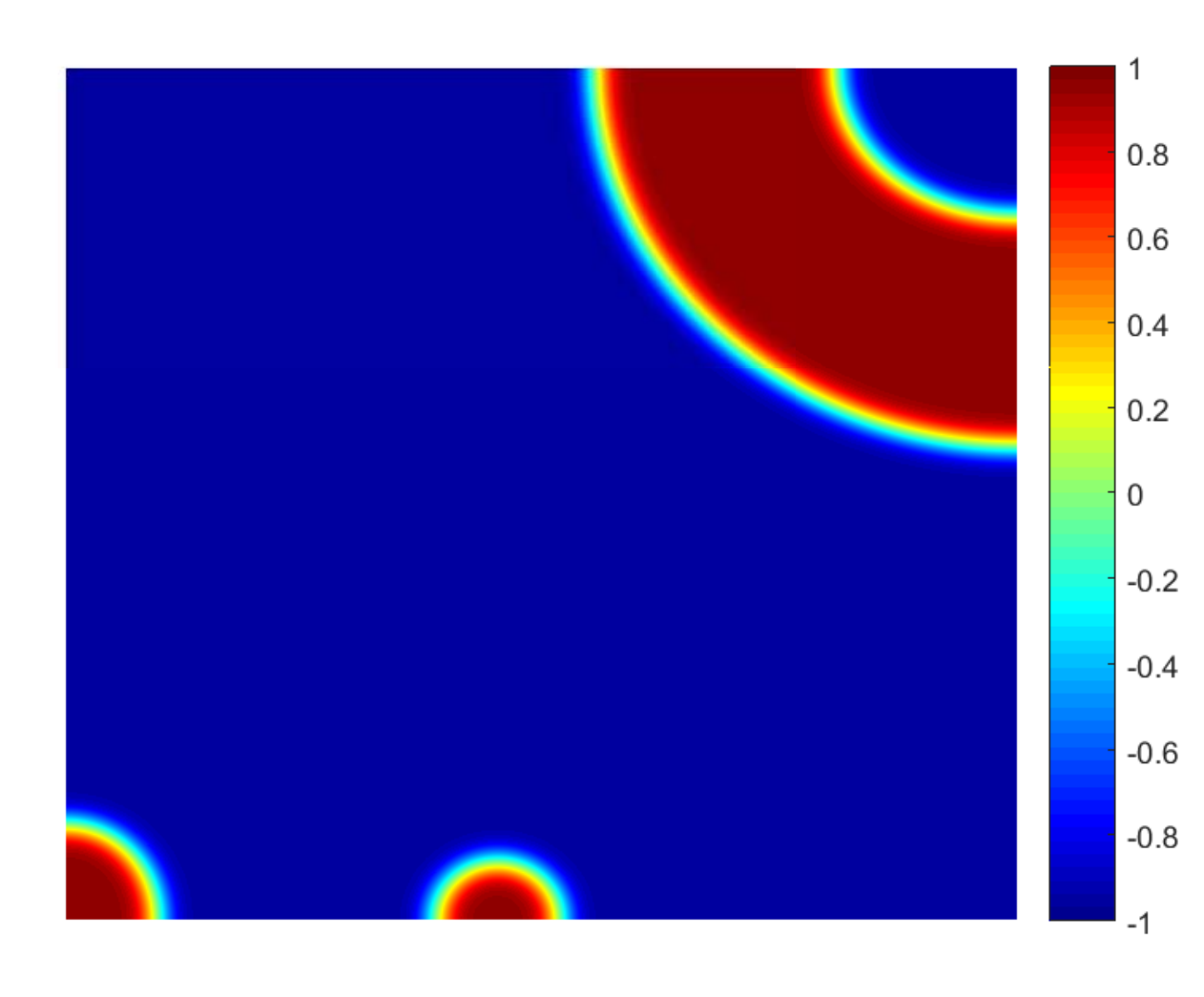}\hspace{-0.2cm}
\includegraphics[width=0.34\textwidth]{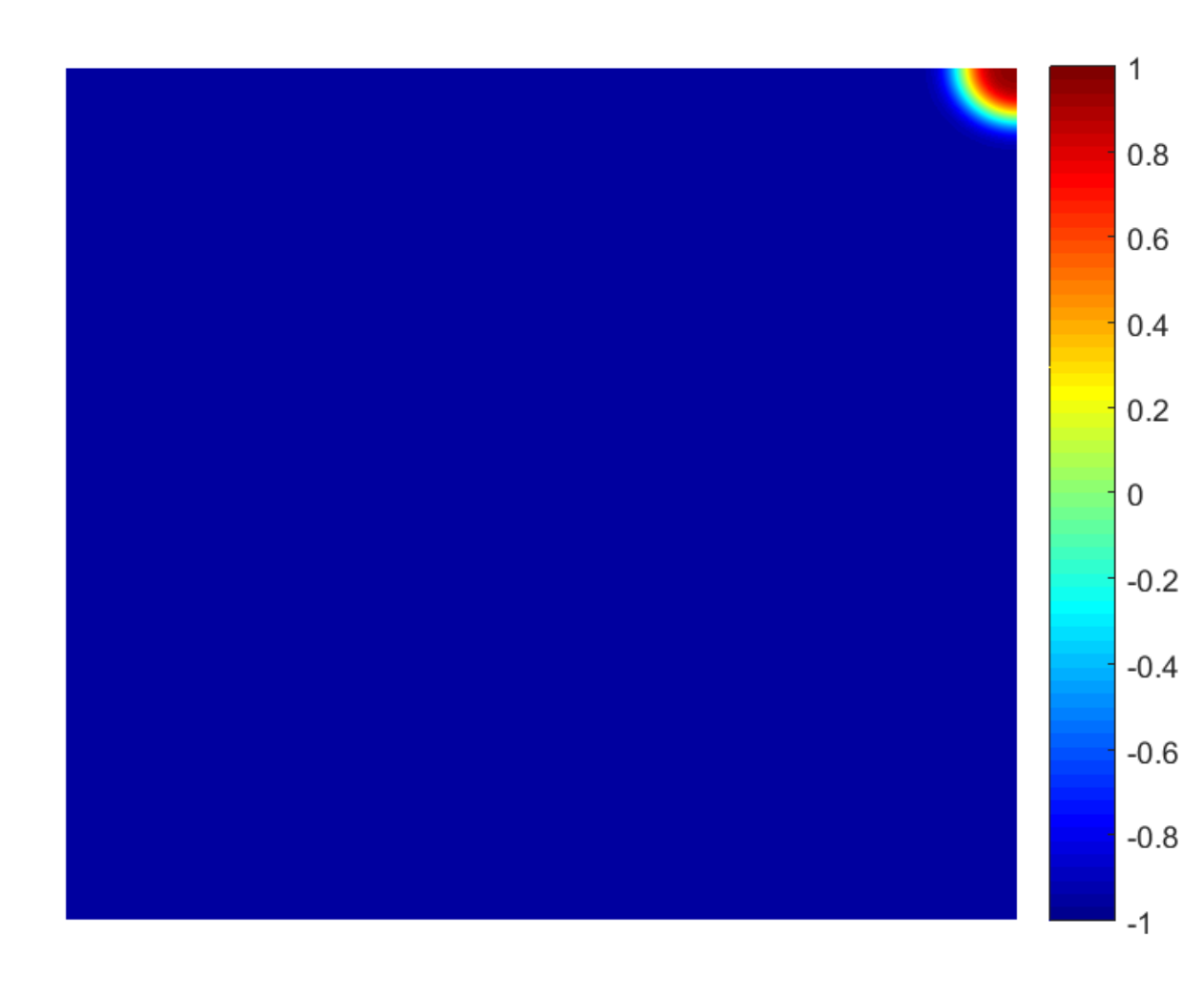}}
\vspace{-0.2cm}
\caption{The simulated solution $u$  subject to the homogeneous Neumann boundary condition
at $t=1$, $5$, $8$, $30$, $160$, and $540$ respectively
(left to right and top to bottom) in Example \ref{eg_num1}.}
\label{fig_scalar_neu}
\end{figure}

\begin{figure}[!ht]
\centerline{\hspace{-0.1cm}
\includegraphics[width=0.54\textwidth]{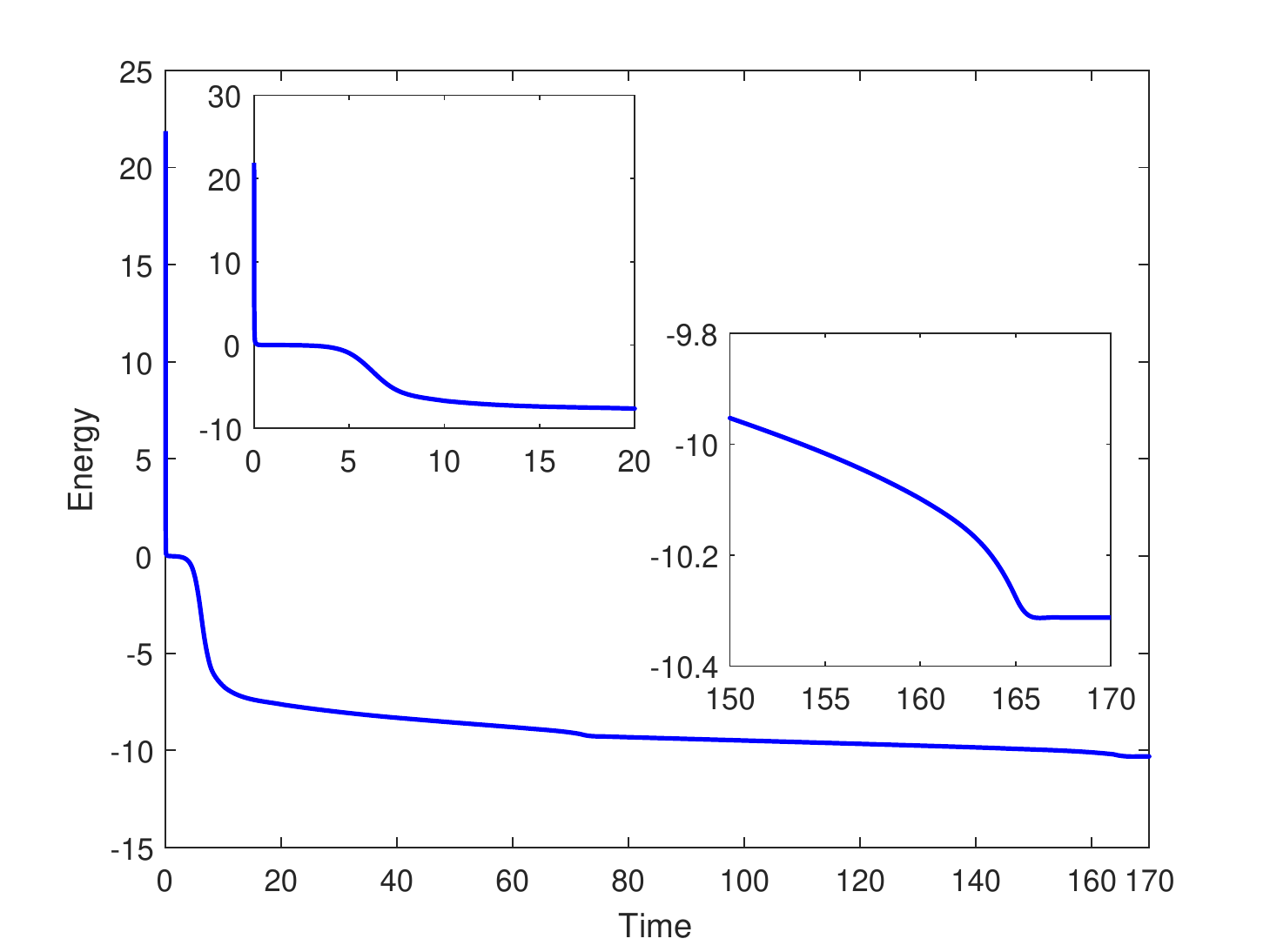}\hspace{-0.5cm}
\includegraphics[width=0.54\textwidth]{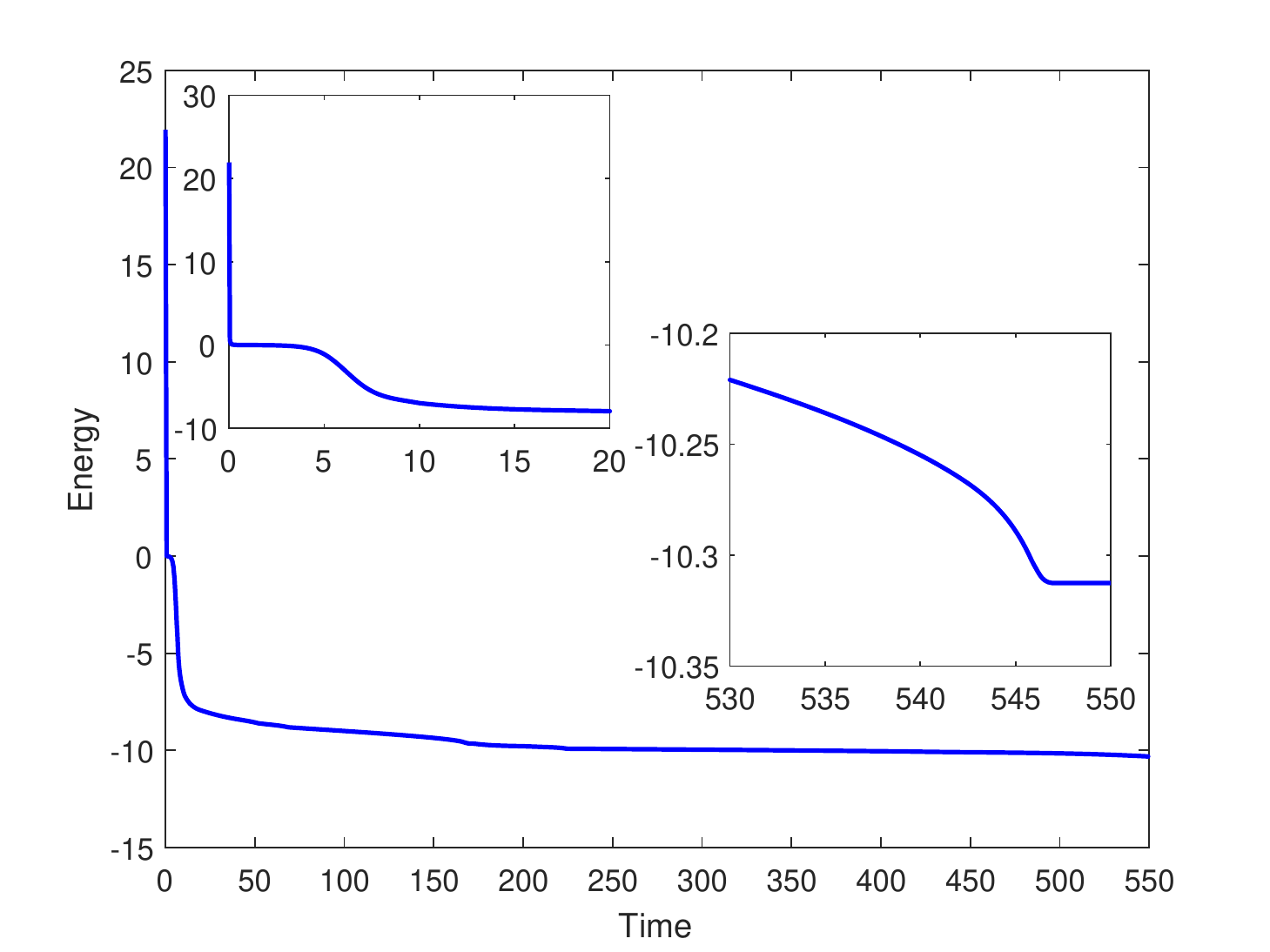}}
\vspace{-0.2cm}
\caption{Evolutions of the energy subject to the periodic  (left)
and homogeneous Neumann (right) boundary conditions respectively in Example \ref{eg_num1}.}
\label{fig_scalar_energy}
\end{figure}

\begin{figure}[!ht]
\centerline{\hspace{-0.1cm}
\includegraphics[width=0.54\textwidth]{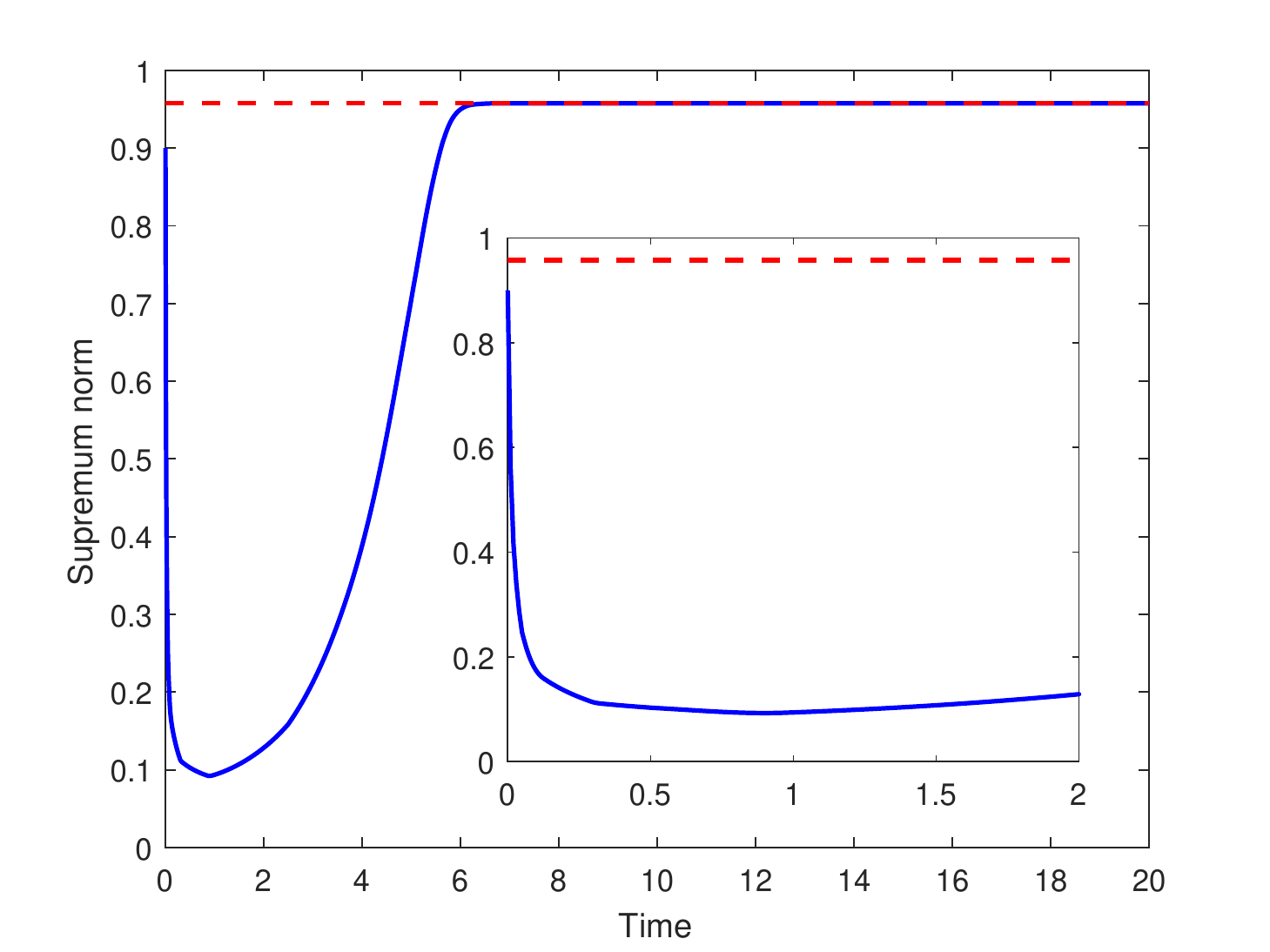}\hspace{-0.5cm}
\includegraphics[width=0.54\textwidth]{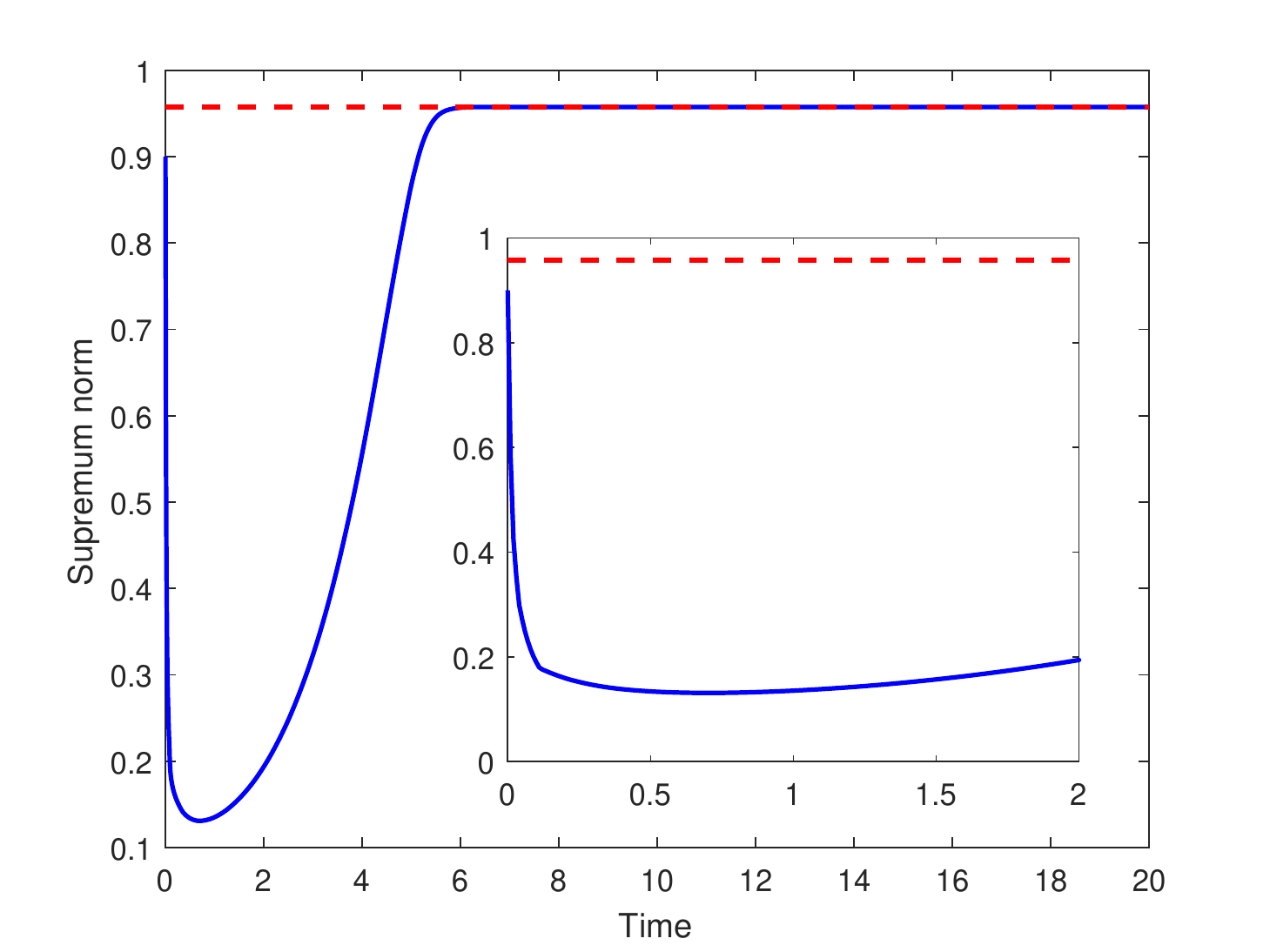}}
\vspace{-0.2cm}
\caption{Evolutions of the supremum  norm of the simulated solution $u$
subject to the periodic (left) and homogeneous Neumann (right) boundary conditions
respectively in Example \ref{eg_num1}.}
\label{fig_scalar_maxmod}
\end{figure}

\begin{example}
\label{eg_num2}
Consider the vector-valued equation \eqref{vector_AllenCahn}
of the unknown vector field $\bu:\Omega\subset\R^2\to\R^2$
with $\Delta$ replaced by $0.005\Delta$.
We test the domain
$$\Omega = \{(x,y)\in\R^2\,|\,\sqrt{x^2+y^2}<1\} \setminus \{(x,y)\in\R^2\,|\,\sqrt{(x-0.2)^2+y^2}\le0.5\},$$
that is, a region inside the unit disk but outside a circle with a shifted center.
The Dirichlet boundary condition is set to be
\[
\bu(x,y)=
\begin{dcases}
(x, y), & \text{if } \sqrt{x^2+y^2}=1, \\
(2x-0.4, -2y), & \text{if } \sqrt{(x-0.2)^2+y^2}=0.5,
\end{dcases}
\]
i.e., the values of the vector field $\bu$ on the outside boundary
are always fixed to be a unit vector in the direction of $(x,y)$
and  on the inside boundary be a unit vector in the direction of $(x,-y)$.
Thus the winding number of the boundary is $2$.
\end{example}

We adopt the $C^0$ finite element spatial discretization with piecewise linear basis functions on triangular meshes
and the mass-lumping technique in this example.
The stabilizing constant is set to be $\kp=2$.
PHIPM \cite{NiWr12} is used for computing linear combinations
of the products of the $\varphi$-functions with vectors in the ETDRK2 scheme.
We generate a triangular mesh with $2210$ nodes and $4158$ elements for the domain $\Omega$
and the initial configuration of the vector field $\bu$ on the interior nodes with the fixed magnitude $0.9$
but random directions according to a uniform distribution.
\figurename~\ref{fig_vector_fem1} shows the simulated vector fields
at $t=0.1$, $0.5$, $1.2$, $2.5$, $15$, and $100$ respectively.
We observe that the initial disordered state quickly transitions into
a more orderly structure which then asymptotically evolves to a steady state.
The obtained steady state of the vector field $\bu$ contain two vortices/defects which are symmetric with respect to the $x$-axis.
The evolution of the energy is plotted in the left graph of \figurename~\ref{fig_maxmod} for this example
and we see the energy decreases monotonically in time.
The right figure in \figurename~\ref{fig_maxmod}
presents the evolution of the maximum value of $|\bu|$ over  the interior nodes
and it demonstrates again that the MBP is perfectly preserved.

\begin{figure}[!ht]
\centerline{\hspace{-0.1cm}
\includegraphics[width=0.34\textwidth]{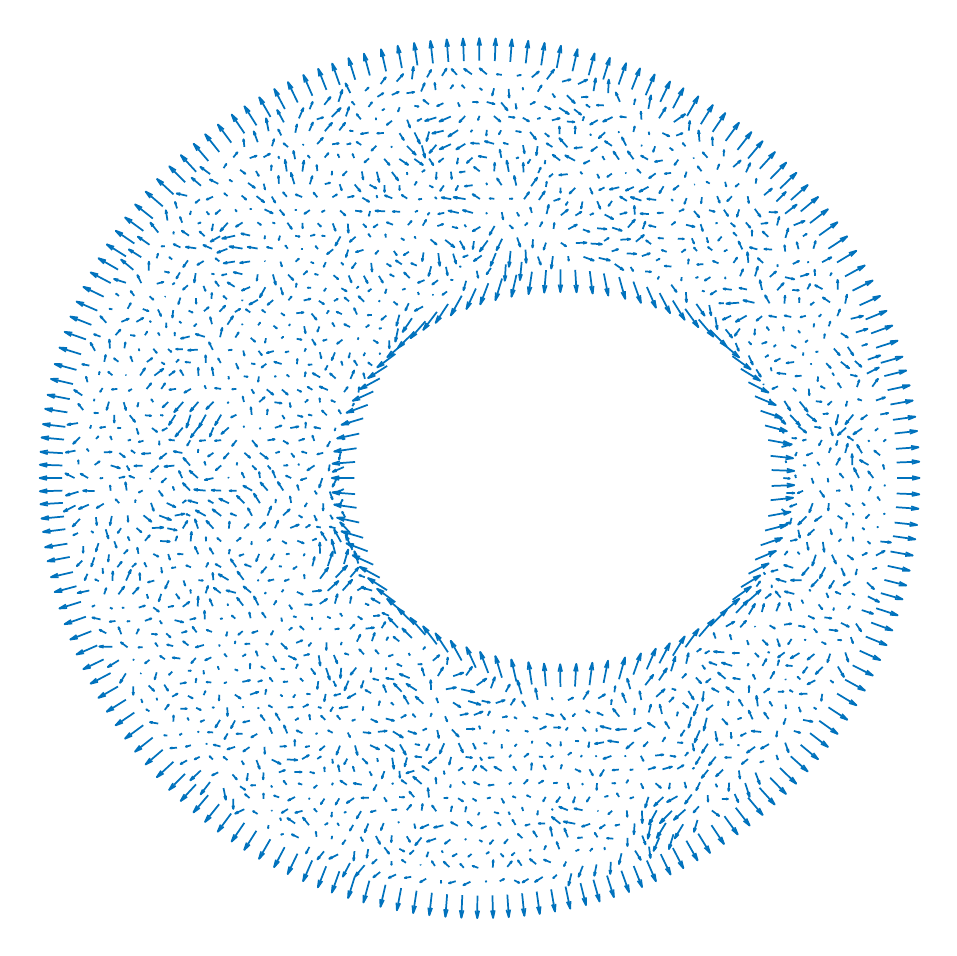}\hspace{-0.2cm}
\includegraphics[width=0.34\textwidth]{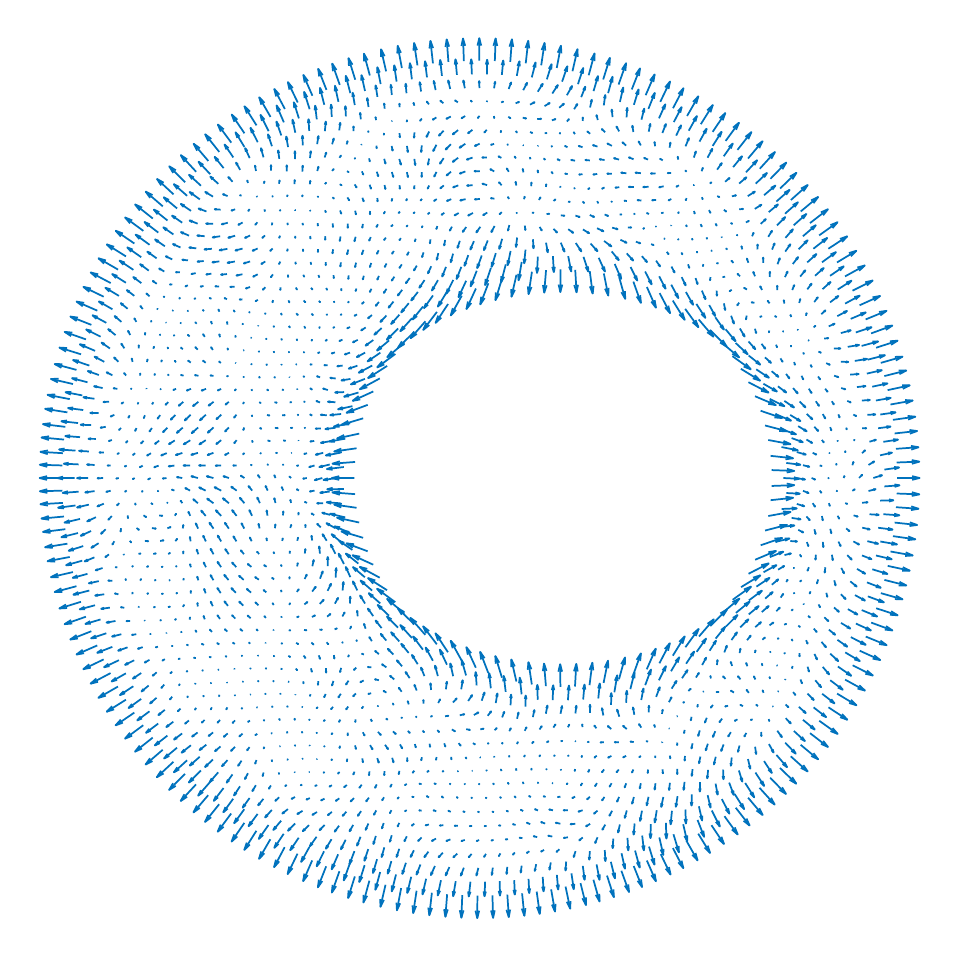}\hspace{-0.2cm}
\includegraphics[width=0.34\textwidth]{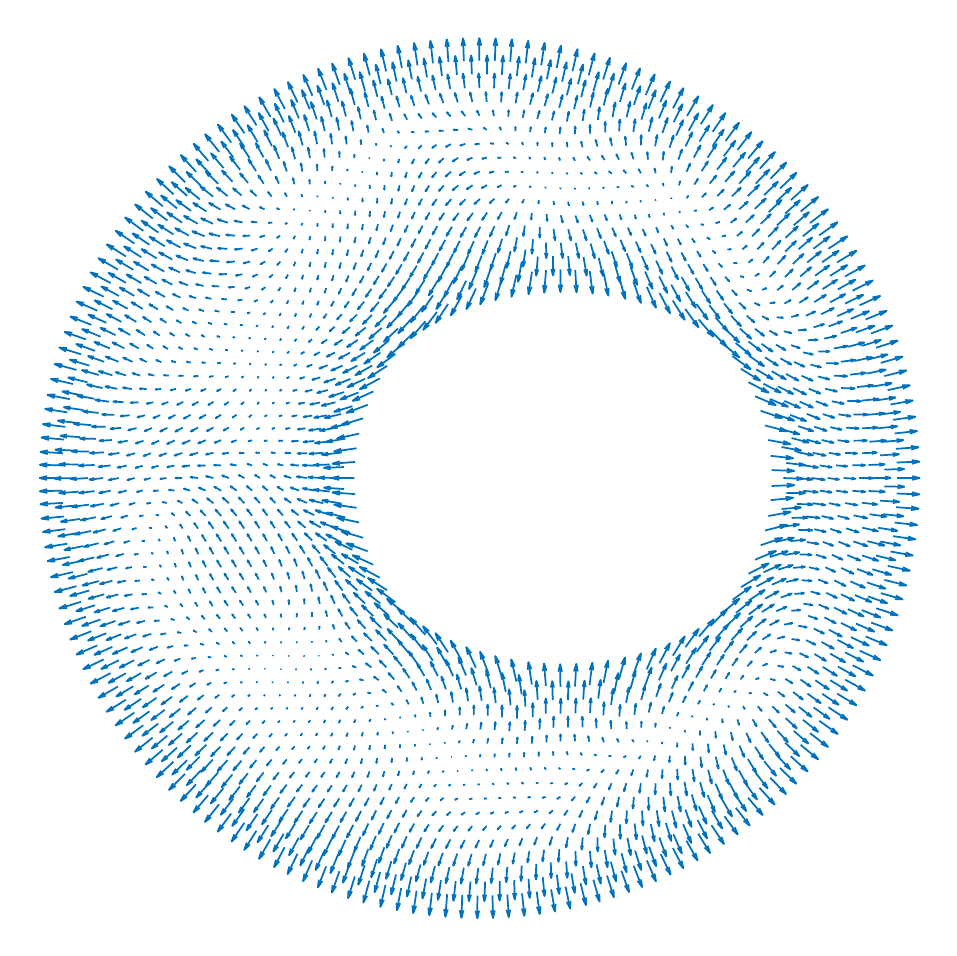}}
\centerline{\hspace{-0.1cm}
\includegraphics[width=0.34\textwidth]{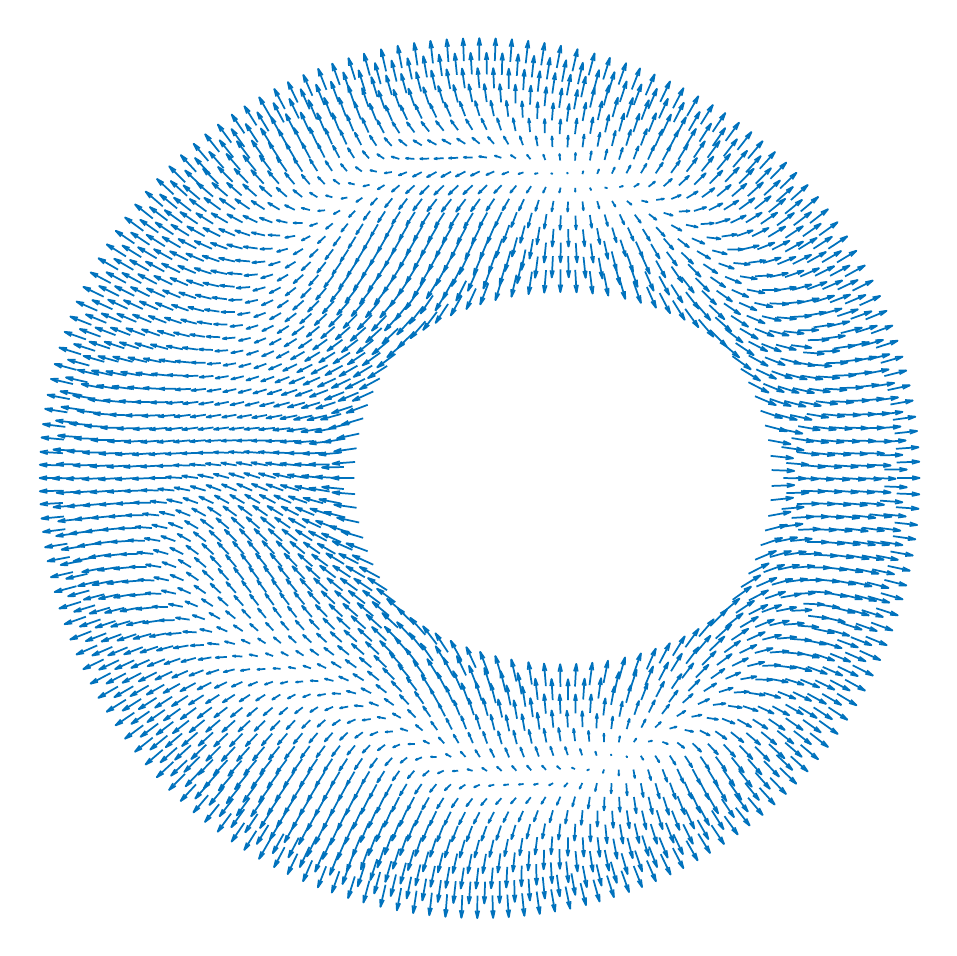}\hspace{-0.2cm}
\includegraphics[width=0.34\textwidth]{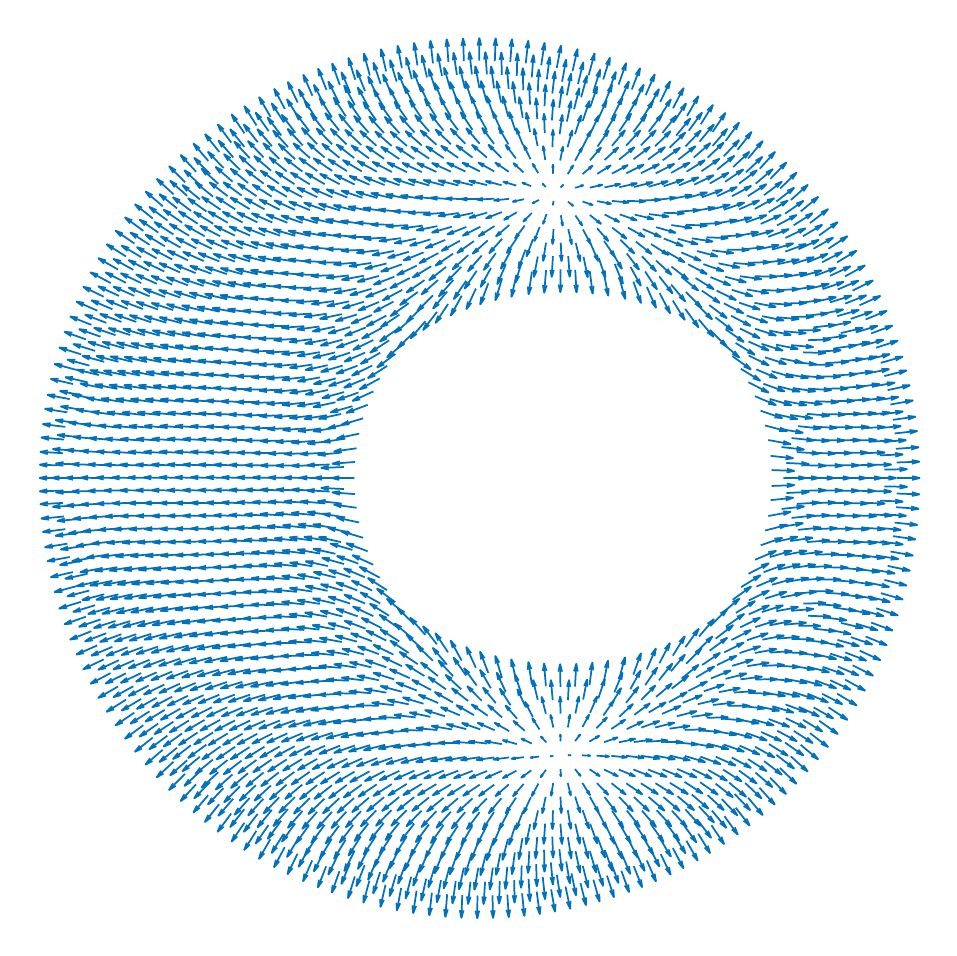}\hspace{-0.2cm}
\includegraphics[width=0.34\textwidth]{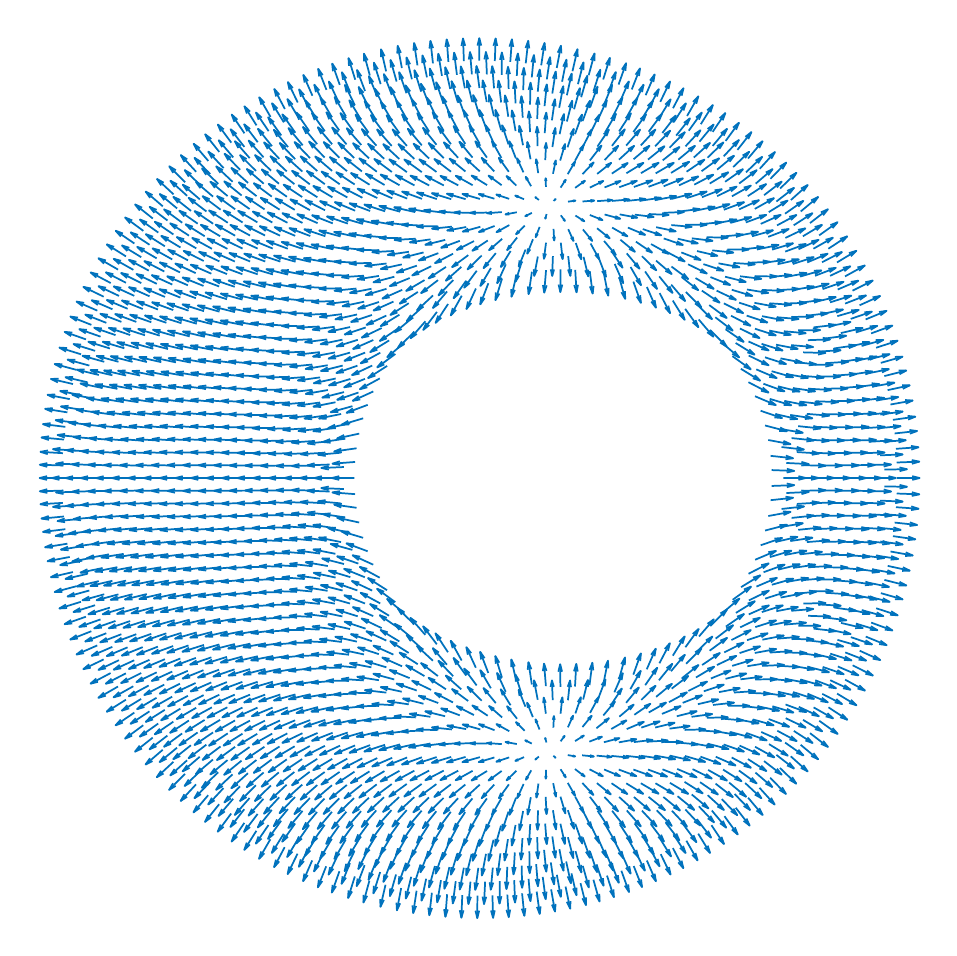}}
\vspace{-0.2cm}
\caption{The simulated vector field $\bu$ at $t=0.1$, $0.5$, $1.2$, $2.5$, $15$, and $100$ respectively
(left to right and top to bottom) of Example \ref{eg_num2}.}
\label{fig_vector_fem1}
\end{figure}

\begin{figure}[!ht]
\centerline{\hspace{-0.1cm}
\includegraphics[width=0.54\textwidth]{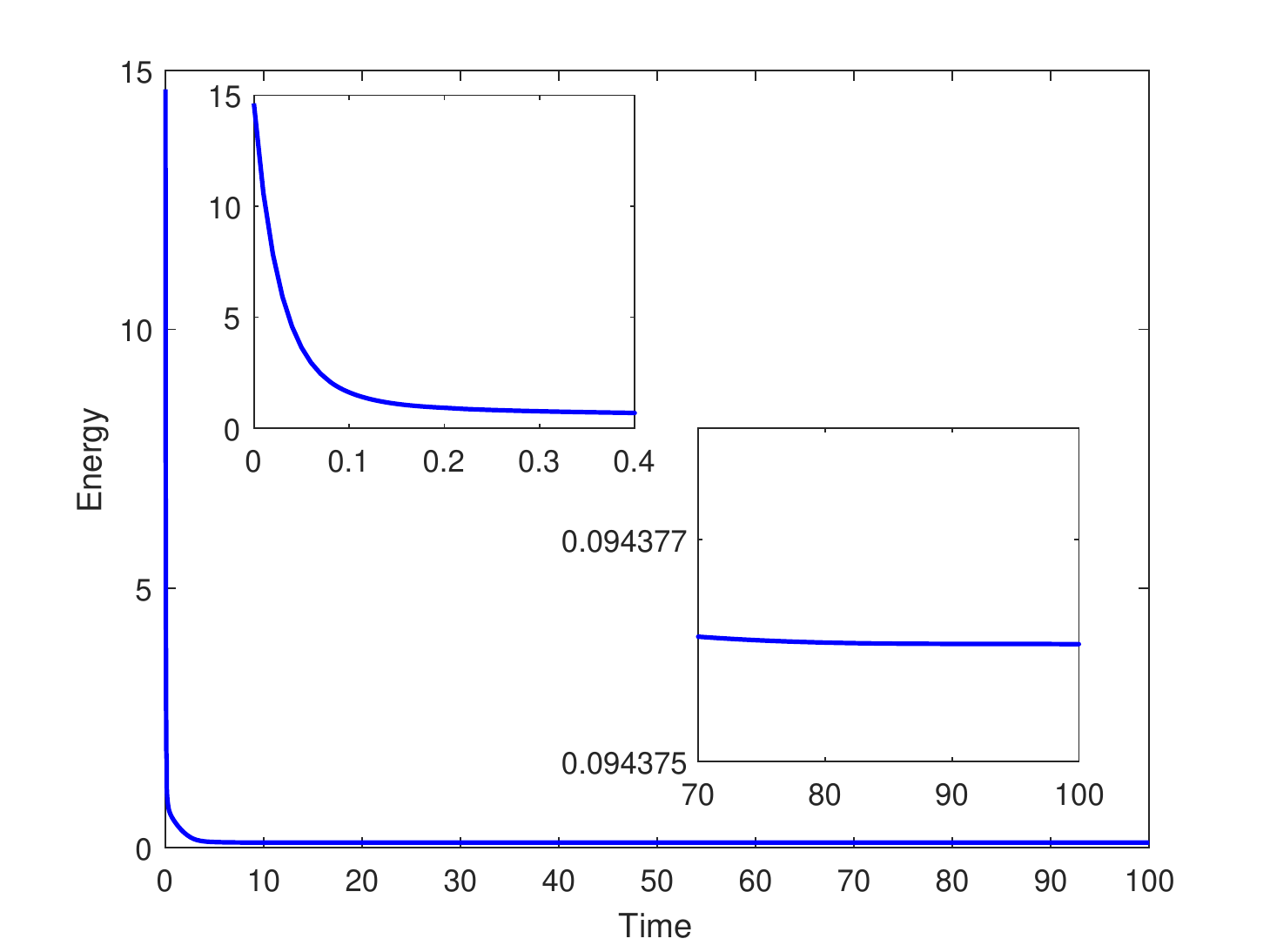}\hspace{-0.5cm}
\includegraphics[width=0.54\textwidth]{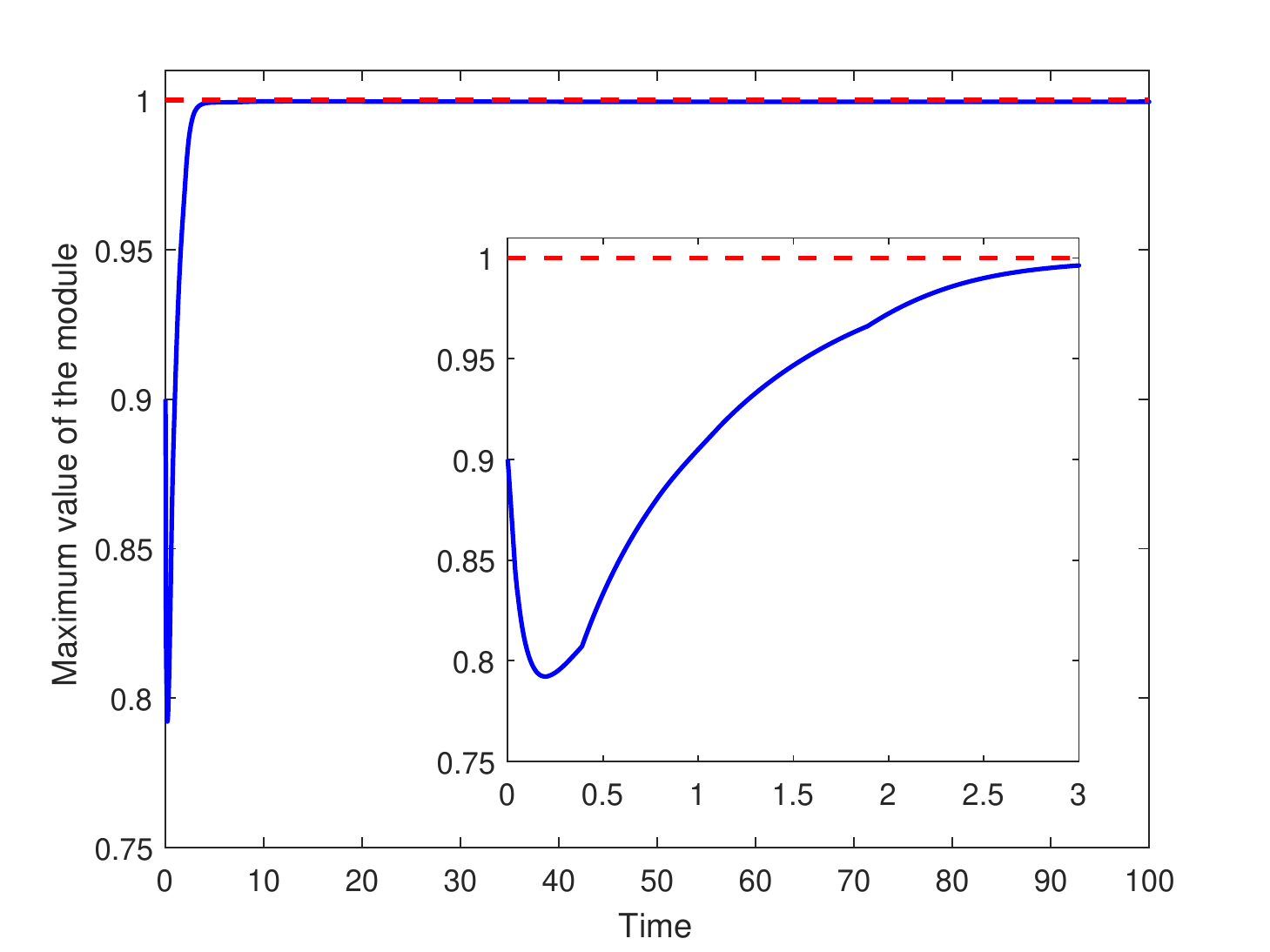}}
\vspace{-0.2cm}
\caption{Evolutions of the energy of the simulated vector field $\bu$ (left) and the maximum value of $|\bu|$
over the interior nodes  (right) in Example \ref{eg_num2}.}
\label{fig_maxmod}
\end{figure}

\section{Concluding remarks}
\label{sect_conclusion}

In this work, we establish an abstract mathematical framework for studying
MBPs of a class of semilinear parabolic equations subject to a variety
of boundary conditions, as well as unconditionally MBP-preserving temporal approximation schemes, ETD1 and ETDRK2,
based on the exponential integrator with a suitably chosen generator of a contraction semigroup.
The motivation is to reveal essential characteristics
of the semilinear parabolic equations having the MBPs
and to analyze the fundamental conditions under which the ETD schemes
unconditionally preserve the MBPs of the underlying problems.
We conclude that, to ensure the MBP of the model equation \eqref{model_eq} and the ETD schemes,
a crucial condition is the dissipation property of the linear operator stated in Assumption \ref{assump_L}
and the sign change property of  the nonlinear term stated in Assumption \ref{assump_f}.
The main results presented in this paper significantly generalize those in \cite{DuJuLiQi19} in many aspects.
The existence of MBP-preserving ETDRK schemes of even higher-order defined in other approaches
is an open question and remains as one of our future works.

Extensions of the framework to the cases of complex-valued,
real vector-valued and matrix-valued equations in the space-continuous setting
are also carried out by considering the nonlinear operators taking the double-well-like forms
and either Dirichlet, homogeneous Neumann, or periodic boundary condition,
where the MBPs are proposed with respect to the vector and matrix $2$-norms, respectively.
We also note that whether the matrix-valued equation \eqref{matrix_AllenCahn} with nonhomogeneous Dirichlet boundary condition
has the MBP with respect to the matrix $2$-norm is still an open question and needs to be further studied.
The main difficulty resides in that
the matrix $2$-norm of the matrix case behaves essentially differently from
the absolute value in the scalar case or the magnitude in the vector case.
Moreover, it is worth noting that
the matrix-valued equation \eqref{matrix_AllenCahn} is derived in \cite{OsWa20} as
the $L^2$ gradient flow under some energy functional
with the nonlinear part as a penalty term for matrix-valued fields which do not take orthogonal matrix values.
The penalty term takes the form $|I_m-U^TU|_F^2$,
measuring the difference between a matrix $U$ and an orthogonal matrix
in the sense of Frobenius norm ($F$-norm),
so it seems more natural to consider the MBP with respect to the $F$-norm.
Actually,  it is not hard to prove that,
if the $F$-norm of the initial data is not greater than $\sqrt{m}$,
then nor is the solution of the equation \eqref{matrix_AllenCahn} at any time.
However, unlike the case of $2$-norm,
the boundedness of the $F$-norm of the solution
is not sufficient to bound the nonlinear part in the splitting form;
thus, whether this $F$-norm-based MBP
could be preserved by the ETD schemes is still an open question.

It is also worth mentioning that, apart from the ETD method,
the integrating factor (IF) method is also a widely-used temporal integration method based on exponential integrators.
While the ETD method only approximates the nonlinear terms as mentioned above,
the IF method often applies numerical quadrature rules to the whole integrand.
The IF method was introduced by Lawson \cite{Lawson67} to solve ODE systems with large Lipschitz constants
and applied to some problems with stiff linear part and nonstiff nonlinear part,
such as the reaction-diffusion problems \cite{JuLiLe14,LiNi10,NiWaZhLi08}
and the advection-diffusion problems \cite{IsGrGo18,LuZh17,ZhOvLiZhNi11}.
Due to highly different decaying rates of the exponential integrator components,
the ETD schemes are usually more accurate than the IF ones for highly stiff systems.
It is also the case that, if the nonlinear function is given by a constant (e.g., $f[u]\equiv c$),
the ETD schemes can produce the exact solution to \eqref{intro_model}
while only approximate solutions are obtained by the IF schemes.
One may ask whether, similar to ETD schemes, the IF schemes could preserve the MBPs.
Related to this question,
the authors of \cite{IsGrGo18} focused on the property of strong stability-preserving (SSP) \cite{GoShTa01}
for the IF Runge--Kutta (IFRK) schemes.
SSP is a stronger stability than the MBP considered here.
In fact, if we weaken the assumptions given in \cite{IsGrGo18} appropriately,
we could obtain MBP-preserving IFRK schemes {under some suitable constraint on the time step size}.
Recall that a scheme is SSP means that if the linear operator $\hL$ satisfies
\begin{equation}
\label{SSP_linear}
\|\e^{\dt\hL}\|\le1,\quad \forall\,\dt>0,
\end{equation}
which is exactly Lemma \ref{lem_L_semigroup}-(ii),
and there exists some $\dt_0>0$ such that the nonlinear mapping $f$ satisfies
\begin{equation}
\label{SSP_nonlinear}
\|w+\dt f[w]\|\le\|w\|,\quad\forall\,w\in \hX,\ \dt\le\dt_0,
\end{equation}
then the solution always satisfies $\|v^{n+1}\|\le\|v^n\|$ for any $\dt\le\dt_0$.
Instead of \eqref{SSP_nonlinear}, if one makes an assumption on $f$ as follows:
\begin{equation}
\label{SSP_nonlinear2}
\|w+\dt f[w]\|\le\beta,\quad\forall\,\text{$w\in \hX$ with $\|w\|\le\beta$ and $\dt\le\dt_0$},
\end{equation}
which is actually equivalent to Lemma \ref{lem_nonlinear}-(i) with $\kp=1/\dt$,
With slight modifications to the proof for the IFRK schemes in \cite{IsGrGo18},
one can also conclude that
the IFRK schemes, which are SSP under the conditions \eqref{SSP_linear} and \eqref{SSP_nonlinear},
are also MBP-preserving under the assumptions \eqref{SSP_linear} and \eqref{SSP_nonlinear2} (or say, Assumptions \ref{assump_L} and \ref{assump_f}).
Note that such property of MBP-preserving is only conditional
in the sense that some constraint on the time step size is necessary.
An open question is
whether the MBP could be preserved unconditionally by the IF schemes as the ETD schemes
if an appropriate stabilizer is introduced and this also remains as one of our future works.

\section*{Acknowledgements}
The authors would like to thank the anonymous reviewers for
their constructive comments and suggestions, which have helped us greatly improve this work.
The authors are also grateful to Dr.~Xiaochuan Tian of University of Texas at Austin
and Dr.~Zhi Zhou of The Hong Kong Polytechnic University
for many valuable discussions.


\begin{thebibliography}{0}

\bibitem{AlCa79}
S.~M. Allen and J.~W. Cahn,
{\it A microscopic theory for antiphase boundary motion and its application to antiphase domain coarsening},
Acta Metall., 27 (1979), pp.~1085--1095.

\bibitem{Amann78}
H. Amann,
{\it Invariant sets and existence theorems for semilinear parabolic and elliptic systems},
J. Math. Anal. Appl., 65 (1978), pp.~432--467.

\bibitem{BaHi15}
S. Badia and A. Hierro,
{\it On discrete maximum principles for discontinuous Galerkin methods},
Comput. Methods Appl. Mech. Engrg., 286 (2015), pp.~107--122.

\bibitem{BeKeVo98}
G. Beylkin, J.~M. Keiser, and L. Vozovoi,
{\it A new class of time discretization schemes for the solution of nonlinear PDEs},
J. Comput. Phys., 147 (1998), pp.~362--387.

\bibitem{BrHu64}
J.~H. Bramble and B.~E. Hubbard,
{\it New monotone type approximations for elliptic problems},
Math. Comp., 18 (1964), pp.~349--367.

\bibitem{BrKoKr08}
J.~H. Brandts, S. Korotov, and M. K{\v{r}}{\'i}{\v{z}}ek,
{\it The discrete maximum principle for linear simplicial finite element approximations of a reaction-diffusion problem},
Linear Algebra Appl., 429 (2008), pp.~2344--2357.

\bibitem{BuEr04}
E. Burman and A. Ern,
{\it Discrete maximum principle for Galerkin approximations of the Laplace operator on arbitrary meshes},
C. R. Math. Acad. Sci. Paris, Ser. I 338 (2004), pp.~641--646.

\bibitem{Chasseigne07}
E. Chasseigne,
{\it The Dirichlet problem for some nonlocal diffusion equations},
Differential Integral Equations, 20 (2007), pp.~1389--1404.

\bibitem{Ciarlet70}
P.~G. Ciarlet,
{\it Discrete maximum principle for finite-difference operators},
Aequations Math., 4 (1970), pp.~338--352.

\bibitem{CiRa73}
P.~G. Ciarlet and P.~A. Raviart,
{\it Maximum principle and uniform convergence for the finite element method},
Comput. Methods Appl. Mech. Engrg., 2 (1973), pp.~17--31.

\bibitem{CoMa02}
S.~M. Cox and P.~C. Matthews,
{\it Exponential time differencing for stiff systems},
J. Comput. Phys., 176 (2002), pp.~430--455.

\bibitem{Du94}
Q. Du,
{\it Global existence and uniqueness of solutions of the time-dependent Ginzburg--Landau model for superconductivity},
Appl. Anal., 53 (1994), pp.~1--17.

\bibitem{Du98}
Q. Du,
{\it Discrete gauge invariant approximations of a time-dependent Ginzburg--Landau model of superconductivity},
Math. Comp., 67 (1998), pp.~965--986.

\bibitem{Du05}
Q. Du,
{\it Numerical approximations of the Ginzburg--Landau models for superconductivity},
J. Math. Phys., 46 (2005), Art.~095109.

\bibitem{Du19}
Q. Du,
{\it Nonlocal Modeling, Analysis, and Computation},
Vol.~94 of CBMS-NSF regional conference series in Applied Mathematics, SIAM, 2019.

\bibitem{DuFe19}
Q. Du and X.~B. Feng,
{\it The phase field method for geometric moving interfaces and their numerical approximations},
in {Geometric Partial Differential Equations - Part I, Handbook of Numerical Analysis}, Vol.~21, Elsevier, 2020.

\bibitem{DuGuLeZh12}
Q. Du, M. Gunzburger, R.~B. Lehoucq, and K. Zhou,
{\it Analysis and approximation of nonlocal diffusion problems with volume constraints},
SIAM Rev., 54 (2012), pp.~667--696.

\bibitem{DuGuLeZh13}
Q. Du, M. Gunzburger, R.~B. Lehoucq, and K. Zhou,
{\it A nonlocal vector calculus, nonlocal volume-constrained problems, and nonlocal balance laws},
Math. Models Methods Appl. Sci., 23 (2013), pp.~493--540.

\bibitem{DuGuPe92}
Q. Du, M. Gunzburger, and J. Peterson,
{\it Analysis and approximation of Ginzburg--Landau models for superconductivity},
SIAM Rev., 34 (1992), pp.~54--81.

\bibitem{DuJuLiQi19}
Q. Du, L. Ju, X. Li, and Z.~H. Qiao,
{\it Maximum principle preserving exponential time differencing schemes for the nonlocal Allen--Cahn equation},
SIAM J. Numer. Anal., 57 (2019), pp.~875--898.

\bibitem{DuNiWu98}
Q. Du, R. Nicolaides, and X. Wu,
{\it Analysis and convergence of a covolume approximation of the Ginzburg--Landau model of superconductivity},
SIAM J. Numer. Anal., 35 (1998), pp.~1049--1072.

\bibitem{DuTaTiYa19}
Q. Du, Y.~Z. Tao, X.~C. Tian, and J. Yang,
{\it Asymptotically compatible discretization of multidimensional nonlocal diffusion models
and approximation of nonlocal Green's functions},
IMA J. Numer. Anal., 39 (2019), pp.~607--625.

\bibitem{DuYi19}
Q. Du and X.~B. Yin,
{\it A conforming DG method for linear nonlocal models with integrable kernels},
J. Sci. Comput., 80 (2019), pp.~1913--1935.

\bibitem{DuZh05}
Q. Du and W.-X. Zhu,
{\it Analysis and applications of the exponential time differencing schemes and their contour integration modifications},
BIT Numer. Math., 45 (2005), pp.~307--328.

\bibitem{DuWyZh18}
S.~W. Duo, H.~W. van Wyk, and Y.~Z. Zhang,
{\it A novel and accurate finite difference method for the fractional Laplacian and the fractional Poisson problem},
J. Comput. Phys., 355 (2018), pp.~233--252.

\bibitem{DuZh19}
S.~W. Duo and Y.~Z. Zhang,
{\it Accurate numerical methods for two and three dimensional integral fractional Laplacian with applications},
Comput. Methods Appl. Mech. Engrg., 355 (2019), pp.~639--662.

\bibitem{EnNa00}
K.-J. Engel and R. Nagel,
{\it One-Parameter Semigroups for Linear Evolution Equations},
Graduate Texts in Mathematics, Vol.~194, Springer-Verlag, New York, 2000.

\bibitem{Evans00}
L.~C. Evans,
{\it Partial Differential Equations},
American Mathematical Society, Providence, Rhode Island, 2000.

\bibitem{EvSoSo92}
L.~C. Evans, H.~M. Soner, and P.~E. Souganidis,
{\it Phase transitions and generalized motion by mean curvature},
Commun. Pure Appl. Math., 45 (1992), pp.~1097--1123.

\bibitem{FaHo06}
I. Farag\'o and R. Horv\'ath,
{\it Discrete maximum principle and adequate discretizations of linear parabolic problems},
SIAM J. Sci. Comput., 28 (2006), pp.~2313--2336.

\bibitem{FaKaKo12}
I. Farag\'o, J. Kar\'atson, and S. Korotov,
{\it Discrete maximum principles for nonlinear parabolic PDE systems},
IMA J. Numer. Anal., 32 (2012), pp.~1541--1573.

\bibitem{FaKoSz13}
I. Farag\'o, S. Korotov, and T. Szab\'o,
{\it On continuous and discrete maximum principles for elliptic problems with the third boundary condition},
Appl. Math. Comput., 219 (2013), pp.~7215--7224.

\bibitem{FaKoSz15}
I. Farag\'o, S. Korotov, and T. Szab\'o,
{\it On continuous and discrete maximum-minimum principles for reaction-diffusion problems with the Neumann boundary condition},
Applications of mathematics, pp.~34--44, Czech. Acad. Sci., Prague, 2015.

\bibitem{FePr03}
X.~B. Feng and A. Prohl,
{\it Numerical analysis of the Allen--Cahn equation and approximation for mean curvature flows},
Numer. Math., 94 (2003), pp.~33--65.

\bibitem{FePr04}
X.~B. Feng and A. Prohl,
{\it Error analysis of a mixed finite element method for the Cahn--Hilliard equation},
Numer. Math., 99 (2004), pp.~47--84.

\bibitem{FeRo16}
X. Fern\'andez-Real and X. Ros-Oton,
{\it Boundary regularity for the fractional heat equation},
Rev. R. Acad. Cienc. Exactas F\'is. Nat. Ser. A Mat., 110 (2016), pp.~49--64.


\bibitem{FrMaSgVe19}
M. Frittelli, A. Madzvamuse, I. Sgura, and C. Venkataraman,
{\it Preserving invariance properties of reaction-diffusion systems on stationary surfaces},
IMA J. Numer. Anal., 39 (2019), pp.~235--270.


\bibitem{GaSa92}
E. Gallopoulous and Y. Saad,
{\it Efficient solution of parabolic equations by Krylov approximation methods},
SIAM J. Sci. Comput., 13 (1992), pp.~1236--1264.

\bibitem{GaPu16}
S. Gaudreault and J.~A. Pudykiewicz,
{\it An efficient exponential time integration method for the numerical solution of the shallow water equations on the sphere},
J. Comput. Phys., 322 (2016), pp.~827--848.

\bibitem{GaRaTo18}
S. Gaudreault, G. Rainwater, and M. Tokman,
{\it KIOPS: A fast adaptive Krylov subspace solver for exponential integrators},
J. Comput. Phys., 372 (2018), pp.~236--255.

\bibitem{GiLa50}
V. Ginzburg and L. Landau,
{\it On the theory of superconductivity},
Zh. Eksperim. Teor. Fiz., 20 (1950), pp.~1064--1082.

\bibitem{GoShTa01}
S. Gottlieb, C.-W. Shu, and E. Tadmor,
{\it Strong stability-preserving high-order time discretization methods},
SIAM Rev., 43 (2001), pp.~89--112.

\bibitem{GuPoSa20}
J.-L. Guermond, B. Popov, and L. Saavedra,
{\it Second-order invariant domain preserving ALE approximation of hyperbolic systems},
J. Comput. Phys., 401 (2020), Art.~108927.

\bibitem{GuPoTo19}
J.-L. Guermond, B. Popov, and I. Tomas,
{\it Invariant domain preserving discretization-independent schemes and convex limiting for hyperbolic systems},
Comput. Methods Appl. Mech. Engrg., 347 (2019), pp.~143--175.

\bibitem{Hig08}
N.~J. Higham,
{\it Functions of Matrices: Theory and Computation},
SIAM, Philadelphia, PA, 2008.

\bibitem{HoLu97}
M. Hochbruck and C. Lubich,
{\it On Krylov subspace approximations to the matrix exponential operator},
SIAM J. Numer. Anal., 34 (1997), pp.~1911--1925.

\bibitem{HoLuSe98}
M. Hochbruck, C. Lubich, and H. Selhofer,
{\it Exponential integrators for large systems of differential equations},
SIAM J. Sci. Comput., 19 (1998), pp.~1552--1574.

\bibitem{HoOs05}
M. Hochbruck and A. Ostermann,
{\it Explicit exponential Runge--Kutta methods for semilinear parabolic problems},
SIAM J. Numer. Anal., 43 (2005), pp.~1069--1090.

\bibitem{HoOs10}
M. Hochbruck and A. Ostermann,
{\it Exponential integrators},
Acta Numer., 19 (2010), pp.~209--286.

\bibitem{HoJo91}
R.~A. Horn and C.~R. Johnson,
{\it Topics in Matrix Analysis},
Cambridge University Press, Cambridge, 1991.

\bibitem{HoTaYa17}
T.~L. Hou, T. Tang, and J. Yang,
{\it Numerical analysis of fully discretized Crank--Nicolson scheme for fractional-in-space Allen--Cahn equations},
J. Sci. Comput., 72 (2017), pp.~1214--1231.

\bibitem{HuOb14}
Y.~H. Huang and A. Oberman,
{\it Numerical methods for the fractional Laplacian: a finite difference-quadrature approach},
SIAM J. Numer. Anal., 52 (2014), pp.~3056--3084.

\bibitem{IsGrGo18}
L. Isherwood, Z.~J. Grant, and S. Gottlieb,
{\it Strong stability preserving integrating factor Runge--Kutta methods},
SIAM J. Numer. Anal., 56 (2018), pp.~3276--3307.

\bibitem{JiLi18}
Y. Jiang and H.~L. Liu,
{\it Invariant-region-preserving DG methods for multi-dimensional hyperbolic conservation law systems,
with an application to compressible Euler equations},
J. Comput. Phys., 373 (2018), pp.~385--409.

\bibitem{JuLiQiZh18}
L. Ju, X. Li, Z.~H. Qiao, and H. Zhang,
{\it Energy stability and error estimates of exponential time differencing schemes
for the epitaxial growth model without slope selection},
Math. Comp., 87 (2018), pp.~1859--1885.

\bibitem{JuLiLe14}
L. Ju, X.~F. Liu, And W. Leng,
{\it Compact implicit integration factor methods for a family of semilinear fourth-order parabolic equations},
Discrete Contin. Dyn. Syst. Ser. B, 19 (2014), pp.~1667--1687.

\bibitem{JuZhDu15}
L. Ju, J. Zhang, and Q. Du,
{\it Fast and accurate algorithms for simulating coarsening dynamics of Cahn--Hilliard equations},
Comput. Mater. Sci., 108 (2015), pp.~272--282.

\bibitem{JuZhZhDu15}
L. Ju, J. Zhang, L.~Y. Zhu, and Q. Du,
{\it Fast explicit integration factor methods for semilinear parabolic equations},
J. Sci. Comput., 62 (2015), pp.~431--455.

\bibitem{Karafiat91}
A. Karafiat,
{\it Discrete maximum principle in parabolic boundary-value problems},
Ann. Polon. Math., 53 (1991), pp.~253--265.

\bibitem{KaKo15}
J. Kar\`atson and S. Korotov,
{\it Some discrete maximum principles arising for nonlinear elliptic finite element problems},
Comput. Math. Appl., 70 (2015), pp.~2732--2741.

\bibitem{KaKoKr07}
J. Kar\`atson, S. Korotov, and M. K{\v{r}}{\'i}{\v{z}}ek,
{\it On discrete maximum principles for nonlinear elliptic problems},
Math. Comput. Simulation, 76 (2007), pp.~99--108.

\bibitem{KaPo80}
D.~R. Kassoy and J. Poland,
{\it The thermal explosion confined by a constant temperature boundary},
SIAM J. Appl. Math., 39 (1980), pp.~412--430.

\bibitem{Kuiper80}
H.~J. Kuiper,
{\it Invariant sets for nonlinear elliptic and parabolic systems},
SIAM J. Math. Anal., 11 (1980), pp.~1075--1103.

\bibitem{Lawson67}
J.~D. Lawson,
{\it Generalized Runge--Kutta processes for stable systems with large Lipschitz constants},
SIAM J. Numer. Anal., 4 (1967), pp.~372--380.

\bibitem{LiQiTa16}
D. Li, Z.~H. Qiao, and T. Tang,
{\it Characterizing the stabilization size for semi-implicit Fourier-spectral method to phase field equations},
SIAM J. Numer. Anal., 54 (2016), pp.~1653--1681.

\bibitem{LiChWu00}
R.~H. Li, Z.~Y. Chen, and W. Wu,
{\it Generalized Difference Methods for Differential Equations: Numerical Analysis of Finite Volume Methods},
Marcel Dekker, Inc., 2000.

\bibitem{LiIt01}
Z.~L. Li and K. Ito,
{\it Maximum principle preserving schemes for interface problems with discontinuous coefficients},
SIAM J. Sci. Comput., 23 (2001), pp.~339--361.

\bibitem{LiNi10}
X.~F. Liu and Q. Nie,
{\it Compact integration factor methods for complex domains and adaptive mesh refinement},
J. Comput. Phys., 229 (2010), pp.~5692--5706.

\bibitem{LiYu14}
H.~L. Liu and H. Yu,
{\it Maximum-principle-satisfying third order discontinuous Galerkin schemes for Fokker--Planck equations},
SIAM J. Sci. Comput., 36 (2014), pp.~A2296--A2325.

\bibitem{LuZh17}
D. Lu and Y.~T. Zhang,
{\it Computational complexity study on Krylov integration factor WENO method for high spatial dimension convection-diffusion problems},
J. Sci. Comput., 73 (2017), pp.~980--1027.

\bibitem{MacRae74}
E.~C. MacRae,
{\it Matrix derivatives with an application to an adaptive linear decision problem},
Ann. Statist., 2 (1974), pp.~337--346.

\bibitem{MaMo84}
C. Mastroserio and M. Montrone,
{\it Invariant regions and asymptotic behaviour for the numerical solution
of reaction-diffusion systems by a class of alternating direction methods},
Calcolo, 21 (1984), pp.~269--279.

\bibitem{MiHo12}
M.~E. Mincsovics and T.~L. Horv\'ath,
{\it On the differences of the discrete weak and strong maximum principles for elliptic operators},
Large-scale scientific computing, pp.~614--621,
Lecture Notes in Comput. Sci., 7116, Springer, Heidelberg, 2012.

\bibitem{MoVan78}
C. Moler and C. Van Loan,
{\it Nineteen dubious ways to compute the exponential of a matrix},
SIAM Rev., 20 (1978), pp.~801--836.

\bibitem{MoVan03}
C. Moler and C. Van Loan,
{\it Nineteen dubious ways to compute the exponential of a matrix, twenty-five years later},
SIAM Rev., 45 (2003), pp.~3--49.

\bibitem{NiWaZhLi08}
Q. Nie, F.~Y.~M. Wan, Y.~T. Zhang, and X.~F. Liu,
{\it Compact integration factor methods in high spatial dimensions},
J. Comput. Phys., 227 (2008), pp.~5238--5255.

\bibitem{NiWr12}
J. Niesen and W.~M. Wright,
{\it Algorithm 919: A Krylov subspace algorithm for evaluating
the $\phi$-functions appearing in exponential integrators},
ACM Trans. Math. Soft., 38 (2008), Art.~22.

\bibitem{NoZh18}
R.~H. Nochetto and W.~J. Zhang,
{\it Discrete ABP estimate and convergence rates for linear elliptic equations in non-divergence form},
Found. Comput. Math., 18 (2018), pp.~537--593.

\bibitem{OsWa20}
B. Osting and D. Wang,
{\it A diffusion generated method for orthogonal matrix-valued fields},
Math. Comp., 89 (2020), pp.~515--550.

\bibitem{EO92}
E.-M. Ouhabaz,
{\it $L^\infty$-contractivity of semigroups generated by sectorial forms},
J. London Math. Soc., 46 (1992), pp.~529--542.

\bibitem{Pao92}
C.~V. Pao,
{\it Nonlinear Parabolic and Elliptic Equations},
Plenum Press, New York, 1992.

\bibitem{PeRo76}
D.-Y. Peng and D.~B. Robinson,
{\it A new two-constant equation of state},
Ind. Eng. Chem. Fundamen., 15 (1976), pp.~59--64.

\bibitem{Price68}
H.~S. Price,
{\it Monotone and oscillation matrices applied to finite difference approximations},
Math. Comp., 22 (1968), pp.~489--516.

\bibitem{QiSu14}
Z.~H. Qiao and S.~Y. Sun,
{\it Two-phase fluid simulation using a diffuse interface model with Peng--Robinson equation of state},
SIAM J. Sci. Comput., 36 (2014), pp.~B708--B728.

\bibitem{ReWa78}
R. Redheffer and W. Walter,
{\it Invariant sets for systems of partial differential equations. I. Parabolic equations},
Arch. Rational Mech. Anal., 67 (1978), pp.~41--52.

\bibitem{Rudin76}
W. Rudin,
{\it Principles of Mathematical Analysis},
Third Edition, McGraw-Hill Book Co., New York, 1976.

\bibitem{SaKiMa93}
S.~G. Samko, A.~A. Kilbas, and O.~I. Marichev,
{\it Fractional Integrals and Derivatives},
Gordon and Breach Science Publishers, Yverdon, 1993.

\bibitem{ShTaYa16}
J. Shen, T. Tang, and J. Yang,
{\it On the maximum principle preserving schemes for the generalized Allen--Cahn equation},
Commun. Math. Sci., 14 (2016), pp.~1517--1534.

\bibitem{ShYa10}
J. Shen and X.~F. Yang,
{\it Numerical approximations of Allen--Cahn and Cahn--Hilliard equations},
Discrete Contin. Dyn. Syst., 28 (2010), pp.~1669--1691.

\bibitem{Sidje98}
R.~B. Sidje,
{\it Expokit: Software package for computing matrix exponentials},
ACM Trans. Math. Soft., 24 (1998), pp.~130--156.

\bibitem{Smoller94}
J. Smoller,
{\it Shock Waves and Reaction-Diffusion Equations},
Fundamental Principles of Mathematical Sciences, Vol.~258, Second Edition, Springer-Verlag, New York, 1994.

\bibitem{StVo15}
P. Stehl\'ik and J. Volek,
{\it Maximum principles for discrete and semidiscrete reaction-diffusion equation},
Discrete Dyn. Nat. Soc., 2015, Art.~791304.

\bibitem{TaYa16}
T. Tang and J. Yang,
{\it Implicit-explicit scheme for the Allen--Cahn equation preserves the maximum principle},
J. Comput. Math., 34 (2016), pp.~471--481.

\bibitem{TiJuDu15}
H. Tian, L. Ju, and Q. Du,
{\it Nonlocal convection-diffusion problems and finite element approximations},
Comput. Methods Appl. Mech. Engrg., 289 (2015), pp.~60--78.

\bibitem{TiDu13}
X.~C. Tian and Q. Du,
{\it Analysis and comparison of different approximations to nonlocal diffusion and linear peridynamic equations},
SIAM J. Numer. Anal., 51 (2013), pp.~3458--3482.

\bibitem{TiDuGu16}
X.~C. Tian, Q. Du, and M. Gunzburger,
{\it Asymptotically compatible schemes for the approximation of fractional Laplacian
and related nonlocal diffusion problems on bounded domains},
Adv. Comput. Math., 42 (2016), pp.~1363--1380.

\bibitem{Varga66}
R.~S. Varga,
{\it On a discrete maximum principle},
SIAM J. Numer. Anal., 3 (1966), pp.~355--359.

\bibitem{WaWiLo09}
C. Wang, S.~M. Wise, and J.~S. Lowengrub,
{\it An energy-stable and convergent finite-difference scheme for the phase field crystal equation},
SIAM J. Numer. Anal., 47 (2009), pp.~2269--2288.

\bibitem{WaJuDu16}
X.~Q. Wang, L. Ju, and Q. Du,
{\it Efficient and stable exponential time differencing Runge--Kutta methods for phase field elastic bending energy models},
J. Comput. Phys., 316 (2016), pp.~21--38.

\bibitem{YaDuZh18}
J. Yang, Q. Du, and W. Zhang,
{\it Uniform $L^p$-bound of the Allen--Cahn equation and its numerical discretization},
Int. J. Numer. Anal. Mod., 18 (2018), pp.~213--227.

\bibitem{YaXiQiXu16}
P. Yang, T. Xiong, J.~M. Qiu, and Z.~F. Xu,
{\it High order maximum principle preserving finite volume method for convection dominated problems},
J. Sci. Comput., 67 (2016), pp.~795--820.

\bibitem{Yanik87}
E.~G. Yanik,
{\it A discrete maximum principle for collocation methods},
Comput. Math. Appl., 14 (1987), pp.~459--464.

\bibitem{Yanik89}
E.~G. Yanik,
{\it Sufficient conditions for a discrete maximum principle for high order collocation methods},
Comput. Math. Appl., 17 (1989), pp.~1431--1434.

\bibitem{YuYu18}
G.~W. Yuan and Y.~L. Yu,
{\it Existence of solution of a finite volume scheme preserving maximum principle for diffusion equations},
Numer. Meth. Part. Diff. Eq., 34 (2018), pp.~80--96.

\bibitem{ZhZhWaJuDu16}
J. Zhang, C.~B. Zhou, Y.~G. Wang, L. Ju, Q. Du, X.~B. Chi, D.~S. Xu, D.~X. Chen, Y. Liu, and Z. Liu,
{\it Extreme-scale phase field simulations of coarsening dynamics on the Sunway Taihulight supercomputer},
in Proceedings of the International Conference for High Performance Computing, Networking, Storage and Analysis (SC'16),
2016, Art.~4.

\bibitem{ZhOvLiZhNi11}
S. Zhao, J. Ovadia, X.~F. Liu, Y.~T. Zhang, and Q, Nie,
{\it Operator splitting implicit integration factor methods for stiff reaction-diffusion-advection systems},
J. Comput. Phys., 230 (2011), pp.~5996--6009.

\bibitem{ZhJuZh16}
L.~Y. Zhu, L. Ju, and W.~D. Zhao,
{\it Fast high-order compact exponential time differencing Runge--Kutta methods for second-order semilinear parabolic equations},
J. Sci. Comput., 67 (2016), pp.~1043--1065.

\end{thebibliography}
\end{document}